\documentclass[letterpaper,11pt]{article}
\usepackage{tabularx} % extra features for tabular environment
\usepackage{amsmath}  % improve math presentation
\usepackage[shortlabels]{enumitem}
\usepackage[margin=1in,letterpaper]{geometry} % decreases margins
\usepackage[final]{hyperref} % adds hyper links inside the generated pdf file
\usepackage{dsfont}
\usepackage{enumitem}
\usepackage{bm}

\usepackage{dingbat}
\usepackage{lmodern}
\usepackage{comment}
\usepackage{xfrac}
\usepackage[T1]{fontenc}
\usepackage{amsmath}
\usepackage{subcaption}
\usepackage{graphicx}
\usepackage{chngcntr}
\usepackage{amsthm}
\usepackage[export]{adjustbox}
\usepackage[english]{babel}
\usepackage{amsmath, varwidth}
\usepackage{bbm}
\usepackage{amsfonts}

\usepackage[refpage,noprefix,intoc]{nomencl} %for list of notation.
\usepackage[dvipsnames]{xcolor}
\usepackage[utf8]{inputenc}
\usepackage{csquotes}
\usepackage[normalem]{ulem}
\usepackage{color,xcolor}
\definecolor{cjl}{rgb}{0.5,0.3,0.9}
\definecolor{darkgreen}{rgb}{0,0.5,0}

\usepackage{etoolbox}
\renewcommand\nomgroup[1]{%
  \item[\bfseries
  \ifstrequal{#1}{A}{Construction of the BBM}{%
  \ifstrequal{#1}{B}{Relation with the F-KPP}{%
  \ifstrequal{#1}{C}{Approximation of convex sets}{%
  \ifstrequal{#1}{D}{Prerequisites}{%
  \ifstrequal{#1}{E}{The shape theorem}{%
  \ifstrequal{#1}{F}{Appendix}
  }}}}}%
]}

\newlist{assumptionlistA}{enumerate}{1} % Defines a new list named 'assumptionlist'
\setlist[assumptionlistA, 1]{
    label=\textbf{(A\arabic*)}, % How the assumption looks in the list: (A1)
    ref=(A\arabic*),             % How the reference looks: A1
    font=\normalfont            % Optional: Ensures the text isn't bold or italic
}

\newlist{assumptionlistB}{enumerate}{1} % Defines a new list named 'assumptionlist'
\setlist[assumptionlistB, 1]{
    label=\textbf{(B\arabic*)}, % How the assumption looks in the list: (A1)
    ref=(B\arabic*),             % How the reference looks: A1
    font=\normalfont            % Optional: Ensures the text isn't bold or italic
}

\newlist{assumptionlistC}{enumerate}{1} % Defines a new list named 'assumptionlist'
\setlist[assumptionlistC, 1]{
    label=\textbf{(C\arabic*)}, % How the assumption looks in the list: (A1)
    ref=(C\arabic*),             % How the reference looks: A1
    font=\normalfont            % Optional: Ensures the text isn't bold or italic
}

\renewcommand*{\pagedeclaration}[1]{\unskip\dotfill}	
\makenomenclature										
\newcommand{\vb}{\bar{\mathbf{v}}}

\newcommand{\csf}{\mathfrak{c}^*}
\newcommand{\css}{\Tilde{c}}
\newtheorem{theorem}{Theorem}[section]
\newtheorem{lemma}[theorem]{Lemma}
\newtheorem{corollary}[theorem]{Corollary}

\newtheorem{proposition}[theorem]{Proposition}

\theoremstyle{remark}
\newtheorem{remark}[theorem]{Remark}



\newcommand{\TT}{\mathbb{T}}
\newcommand{\R}{\mathbb{R}}
\newcommand{\N}{\mathbb{N}}

\newcommand{\bP}{\mathbb{P}}

\newcommand{\bQ}{\mathbb{Q}}
\newcommand{\bE}{\mathbb{E}}
\newcommand{\Z}{\mathbb{Z}}

\newcommand{\cM}{\mathcal{M}}
\newcommand{\cN}{\mathcal{N}}
\newcommand{\cH}{\mathcal{H}}
\newcommand{\cW}{\mathcal{W}}
\newcommand{\cR}{\mathcal{R}}
\newcommand{\cQ}{\mathcal{Q}}

\newcommand{\cF}{\mathcal{F}}
\newcommand{\cs}{c^*}

\newcommand{\gs}{\gamma}

\newcommand{\proba}[3]{\mathbf{P}_{#1}^{#2}[#3]}
\newcommand{\probaa}[3]{\mathbf{P}_{#1}^{#2}\bigl[#3\bigl]}
\newcommand{\probaaa}[3]{\mathbf{P}_{#1}^{#2}\Bigl[#3\Bigl]}

\newcommand{\prob}[3]{\mathbb{P}_{#1}^{#2}[#3]}

\newcommand{\probbb}[3]{\mathbb{P}_{#1}^{#2}\Bigl[#3\Bigl]}

\newcommand{\parentesiss}[1]{\bigl(#1\bigl)}

\newcommand{\expec}[3]{\mathbb{E}_{#1}^{#2}[#3]}
\newcommand{\expecdd}[3]{\mathbb{E}_{#1}^{#2}\bigl[#3\bigl]}
\newcommand{\expecddd}[3]{\mathbb{E}_{#1}^{#2}\Bigl[#3\Bigl]}
\newcommand{\expecdddd}[3]{\mathbb{E}_{#1}^{#2}\biggl[#3\biggl]}

\newcommand{\expected}[3]{\mathbf{E}_{#1}^{#2}[#3]}

\newcommand{\expecteddd}[3]{\mathbf{E}_{#1}^{#2}\Bigl[#3\Bigl]}

\newcommand{\indc}{\mathds{1}}

\numberwithin{equation}{section}

\usepackage[numbers, sort&compress]{natbib}

\begin{document}

\title{A shape theorem for periodic branching diffusion}
\author{
Arturo Arellano Arias
\thanks{Department of Mathematics and Statistics, McGill University.
{\footnotesize \href{mailto:arturo.arellanoarias@mail.mcgill.ca}{arturo.arellanoarias@mail.mcgill.ca}.}
}
}
\date{\today}
\maketitle

\begin{abstract}
    We prove a shape theorem for a multidimensional periodic branching diffusion, which is a branching particle system where particle trajectories evolve according to a diffusion, drift, branching rate and offspring distribution which all depend periodically on space. More precisely, we show that almost surely as $t\to\infty$, the particle cloud at time $t$, normalized by $t$, converges in Hausdorff distance to a deterministic convex set $\cW$. We further establish that the boundary of the asymptotic shape $\cW$ is of class $C^2$, strictly convex, and has positive Gaussian curvature bounded away from zero. This work generalizes the approach of \cite{addarioberry2025shapetheorembbmperiodic} for $g$-BBM and furthers our understanding of the shape $\cW$ in that setting.
\end{abstract}

%\tableofcontents
\newpage

\section{Introduction}\label{sec:Intro}

    In the recent paper \cite{addarioberry2025shapetheorembbmperiodic}, Addario-Berry, Lin and the author proved a shape theorem for $g$-BBM, a multidimensional dyadic branching Brownian motion in a periodic environment, where particles branch heterogeneously in space according to a $\Z^d$-periodic branching rate $g:\R^d\to [0,\infty)$. They show that there exists a deterministic convex set $\cW$ such that, after rescaling the particle cloud by $t$, the $g$-BBM at time $t$ converges asymptotically, almost surely in the Hausdorff distance to $\cW$. Here, we generalize their results to a general periodic setting, allowing particle trajectories to undergo periodic diffusion with periodic drift, and the offspring distribution to depend periodically on space. We further prove more detailed properties of the limit shape $\cW$, including its regularity, which also applies to the setting in \cite{addarioberry2025shapetheorembbmperiodic}.

\subsection{Main results}
    Let both $b:\R^d\to \R^d$ and the matrix valued function $\sigma:\R^d \to M^{d\times d}$ be $\Z^d$-periodic, where $M^{d\times d}$ denotes the space of $d\times d$ matrices over $\R$. In \textit{periodic branching diffusion}, an initial particle named $\emptyset$ is located at $x\in \R^d$ and evolves in space as a weak solution of the SDE
    \begin{equation}\label{eq:Intro:Diffusion}
        Z_t = Z_0 + \int_0^t b(Z_s) \ ds + \int_0^t \sigma(Z_s) \ dB_s, \ \ \ Z_0 = x,
    \end{equation}
    where $(B_t)_{t\geq 0}$ is a $d$-dimensional Brownian motion. 

    The initial particle dies at a random time $\tau_\emptyset$, called the \textit{branching time}, according to an instantaneous $\Z^d$-periodic spatial branching rate $g:\R^d \to [0,+\infty)$. At the exact moment the particle dies, it is replaced by a random number of offspring $N_\emptyset$, where the distribution of $N_\emptyset$ depends periodically on the position $X_{\tau_\emptyset}(\emptyset)$ at which the initial particle dies. More rigorously, we consider a family of probability measures $\cM=(\mu^{(x)})_{x\in \R^d}$ with support on $\N\setminus \{0,1\}$ such that, for every $k\in \N$, $\mu^{(x)}(k)=\mu^{(y)}(k)$ whenever $x-y \in \Z^d$, and we say that the family $\cM$ is periodic. Thus, $N_\emptyset$ satisfies $\mathbb{P}[N_\emptyset = k | X_{\tau_\emptyset}(\emptyset)] = \mu^{(X_{\tau_\emptyset}(\emptyset))}(k)$.
    
    Positioned at the location where its parent died, each of the newborn particles repeats the same process independently. Namely, given $X_{\tau_\emptyset}(\emptyset)$, each of the particles evolve as independent weak solutions of \eqref{eq:Intro:Diffusion} with $Z_0 = X_{\tau_\emptyset}(\emptyset)$, and each of them branches independently with branching rate $g$. At their branching time, they independently produce a random number of offspring according to the periodic family of offspring distributions $\cM$. The process continues in this way indefinitely.

    Each particle is named using a word from the dictionary $\mathcal{U} := \bigcup_{n=0}^\infty \N^n$. If a particle is named $v \in \mathcal{U}$, its offspring are named with the concatenated words $v1,...,vN_\emptyset$. Let $\cN_t$ denote the random subset of $\mathcal{U}$ formed by the names of the particles that are alive at time $t$, observe that $\cN_0=\{\emptyset\}$. For $v\in \cN_t$ and $s\in [0,t]$, let $X_s(v)$ denote the position of the unique ancestor of $v$ at time $t$; and set $\mathcal{X}_t=\{X_t(v): v \in \cN_t\}$. For $x\in \R^{d}$, we let \( \mathbf{P}_x \) denote the law of the periodic branching diffusion starting with a single particle at $x$.
     
    Our first main result generalizes \cite[Theorem 1.1]{addarioberry2025shapetheorembbmperiodic} to this setting. It says that, after normalizing by $t$ and centering by the initial position, there exists a convex set $\cW$ such that with probability $1$, $\mathcal{X}_t$ approaches $\cW$ in the Hausdorff distance as $t$ goes to $+\infty$. 

    We impose the following regularity on the diffusion: $\sigma$ is of class $C^{2,\alpha}(\R^d)$, $\sigma^T \sigma$ is uniformly elliptic, and $b$ is of class $C^{1,\alpha}(\R^d)$. We further require that the branching rate $g$ is of class $C^{0,\alpha}(\R^d)$, and that the first moment function of the family of offspring distributions $\cM$ is of class $C^{0,\alpha}(\R^d)$. Additionally, we impose pure branching events, and a stochastic dominance condition on the family of offspring distributions $\cM$, in order to gain uniform control over the expected number of particles. For a detailed mathematical description of the assumptions we refer to assumptions \textbf{\ref{AssumptionA:bSigmaRegularity}}, \textbf{\ref{AssumptionA:SigmaEllipticity}}, \textbf{\ref{AssumptionB:g}}, \textbf{\ref{AssumptionC:measurability}}, \textbf{\ref{AssumptionC:NoKillingPureBranching}}, \textbf{\ref{AssumptionC:HolderContinuity}} and \textbf{\ref{AssumptionC:StochasticDomination}} in Section \ref{sec:Intro:Assumptions}.
    
    \begin{theorem}\label{thm:Intro:ShapeHauss}
        Under the assumptions above, there exists a convex set $\cW= \cW(b, \sigma, g, \cM)\subset \R^d$ with non-empty interior, such that, for every $\epsilon\in (0,1)$ and $x\in \R^d$,
        \begin{equation*}
            \proba{x}{}{\exists T>0 \ \mathrm{s.t.} \ \forall t\geq T, d_{\mathrm{H}}(\cW, t^{-1}(\mathcal{X}_t-x))<\epsilon} = 1.
        \end{equation*}
    \end{theorem}

    In contrast with \cite[Theorem 1.1]{addarioberry2025shapetheorembbmperiodic}, here the shape $\cW$ does not necessarily contain the origin in its interior. This has to do with the fact that the periodic diffusion \eqref{eq:Intro:Diffusion} has a non-zero drift, which may induce an effective drift in the periodic branching diffusion. In proving Theorem \ref{thm:Intro:ShapeHauss}, we decompose the shape in the form $\cW = \Tilde{\cW} + \vb$, where $0 \in \mathrm{int}(\Tilde{\cW})$ and $\vb\in \R^d$ acts as this effective drift for the periodic branching diffusion.

    Our second main result concerns regularity properties of the shape $\cW$. For the classical binary branching Brownian motion, i.e., when $g\equiv 1$, $\sigma=\mathrm{Id}$, $b=0$, and $\mu^{(x)}(n) \equiv \indc_{n = 2}$, the shape $\cW$ is the Euclidean ball $B_{\sqrt{2}}$. In this case, the shape is a strictly convex set and its boundary $\partial B_{\sqrt{2}}$ is of class $C^\infty$, meaning that locally, up to a change of coordinates, it is the graph of a $C^\infty$ function defined on an open subset of $\R^{d-1}$. Moreover, the Gaussian curvature $K_{B_{\sqrt{2}}}:\partial B_{\sqrt{2}} \to \R$ of $\partial B_{\sqrt{2}}$ (considered as a manifold) is constant, and satisfies $K_{B_{\sqrt{2}}} \equiv 2^{-(d-1)/2}>0$. Under the assumptions of Theorem \ref{thm:Intro:ShapeHauss} (see Section \ref{sec:Intro:Assumptions}), we obtain the following geometric properties of the limiting shape of the periodic branching diffusion.

    \begin{proposition}\label{prop:Intro:Regularity}
        Under the assumptions of Theorem \ref{thm:Intro:ShapeHauss}, the shape $\cW$ is strictly convex, and its boundary $\partial \cW$ is of class $C^2$. Furthermore, the Gaussian curvature of $\partial \cW$ is positive and bounded away from zero.
    \end{proposition}

    To prove Theorem \ref{thm:Intro:ShapeHauss}, we adapt the approach in \cite{addarioberry2025shapetheorembbmperiodic} to this general setting. Some tools for our analysis have already been introduced by Hebbar, Koralov and Nolen in \cite{MR4162842} to study the transition kernel of the periodic branching diffusion. We begin by introducing these tools, and then we briefly describe the approach.

    \subsection{The transition kernel}\label{sec:TransitionKernel}

The approach to prove Theorem \ref{thm:Intro:ShapeHauss}   relies on bounds on the expected number of particles in different regions of space. This can be achieved by studying the transition kernel, which we proceed to introduce. The \textit{transition kernel} $p:(0,\infty)\times \R^d \times \R^d \to \R$ of the periodic branching diffusion, satisfies, for every measurable function $f:\R^d\to [0,\infty)$, bounded and with compact support, 
\begin{equation}\label{eq:Intro:AdditiveFunctional:TransitionKernel}
    \int_{\R^d} p(t,x,y)f(y) dy = \mathbf{E}_x\Big[\sum_{v\in \cN_t}f(X_t(v)) \Big].
\end{equation}

While $p$ does not characterize the periodic branching diffusion (see \cite[Remark 2.1]{maillard2025generalisedprincipaleigenvaluesglobal}), under our assumptions, we will see that it does characterize the limiting shape. Let us denoted by $m_1:\R^d \to (0,+\infty)$ the first-moment function associated to the periodic family of offspring distributions $\cM$; that is, for every $x\in \R^d$,
\begin{equation*}
    m_1(x):=\sum_{n=2}^\infty n\mu^{(x)}(n)>1. 
\end{equation*}
Under the hypotheses of Theorem \ref{thm:Intro:ShapeHauss}, we will see that the transition kernel of any periodic branching diffusion, with underlying diffusion \eqref{eq:Intro:Diffusion}, is uniquely determined by its \textit{effective branching rate}, which is the {$\Z^d$-periodic} function $R:\R^d\to \R$ defined by
\begin{equation}\label{eq:Intro:EffectiveBranching}
    R(x):=(m_1(x)-1)g(x), \ \forall x\in \R^d.
\end{equation}
We impose hypotheses so that $R\in C^{0,\alpha}(\R^d)$.

The effective branching rate combines the input of the original branching rate and the family of offspring distributions $(\mu^{(x)})_{x\in \R^d}$. To study how this interacts with the trajectory of the particles, we introduce the following family of eigenvalue problems parametrized by $\zeta \in \R^d$. Let us consider the matrix-valued function $A=\sigma \sigma^T$. We will denote by $a_{ij}$ the entries of $A$, which clearly satisfy $a_{ij}= a_{ji}$, and we will denote the $i$-th row of $A$ by $a_i$. Fixing $\zeta \in \R^d$, and letting $\psi \in C^2(\R^d)$, and $x\in \R^d$, we define
\begin{equation}\label{eq:Intro:GeneratorExponential}
    \mathcal{L}^\zeta \psi(x) = \frac{1}{2} \sum_{i=1}^d \sum_{j=1}^d (a_{ij} \partial_{ij} \psi)(x) + (\zeta^T A \nabla \psi)(x) + (b\cdot \nabla \psi)(x) \ \text{ and } \ K^\zeta(x)   = \frac{1}{2}(\zeta^T A \zeta)(x) + \zeta\cdot b(x) + R(x).
\end{equation}
We seek the largest $\gamma(\zeta)\in \R$ for which we can find a non-trivial $\Z^d$-periodic function $\psi(\cdot;\zeta) \in C^{2}(\R^d)$ satisfying,
\begin{equation}\label{eq:Intro:EigenvalueProblem}
         \mathcal{L}^\zeta \psi(x;\zeta) + K^\zeta(x)\psi(x;\zeta) = \gamma(\zeta) \psi(x;\zeta), \ \ \ \mathrm{ for } \ \mathrm{ every } \ x\in \R^d.
\end{equation}
If such a number $\gamma(\zeta)$ exists, it is known as the \textit{principal eigenvalue} of \eqref{eq:Intro:EigenvalueProblem}, and any function $\psi(\cdot;\zeta)\not \equiv 0$ satisfying \eqref{eq:Intro:EigenvalueProblem} is a \textit{principal eigenfunction} of \eqref{eq:Intro:EigenvalueProblem}. The principal eigenvalue has been studied in many works, including for example \cite[Theorem 4.11.1]{MR1326606}, \cite{berestycki:hal-00003742}, \cite{MR2155900}, \cite{MR1900178}, \cite{MR2555178}, \cite{MR2127993}. We present some properties of the principal eigenvalue in Section \ref{sec:EigenvalueProblem}, as they play a crucial role in the proof of the shape theorem, and in understanding regularity properties of the shape. 

With the principal eigenvalue at hand, we introduce the effective drift of the periodic branching diffusion. For $\gamma$ as in \eqref{eq:Intro:EigenvalueProblem}, let the \textit{effective drift} $\vb\in \R^d$ be defined by
\begin{equation}\label{eq:Intro:EffectiveDrift}
    \vb = \vb(b,\sigma,R):= \nabla \gamma(0).
\end{equation}
The fact that $\gamma:\R^d \to \R$ is differentiable will be shown in Proposition \ref{Prop:PropertiesGamma}. The definition of $\vb$ is borrowed from \cite[Lemma 4.1]{MR4162842}, its name is justified by the following asymptotics for $p$. Given $L>0$, for $x,y \in \R^d$ such that $|x-y|<Lt$, in  \cite{MR4162842}, Hebbar, Koralov and Nolen prove
\begin{equation}\label{eq:Intro:AsympKernel}
    p(t,x,y) = (2\pi t)^{-2d}e^{-t \Phi (\frac{y-x}{t}) } q\Big(\frac{y-x}{t},x,y\Big)(1+o_L(1)),
\end{equation}
where $\Phi:\R^d \to \R$ is a strictly convex function that achieves a unique negative minimum at $\vb$, and $q$ is bounded on compact sets (see Section \ref{sec:PropertiesShape} for a precise definition of $\Phi$). The exponential term in \eqref{eq:Intro:AsympKernel} is maximal at $y(t,x) = t\vb + x$, so up to small periodic fluctuations, the largest expected concentration of particles effectively drifts as $t\vb$. 

Our argument for the shape theorem does not use the asymptotics \eqref{eq:Intro:AsympKernel}, nor the function $\Phi$ above. While we borrow the definition of the effective drift and use some of its properties from \cite{MR4162842}, we also use a different characterization of $\vb$ in terms of an associated periodic diffusion, which will be useful for establishing its ergodic properties (see Corollary \ref{corollary:Ergodic:AsymptoticDriftDiffusion}). The effective drift has also appeared previously in the Aronson-type estimates for the transition kernel in \cite{MR1482931}. 

Given $e\in S^{d-1}$, relative to the effective drift $\vb$, the speed at which the periodic branching diffusion spreads into half-spaces in the direction $e$ (see Proposition \ref{prop:Intro:DirectionalBounds}) is given by
\begin{equation}\label{eq:Intro:SpeedRelativeTo}
    \css(e) := \Bigg(\min_{\lambda>0}\frac{\gamma(\lambda e)}{\lambda}\Bigg) - \vb \cdot e > 0.
\end{equation}
The expression above and its prior description will be justified in Section \ref{sec:EigenvalueProblem} (see Remark \ref{r:derivative_cs}). The limiting shape is obtained as an intersection of half-spaces. For $e\in S^{d-1}$ and $c\in \R$, let us introduce the lower half-space $\cH^-_{e,c}:=\{x\in \R^d:x\cdot e\leq c\}$. With the function $\css:S^{d-1}\to \R$ at hand, the shape $\Tilde{\cW}$ is obtained by the following Wulff shape construction,
    \begin{equation}\label{eq:Intro:WulffShape}
        \Tilde{\cW} =  \Tilde{\cW}(\css) := \bigcap_{e\in S^{d-1}} \cH_{e,\css(e)}^-.
    \end{equation}
The convex set $\cW$ appearing in Theorem \ref{thm:Intro:ShapeHauss} is the set $\cW = \vb + \Tilde{\cW}$.

\subsection{Brief literature review and the approach}\label{sec:EffectiveDrift}

To prove Theorem \ref{thm:Intro:ShapeHauss}, we adapt the approach of \cite{addarioberry2025shapetheorembbmperiodic} to this general periodic setting to include diffusion, drift and general offspring distribution. While we do not use the asymptotics in \eqref{eq:Intro:AsympKernel} to prove Theorem \ref{thm:Intro:ShapeHauss}, we do use said asymptotics to prove further properties of $\cW$ in Proposition \ref{prop:Intro:Regularity}.

We start by describing the approach to prove Theorem \ref{thm:Intro:ShapeHauss}, which we adapt from the argument in \cite[Theorem 1.1]{addarioberry2025shapetheorembbmperiodic}. Theorem \ref{thm:Intro:ShapeHauss} is a fairly straightforward consequence of the following quantitative shape theorem for the convex hull of the periodic branching diffusion (see Section \ref{sec:ProofOfHaussShape}). 
    
\begin{proposition}\label{prop:Intro:QuantShape}
    Let $H_t$ be the convex hull of $\mathcal{X}_t$. Under the assumptions of Theorem \ref{thm:Intro:ShapeHauss} on the $\Z^d$-periodic coefficients $b,\sigma$, the branching rate $g$, and the family of offspring distributions $\cM$, there exists $\Tilde{\cW}= \Tilde{\cW}(b, \sigma, g, \cM)\subset \R^d$ with $0 \in \mathrm{int}(\Tilde{\cW})$, and $\vb = \vb(b, \sigma, g, \cM)\in \R^d$ such that, for every $\epsilon \in (0,1)$, there exists $C=C(\epsilon)>0$, $\Gamma=\Gamma(\epsilon)>0$ and $T_0=T_0(\epsilon, d, g, \sigma, b, \cM)$ such that, for every $t>T_0$,
        \begin{equation}\label{eq:quantShapeThm}
            \inf_{x\in \R^d}\mathbf{P}_x[(1-\epsilon)\Tilde{\cW}\subset t^{-1}(H_t-x-t\vb) \subset(1+\epsilon)\Tilde{\cW}] \geq 1-Ce^{-t\Gamma(\epsilon)}.
        \end{equation}
\end{proposition}
    
The convex hull $H_t$ of the periodic branching diffusion is completely determined by its extreme points, and thus, the quantitative shape theorem above simplifies the problem to determine the typical position of the extreme particles of the periodic branching diffusion. This can be achieved employing a direction-by-direction approach. We adopt the following notation: for $e\in S^{d-1}$ and $c\in \R$, we let
    \begin{equation}\label{eq:Intro:HyperplaneHalf-space}
        \cH_{e,c}:=\{x\in \R^d: x\cdot e = c\},\
        \cH_{e,c}^+:=\{x\in \R^d: x\cdot e> c\}  \text{ and } \cH_{e,c}^-:=\{x\in \R^d: x\cdot e \leq c\},
    \end{equation}
and we say that $\cH_{e,c}$ is a \textit{hyperplane} with normal $e$, $\cH_{e,c}^+$ is an \textit{upper half-space} in the direction $e$, and $\cH_{e,c}^-$ is a \textit{lower half-space} in the direction $e$. 

In \cite{addarioberry2025shapetheorembbmperiodic}, the authors introduced geometric tools that reduce the proof of the quantitative shape theorem above to obtaining estimates on the probability of particles traveling into moving half-spaces. Relative to the effective drift, we identify the speed at which the periodic branching diffusion spreads into half-spaces in a given direction $e$.

\begin{proposition}\label{prop:Intro:DirectionalBounds}
    Assume the hypotheses of Theorem \ref{thm:Intro:ShapeHauss}. Let $\css:S^{d-1}\to (0,+\infty)$ be defined by \eqref{eq:Intro:SpeedRelativeTo}. There exists a positive constant $C>0$ depending on $b$, $\sigma$, $g$ and $\mathcal{M}$,  such that, for every $\epsilon\in (0,1)$, there exists $\Gamma>0$ with
        \begin{equation}\label{eq:Intro:UpperEstimate}
            \sup_{x\in \R^d}\probaa{x}{}{ (\mathcal{X}_t-x-t\vb)\cap t\cH^+_{e,(1+\epsilon)\css(e)} \neq \emptyset} \leq Ce^{-t\Gamma}, \ \forall\ t\geq 0, \ \mathrm{and } \  e\in S^{d-1},
        \end{equation}
        and, for every $e\in S^{d-1}$, there exists $T=T(e)>0$ with
        \begin{equation}\label{eq:Intro:LowerEstimate}
            \sup_{x\in \R^d}\probaa{x}{}{(\mathcal{X}_t-x-t\vb)\subset t\cH^-_{e, (1-\epsilon)\css(e)} } \leq Ce^{-t\Gamma},\ \text{for all } t>T.
        \end{equation}
    \end{proposition}
We refer to \eqref{eq:Intro:UpperEstimate} as the upper half-space estimate, and \eqref{eq:Intro:LowerEstimate} as the lower half-space estimate. These bounds are achieved by using a similar approach to the one employed in Lubetzky, Thornett and Zeitouni \cite{MR4492971}, who obtain a finer convergence result for general one-dimensional periodic branching diffusions. Using our notation, for a constant $\lambda>0$, they determine the exact limit of the distribution of ${\sup\{c\in \R: (\mathcal{X}_t-x)\cap \cH_{1,c}^+\neq\emptyset\} - t\cs(1) + \frac{3}{2\lambda}\log(t)}$ (see \cite[Theorem 1]{MR4492971}), this in particular encodes a log correction similar to \cite{MR0494541, MR3043938, MR3463416, shabani2019logarithmic, MR0670523}. Their approach relies on a renewal argument that does not directly generalize to the multidimensional setting, and in particular does not give the logarithmic correction in our setting. However, their ideas adapt nicely at the level of large deviations, which is sufficient to prove Proposition \ref{prop:Intro:DirectionalBounds}.
    
With Proposition \ref{prop:Intro:DirectionalBounds} at hand, the proof of Proposition \ref{prop:Intro:QuantShape} is a straightforward consequence of an approximation result for Wulff shapes (see \cite[Proposition 1.6]{addarioberry2025shapetheorembbmperiodic}). The proof of Theorem \ref{thm:Intro:ShapeHauss} follows from Proposition \ref{prop:Intro:QuantShape} via a geometric argument. We refer to Section \ref{sec:TheProof}, or \cite[Section 1.3]{addarioberry2025shapetheorembbmperiodic} for a more detailed exposition.

We proceed to explain our approach to prove further properties of the shape $\cW$ in Proposition \ref{prop:Intro:Regularity}, for which we use the asymptotics \eqref{eq:Intro:AsympKernel}. From \eqref{eq:Intro:AsympKernel} and \eqref{eq:Intro:AdditiveFunctional:TransitionKernel}, we deduce that for fixed $x\in \R^d$, $\epsilon\in (0,1)$ sufficiently small, and any compact set $K \subset\{{y\in \R^d:\Phi(y)\leq -\epsilon}\}$, the expected number of particles in the set $tK+x$ grows exponentially fast to $+\infty$ as $t\to \infty$. This suggest that the shape $\cW$, which in principle is obtained via the Wulff construction in \eqref{eq:Intro:WulffShape}, can be characterized in terms of the level sets of the function $\Phi$, similar to \cite[Theorem VII.3.1]{Freidlin+1985} for the spreading of solutions of the F-KPP equations in periodic media. We achieve this in Proposition \ref{prop:shape=LargeDev}, employing the asymptotics \eqref{eq:Intro:AsympKernel}. The regularity of the shape is then derived from known regularity properties of $\Phi$ (see \cite[Remark 4.2]{MR4162842}).

Finally, the speed $\css:S^{d-1}\to (0,+\infty)$ appearing in Proposition \ref{prop:Intro:QuantShape} inherits some regularity properties from the regularity of $\Tilde{\cW}$ (see Proposition \ref{prop:PropertiesShape:regularitySupport}). This is achieved by observing that $\css:S^{d-1}\to (0,+\infty)$ can be recovered as the \textit{support function} of the convex body $\Tilde{\cW}$, that is, for every $e\in S^{d-1}$, $\css(e) := \inf\{c\in \R:  \Tilde{\cW}\subset \cH^-_{e,c} \}$.

For a detailed exposition of a closely detailed approach, and more detailed literature review, we refer the reader to \cite[Section 1.3, Section 1.4]{addarioberry2025shapetheorembbmperiodic}.

\subsection{Brief discussion of the F-KPP equation}
The shape theorem (Theorem \ref{thm:Intro:ShapeHauss}) is fundamentally linked to the spreading properties of solutions to the Fisher-Kolmogorov-Petrovsky-Piscounov (F-KPP) equation, which we now briefly discuss. It is known (see \cite[Theorem A.4]{maillard2025generalisedprincipaleigenvaluesglobal}, \cite[Chapter IV]{MR246376}) that for any continuous function $q_0:\R^d \to [0,1]$, such that $1-q_0$ has compact support, the multiplicative functional $q$ defined by 
\begin{equation}\label{eq:FKPP:multiplicative}
    q(t,x):=\mathbf{E}_x\Big[\prod_{v\in \cN_t}q_0(X_{t}(v))\Big],
\end{equation}
is a classical solution of the non-linear partial differential equation
\begin{equation}\label{eq:FKPP_Proba}
    \begin{cases}
        \partial_t q(t,x) = \mathcal{L} q + f(x,q) & \mathrm{for} \ (t,x)\in (0,\infty)\times \R^d,\\
        q(0,x) = q_0(x) & \mathrm{for} \ x\in \R^d,
    \end{cases}
\end{equation}
where $\mathcal{L}$ is the generator of the process defined by \eqref{eq:Intro:Diffusion} (see \eqref{eq:Prere:Generator}), and the non-linearity $f:\R^d \times [0,1] \to \R$ is given by
\begin{equation}
    f(x, q) = g(x)\sum_{n\geq 2} \mu^{(x)}(n) (q^n-q).
\end{equation}
The traditional F-KPP literature often focuses on the function $q^-(t,x) = 1-q(t,x)$ satisfying
\begin{equation}\label{eq:FKPP_PDE}
    \begin{cases}
        \partial_t q^-(t,x) = \mathcal{L} q^- + f^{-}(x,q^-) & \mathrm{for} \ (t,x)\in (0,\infty)\times \R^d,\\
        q^-(0,x) = q^-_0(x) & \mathrm{for} \ x\in \R^d,
    \end{cases}
\end{equation}
where the nonlinearity $f^{-}(x,q^-) := -f(x,1-q^-)$ satisfies the constraint $f^{-}(x,q^-) \geq 0$. 

The level set $\cW^{-}_{t,\epsilon}:=\{x\in \R^d:q^{-}(t,x)\geq 1-\epsilon\}$, where the function $q^{-}(t, \cdot)\approx 1$, satisfies a classical spreading result, analogous in nature to Theorem \ref{thm:Intro:ShapeHauss} (see \cite[Theorem 2]{MR553200}, \cite[Theorem 7.3.1]{Freidlin+1985}). This result for $\cW^-_{t,\epsilon}$ has been revisited many times in the PDE literature, employing different approaches including, a perturbation method of viscosity solutions by Evans and Souganidis \cite{MR982575}, a discrete dynamical system method by Weinberger \cite{MR1943224}, and a PDE approach for space-time periodic coefficients by Berestycki, Hamel and Nadin in \cite{MR2473253} and for general nonlinearities by Rossi in \cite{MR3682669}. 

A finer result for the spreading of solutions has been achieved both in homogeneous media, employing probabilistic techniques in \cite{MR0670523}, and in periodic media using purely PDE techniques in \cite{shabani2019logarithmic}. In the homogeneous setting, for $b\equiv 0$ and $\sigma\equiv \mathrm{Id}$, G\"artner shows in \cite{MR0670523} that for solutions with compactly supported and radially symmetric initial data, $\cW^{-}_{t,\epsilon}$ is within constant Hausdorff distance from a ball with radius $\sqrt{2}t-\tfrac{d-2}{2\sqrt{2}}\log t $. In \cite{shabani2019logarithmic}, Shabani proves an analogous result in periodic media for $b\equiv 0$ and $\sigma\equiv \mathrm{Id}$, and solutions with compactly supported initial conditions. In particular, Shabani determines continuous functions $w,\ell:S^{d-1}\to (0,\infty)$ such that, for every $t$ sufficiently large, $\cW^{-}_t$ is within a constant Hausdorff distance from the set $$\{se:e\in S^{d-1}, 0\leq s\leq w(e)-\tfrac{d-2}{\ell(e)}\log t\}.$$

For a more detailed discussion of this connection, we refer to \cite[Section 1.2]{addarioberry2025shapetheorembbmperiodic}.

\subsection{Challenges, improvements and further work}

Let us introduce the challenges and improvements in this paper. The new difficulties arise from generalizing key tools used in the proof of Proposition \ref{prop:Intro:DirectionalBounds}, and are summarized below.

        \textbf{Ergodic properties of diffusions on the torus:} 
        The approach in \cite{addarioberry2025shapetheorembbmperiodic} relies on ergodic properties of Brownian motion on the torus. Ergodic properties of general diffusions on the torus have been studied in \cite[Chapter 3]{MR503330}, \cite[Section 4.11]{MR1326606}, and \cite[Chapter 6.4]{MR2382139}, among other works. Building on these results, we develop the tools needed for our analysis, including well-posedness of the Poisson equation with periodic boundary conditions, quantitative ergodic theorems on the torus, and large deviation estimates for the displacement of general periodic diffusions from their asymptotic drift.

        \textbf{Generalization of the many-to-one lemma:} We obtain a many-to-one lemma in this setting (Proposition \ref{lemma:ManyToOneLemma}) by applying a general many-to-one lemma from \cite{MR3606740}. This requires that we introduce a martingale-functional of the trajectories of the diffusion \eqref{eq:Intro:Diffusion}, and characterize a diffusion under a Girsanov change of measure (see \cite[Proposition VIII.3.4]{MR1725357}).  

        \textbf{Interpolation lemma:} In \cite{addarioberry2025shapetheorembbmperiodic}, the heat kernel of the Brownian motion is used to control the probability of the event that a particle behaves ballistically on time intervals of order $O(1)$. This is called the interpolation lemma, as it allows us to strengthen some of the results for the branching process at discrete times, to continuous time.  Unlike in \cite{addarioberry2025shapetheorembbmperiodic}, we avoid using heat kernel estimates by applying ergodic properties of an associated periodic diffusion.

We proceed to mention improvements over existing literature. To our knowledge, one of the only references treating the regularity of the shape $\cW$ is that of Shabani \cite{shabani2019logarithmic}, for the case $b\equiv 0$, $\sigma\equiv \mathrm{Id}$, and assuming that the $\Z^d$-periodic branching rate $g$ is of class $C^\infty$ and bounded away from zero. In \cite[Section 2.2]{shabani2019logarithmic}, it is shown that $\cW$ is strictly convex. Additionally, it can be deduced from those results that $\partial \cW$ is of class $C^1$. Thus, Proposition \ref{prop:Intro:Regularity} is a direct improvement over the regularity obtained in \cite{shabani2019logarithmic}, as we allow for more general hypotheses on $\sigma$, $b$ and $g$, and obtain further regularity properties, including the fact that $\partial \cW$ is of class $C^2$, and that the Gaussian curvature is bounded away from zero.  Additionally, we present an independent complete proof of the strict convexity of the shape $\cW$.

The argument in \cite{shabani2019logarithmic} to prove that $\partial \cW$ is of class $C^1$ relies on the fact that the principal eigenvalue $\gamma$ is analytic. This fact has been obtained in \cite[Section 2.1]{MR1484944} by spectral methods, when the operator $\mathcal{L}$ associated to the SDE \eqref{eq:Intro:Diffusion} is self-adjoint, and under stronger regularity assumptions. For non self-adjoint operators $\mathcal{L}$, the fact that $\gamma$ is analytic has also been used in different references in the literature, but we were unable to find a complete proof of this fact. Under less restrictive assumptions on $\sigma$ and $b$, this has appeared in \cite[Section 4]{MR2746770} and \cite[Proof of Lemma 2.1]{MR2401143}, with a proof outline. We provide a detailed proof of this fact in Proposition \ref{prop:Eigenvalue:Analytic}. In addition to classical perturbation theory results \cite[Section VII.2]{MR1335452}, our argument relies on a device by Siciak \cite[Theorem 1]{MR279263}. The extra regularity we impose on $\sigma$ and $b$ is not necessary to prove that $\gamma$ is analytic. Nevertheless, the extra regularity we impose on $\sigma$ and $b$ is needed to applying \cite[Theorem 8.2.10, Theorem 4.11.1]{MR1326606}, which establish fundamental properties of the principal eigenvalue which are key to our analysis.

The regularity of the speed $\css:S^{d-1}\to (0,
\infty)$, introduced in \eqref{eq:Intro:SpeedRelativeTo}, appears to be of interest in the PDE literature. It is known that $\css$ is continuous, see for example \cite[Proposition 3.2(ii)]{MR4575461} for an explicit statement, and it is also a consequence of the PDE results in \cite{MR1900178, MR2607315}. We prove that $\css$ is in fact of class $C^1$ in Lemma \ref{lemma:DifferentiabilityLambda_e}. We also prove the continuity of $\css$ via a probabilistic argument in Proposition \ref{prop:ShapeTheorem:ContinuityCSS}, employing the fact that $\css$ is the spreading speed of the periodic branching diffusion. Under our hypotheses, we obtain further regularity for $\css$, showing that it is in fact of class $C^2$ in Proposition \ref{prop:PropertiesShape:regularitySupport}.

We finish this section by mentioning potential future research directions. A natural direction is to obtain a logarithmic correction for the particles that travel the farthest in a fixed direction. Two key ingredients for this would be a barrier estimate similar to \cite[Lemma 2.9]{MR4492971}, which can be likely achieved employing the ideas introduced in Section \ref{sec:Ergodic}, and a second-moment method employing a many-to-two lemma (see \cite[Section 4.2]{MR3606740}).

Going further, one could hope to obtain a multidimensional logarithmic correction that considers all directions simultaneously, in the spirit of \cite{MR3417448} or \cite{MR4564433} for the maximal displacement of $d$-dimensional branching Brownian motion (BBM). Analogous results for the F-KPP equation in homogeneous and periodic media have been obtained in \cite{MR0670523} and \cite{shabani2019logarithmic} respectively. The key insight in \cite{MR3417448} is that, employing an approximation result of the Euclidean ball (see \cite[Remark 1.2, Lemma 2.3]{MR3417448}), the maximal displacement of the BBM arises by examining the spread of the branching process in approximately $t^{\frac{d-1}{2}}$ independent directions. In our setting, the main difference is the possibility of an anisotropic shape $\cW$, for which analogous approximation results exist under the regularity achieved in Proposition \ref{prop:Intro:Regularity} (see \cite[Lemma 2.3]{MR3417448} and \cite[Section 4]{MR1242984}). The analysis would therefore need to be performed over $t^{\frac{d-1}{2}}$ independent directions, each with their own spreading rate. Beyond the barrier estimates and approximation results mentioned above, additional work is needed, including controlling the correlations of the number of extreme particles at different angles as in \cite[Lemma 4.4]{MR3417448}.

Finally, we believe that the approach in this work can be further adapted to prove a shape theorem for periodic branching diffusion with extinction events. Common approaches for introducing extinction are: by \textit{soft killing}, allowing $\mu^{(x)}(0)>0$; or by \textit{hard killing}, where particles are killed upon hitting a periodic hard obstacle. The criterion for survival is given by the principal eigenvalue $\gamma(0)$ in the spirit of \cite[Theorem 2.1]{berestycki:hal-00003742}, \cite[Theorem 3.1]{MR1698955} or \cite[Theorem 3]{MR2040776}. We conjecture that the branching process, conditioned on the event of survival, satisfies a shape theorem. Nevertheless, conditioning on survival introduces technical challenges. For instance, the trajectory of particles is expected to become biased towards regions of space that favor survival, or to avoid hard obstacles. This bias potentially alters the shape given by the transition kernel, as additive functionals may overestimate the probability that particles access hard-to-reach regions of space. A backbone decomposition of the branching process conditioned on survival in the spirit of \cite{MR3456337} would be needed, and our approach could be adapted to the backbone of said decomposition. Such an approach involves considerable technical difficulties, and will not be pursued here.

\subsection{Notation and assumptions}\label{sec:Intro:Assumptions}

We will use the convention that $[n]=\{1,...,n\}$. We let $M^{d\times d}$ denote the set of $d\times d$ matrices. For a vector $v= (v_1,...,v_d)\in \R^d$, let $|v|$ denote its Euclidean norm, and let $\|v\|_\infty=\max_{i\in [d]}|v_i|$. Also, for $\mathcal{Y}\subset \R^d$ and $f:\mathcal{Y}\to \R$ we let $\|f\|^{\mathcal{Y}}_\infty:= \sup_{x\in \mathcal{Y}}|f(x)|$. When $\mathcal{Y}=\R^d$ we simply write $\|f\|_\infty$. 

We let $L^\infty(\mathcal{Y})$ be the set formed by the functions $f:\mathcal{Y}\to \R$ with $\sup_{x\in \mathcal{Y}}|f(x)|<\infty$, $C^0(\mathcal{Y})$ the set of functions $f:\mathcal{Y}\to \R$ that are continuous, $C_b^0(\mathcal{Y})=C^0(\mathcal{Y}) \cap L^\infty(\mathcal{Y})$, and $C_p^0(\mathcal{Y})$ the subset of $C^0(\mathcal{Y})$ formed by $\Z^d$-periodic functions. 

For $k\in \N$, we let $C^k(\R^d)$ denote the set of $k$-times continuously differentiable real-valued functions, and $C_p^k(\R^d)$ the subset of $C^k(\R^d)$ formed by the $\Z^d$-periodic functions. Additionally, for $\alpha\in (0,1)$, we let $C^{0,\alpha}_p(\R^d)$ denote the set of $\Z^d$-periodic functions $f:\R^d\to \R$ such that $\sup_{x\neq y}|f(x)-f(y)|/|x-y|^\alpha<\infty$, and for $f\in C^{0,\alpha}_p(\R^d)$ we denote its $\alpha$-H\"older norm by ${\|f\|_{0,\alpha}= \|f\|_\infty + \sup_{x\neq y}|f(x)-f(y)|/|x-y|^\alpha}$. Finally, for $k\in \N$, we let $C^{k,\alpha}_p(\R^d)$ be the set formed by those $f\in C^{k}_p(\R^d)$ such that $\tfrac{\partial^k}{\partial_{x_{i_1}x_{i_2}...x_{i_k}}}f\in C^{0,\alpha}_p(\R^d)$ for every $i_1,i_2,...,i_k\in [d]$, with the norm
\begin{equation*}
        \|f\|_{k,\alpha} : = \sum_{0\leq |\beta| \leq k} \Big \| \frac{\partial^\beta f}{\partial x^\beta} \Big \|_\infty + \sup_{|\beta |=k} \sup_{x\neq y}\Big | \frac{\partial^\beta f}{\partial x^\beta}(x)- \frac{\partial^\beta f}{\partial x^\beta}(y)\Big ||x-y|^{-\alpha}.
    \end{equation*}

Let us first describe the assumptions made for the $\Z^d$-periodic drift $b:\R^d\to \R^d$ and the $\Z^d$-periodic matrix-valued function $\sigma:\R^d \to M^{d\times d}$. We assume:

\begin{assumptionlistA}
    \item \label{AssumptionA:bSigmaRegularity} For some $\alpha\in (0,1)$, each of the entries $b_i$ of $b$ belongs to $C_p^{1,\alpha}(\R^d)$, and each of the entries $\sigma_{ij}$ of $\sigma$ belongs to $C_p^{2,\alpha}(\R^d)$.
    \item \label{AssumptionA:SigmaEllipticity} There exists $\theta\in (0,1)$ such that $\theta \mathrm{Id} \leq \sigma \sigma^T \leq \theta^{-1} \mathrm{Id}$.
\end{assumptionlistA}

Assumption \textbf{\ref{AssumptionA:bSigmaRegularity}} is used in the application of results for periodic diffusions from \cite{MR1326606} (See \cite[Theorem 4.11.1]{MR1326606} and Assumption $\Tilde{H}$ in \cite[Section 3.7]{MR1326606}), which are stated for operators in divergence form. The results in our work apply to an operator in non-divergence form, which is why we require extra regularity to transform the operator into divergence form. Assumption \textbf{\ref{AssumptionA:SigmaEllipticity}} is a uniform ellipticity condition on $\sigma$.

For the $\Z^d$-periodic branching rate function $g:\R^d \to [0,+\infty)$ we assume:
\begin{assumptionlistB}
    \item \label{AssumptionB:g} For some $\alpha\in (0,1)$, $g$ belongs to $C^{0,\alpha}_p(\R^d)$, and $g\not \equiv 0$.
\end{assumptionlistB}
Without loss of generality, we can assume that the $\alpha$ appearing in \textbf{\ref{AssumptionA:bSigmaRegularity}} and the $\alpha$ appearing in \textbf{\ref{AssumptionB:g}} are the same, if they are different we simply take $\alpha$ to be the minimum of both. Notice that we allow $g$ to take the value zero. As long as $g$ is not identically zero, by ergodic properties of diffusions on the torus, the periodic diffusion spends enough time in positive branching rate regions to eventually branch into a shape.

The assumptions on the periodic family of offspring distributions $\cM =(\mu^{(x)})_{x\in \R^d}$ are:
\begin{assumptionlistC}
    \item \label{AssumptionC:measurability} For every $n\in \N$, the function $x \mapsto \mu^{(x)}(n)$ is Borel-measurable.
    \item \label{AssumptionC:NoKillingPureBranching}For every $x\in \R^d$, the support of $\mu^{(x)}$ is contained in $\N\setminus \{0,1\}$.
    \item \label{AssumptionC:HolderContinuity} Let $m_1:\R^d \to \R$ be defined by $m_1(x):=\sum_{i=2}^\infty n \mu^{(x)}(n)$. Then, $m_1\in C_p^{0,\alpha}(\R^d)$.
    \item \label{AssumptionC:StochasticDomination} There exists a probability measure $\bar{\mu}$ on $\N$, with finite first moment, that stochastically dominates $\mu^{(x)}$ for every $x\in \R^d$.
\end{assumptionlistC}
Assumption \textbf{\ref{AssumptionC:measurability}} is required for measurability purposes. Assumption \textbf{\ref{AssumptionC:NoKillingPureBranching}} ensures pure branching events. Assumption \textbf{\ref{AssumptionC:HolderContinuity}}, combined with \textbf{\ref{AssumptionB:g}}, provides regularity of the effective branching rate $R$ defined in \eqref{eq:Intro:EffectiveBranching}. An example of a family satisfying Assumption \textbf{\ref{AssumptionC:HolderContinuity}} is a family  $\cM=(\mu^{(x)})_{x\in \R^d}$ of integrable measures satisfying that, for every $n\in \N$, the function $x\mapsto \mu^{(x)}(n)$ belongs to $C^{0,\alpha}_p(\R^d)$. Assumption \textbf{\ref{AssumptionC:StochasticDomination}} provides uniform control over the family of offspring distributions $\cM=(\mu^{(x)})_{x\in \R^d}$. For instance, any family of distributions $\cM$ for which we can find $K \geq 2$ such that for every $k>K$ and $x\in \R^d$, $\mu^{(x)}(k) = 0$, is stochastically dominated by the homogeneous branching rate
\begin{equation*}
        \mu(n):=\begin{cases}
            1, \ \mathrm{if } \ n=K\\
            0, \ \mathrm{otherwise}.
        \end{cases}
\end{equation*}

\subsection{Outline}
We finish the introduction with an outline of the paper. In Section \ref{sec:prelim} we present preliminaries, including the ergodic properties of diffusions on the torus in Section \ref{sec:Ergodic}, properties of the principal eigenvalue in Section \ref{sec:EigenvalueProblem}, the many-to-one lemma in Section \ref{sec:ManyToOne}, and an interpolation result in Section \ref{sec:interpolation}. The proof of the upper and lower estimates (Proposition \ref{prop:Intro:DirectionalBounds}) are given in Section \ref{sec:Half-Space}, a survival result is proved in Section \ref{sec:SurvivalLargeDeviation}, and the geometric argument to prove Theorem \ref{thm:Intro:ShapeHauss} is given in Section \ref{sec:TheProof}. The geometric properties of the shape and its support function are studied in Section \ref{sec:PropertiesShape}.

\section*{Acknowledgements}

The author gratefully acknowledges the financial support provided by the Fonds de recherche du Qu\'ebec Doctoral research scholarship (DOI: \hyperlink{}{https://doi.org/10.69777/336282}). This work was partially funded by this scholarship, which provided the crucial resources and dedicated time necessary to conduct this study. The author would also like to thank Jessica Lin and Louigi Addario-Berry for their careful proofreading, useful discussions and guidance throughout the development of this work.

\section{Preliminaries}\label{sec:prelim}

Through this work, we will consider a weak solution of \eqref{eq:Intro:Diffusion} defined on a filtered probability space $(\Omega,\mathcal{F},\{\cF_t\}_{t \ge 0},\bP)$, where the filtration $\cF_t$ is right-continuous and complete. For $x\in \R^d$, we let $\bP_x$ be the measure on $(\Omega,\cF)$ under which $Z_t$ is a weak solution of \eqref{eq:Intro:Diffusion} with $Z_0 =x$. We let $\bE_x$ denote the expectation operator associated to $\bP_x$.

Let $A=\sigma \sigma^T$ be the \textit{diffusion matrix} associated with the diffusion \eqref{eq:Intro:Diffusion}. Let us denote by $(a_{ij})_{i,j \in [d]\times [d]}$ the entries of $A$ and by $a_i$ the $i$-th row of $A$. The generator of the diffusion \eqref{eq:Intro:Diffusion} is defined for every function $q\in C^2(\R^d)$ by
\begin{equation}\label{eq:Prere:Generator}
    \mathcal{L} q (x) = \frac{1}{2}\sum_{i=1}^d \sum_{j=1}^d a_{ij} (x) \partial_{ij} q(x) + b(x) \cdot \nabla q(x).
\end{equation}
The assumptions on $A$ and $b$ are \textbf{\ref{AssumptionA:bSigmaRegularity}} and \textbf{\ref{AssumptionA:SigmaEllipticity}} defined in Section \ref{sec:Intro:Assumptions}. We let $\kappa:\R^d \to \R^d$ be defined by
\begin{equation}\label{eq:Ergodic:DivDrift}
    \kappa(x):= b(x) - \frac{1}{2}\sum_{j=1}^d \partial_j a_j(x),
\end{equation} 
where by assumption \textbf{\ref{AssumptionA:bSigmaRegularity}}, each of the entries of $\kappa$ belongs to $C^{1,\alpha}_p(\R^d)$.

Since the coefficients of \eqref{eq:Intro:Diffusion} are $\Z^d$-periodic, the projected process $ \dot{Z}_t= Z_t \mod \mathbb{Z}^d$ is a well-defined diffusion on the torus, whose generator coincides with $\mathcal{L}$ on the core $C^2_p(\R^d)$. Expressed in divergence-form, \eqref{eq:Prere:Generator} has the form
\begin{equation}\label{eq:Generator:DivergenceForm}
    \mathcal{L} q (x) = \frac{1}{2}\mathrm{div}(A(x)\nabla q(x))+ \kappa(x) \cdot \nabla q(x).
\end{equation}
The formal adjoint of $\mathcal{L}$ is thus given by
\begin{equation}\label{eq:Prere:Adjoint}
    \mathcal{L}^* p(x) = \frac{1}{2}\mathrm{div}(A(x) \nabla p )(x)  - \kappa(x)\cdot \nabla p(x) - \mathrm{div}(\kappa)(x) p(x).
\end{equation}

Notice that any $\Z^d$-periodic function $f$ can always be identified with a function $\dot{f}$ defined on the torus $\TT^d$. We will use $f$ and $\dot{f}$ interchangeably without distinction for ease of notation. Given a measurable function $\nu:\TT^d \to [0,\infty)$, and $f:\TT^d\to \R$ such that the product $f\nu$ is integrable on the torus, we denote this integral by
\begin{equation}
    \nu(f):=\int_{\TT^d} f(x)\nu(x)dx.
\end{equation}

Several results in this section rely on the Feynman-Kac formula, which provides a probabilistic representation for classical solutions of the parabolic equation
\begin{equation}
    \begin{cases}\label{eq:FK}
        \partial_t q = \mathcal{L} q + K(x)q & \text{in } (0, \infty) \times \mathbb{R}^d, \\
        q(0, x) = q_0(x) & \text{in } \mathbb{R}^d.
    \end{cases}
\end{equation}
We will state this result in terms of the operator $\mathcal{L}$ in \eqref{eq:Intro:Diffusion} for ease of notation. This result holds for general operators with coefficients $\sigma$ and $b$ in \eqref{eq:Intro:Diffusion} which are Lipschitz continuous. For $K,q_0\in C^0_b(\R^d)$, the Feynman-Kac formula states that if $q$ is a classical solution of \eqref{eq:FK}, with $q\in L^\infty([0,T]\times \R^d)$ for every $T>0$, and such that $\partial_t q, \partial_{x_i} q, \partial_{x_i x_j}q \in  L^\infty([0,\infty)\times \R^d)$ for all $i,j\in [d]$, then it has the following probabilistic representation:
\begin{equation}\label{eq:FeynKac}
    q(t,x):=\mathbb{E}_x\Bigl[q_0(Z_{t})e^{\int_0^t K(Z_s)ds}\Bigl],
\end{equation}
where $Z_t$ is a weak solution of the periodic diffusion \eqref{eq:Intro:Diffusion}. For a more detailed exposition, see \cite[Section 2.2, Theorem 2.2]{Freidlin+1985}.

\subsection{Ergodic properties of periodic diffusions}\label{sec:Ergodic}

The ergodic properties of $\dot{Z}_t$ will play an important role in describing the long-term behavior of the periodic diffusion $Z_t$, which in turn will be useful to describe the long term behavior of the branching process. 

To avoid introducing unnecessary notation, we will express the results in terms of the periodic diffusion \eqref{eq:Intro:Diffusion} and the generator \eqref{eq:Prere:Generator}. All of the results that follow in this section only rely on \textbf{\ref{AssumptionA:bSigmaRegularity}} and \textbf{\ref{AssumptionA:SigmaEllipticity}}. In the future we will apply these results to other periodic diffusions of interest.

Let us denote by $p(t,x,y)$ the transition kernel of $\dot{Z}_t$ in the torus, that is, $p$ satisfies that, for every Borel subset $A\subset \TT^d$, $\mathbb{P}(\dot{Z}_t\in A) = \int_A p(t,x,y) dy$. We say that a probability measure $\mu$ on the Borel $\sigma$-algebra $\mathcal{B}(\TT^d)$ of $\TT^d$ is \textit{invariant} for $\dot{Z}_t$, if, for every $t\geq 0$ and $A\in \mathcal{B}(\TT^d)$, $\int_A p(t,x,y) \mu(dy) = \mu(A)$. The invariant measure of a diffusion on the torus has a known characterization in terms of the adjoint operator $\mathcal{L}^*$ described in the next proposition. 

\begin{proposition}\cite[Theorem 4.8.6 and Theorem 4.11.1 (vii)]{MR1326606}\label{prop:Ergodic:InvariantMeasure}
    Let $(Z_t)_{t\geq 0}$ be the periodic diffusion \eqref{eq:Intro:Diffusion} with $\sigma$ and $b$ satisfying \textbf{\ref{AssumptionA:bSigmaRegularity}} and $\textbf{\ref{AssumptionA:SigmaEllipticity}}$. Let $(\dot{Z}_t)_{t\geq 0}$ be its projection into $\TT^d$. There exists a unique invariant measure of $(\dot{Z}_t)_{t\geq 0}$ which has a strictly positive density $\nu \in C^{2,\alpha}_p(\R^d)$. Furthermore, $\nu$ satisfies the adjoint equation $\mathcal{L}^*\nu=0$, and $\nu$ is the unique periodic classical solution of the adjoint equation with the property that $\int_{\TT^d}\nu(x)dx=1$.
\end{proposition}

\begin{proof}
    To help make the connection with \cite[Theorem 4.8.6, Theorem 4.11.1 (vii)]{MR1326606}, let us translate their notation to ours. In the notation of \cite[Theorem 4.11.1 (vii)]{MR1326606}, $\mathcal{L}$ plays the role of $L_0$, we take $V=0$ in their notation, and $\mathcal{L}^*$ plays the role of $\Tilde{L}$. In the same theorem, $\nu$ plays the role of $\Tilde{\phi}_0$, the eigenfunction associated to $\lambda_0$, which is equal to $0$ by \cite[ Theorem 4.11.1 (ii)]{MR1326606}. Thus, this result establishes the existence and uniqueness (up to a multiplicative constant) of the periodic solution $\nu$ to the adjoint equation $\mathcal{L}^*\nu=0$, its regularity, and the fact that $\nu$ can be taken to be strictly positive.

    What remains to be shown is that such a solution, properly normalized, is the density of the invariant measure of the projection of the periodic diffusion on the torus. This is established in \cite[Theorem 4.8.6]{MR1326606}, also in the proof of \cite[Corollary 8.2.11]{MR1326606}, and independently in \cite[Remark 3.3.4]{MR503330}
\end{proof}

The following result can be derived by combining some results for periodic diffusions from \cite{MR1326606} with the Fredholm alternative. While this is a standard method, we provide a proof here. For a function $\eta \in C^{2,\alpha}(\R^d)$, consider the $C^{2,\alpha}$-H\"older norm
    \begin{equation*}
        \|\eta\|_{2,\alpha} : = \sum_{0\leq |\beta| \leq 2} \Big \| \frac{\partial^\beta \eta}{\partial x^\beta} \Big \|_\infty + \sup_{|\beta |=2} \sup_{x\neq y}\Big | \frac{\partial^\beta \eta}{\partial x^\beta}(x)- \frac{\partial^\beta \eta}{\partial x^\beta}(y)\Big ||x-y|^{-\alpha}.
    \end{equation*}

\begin{proposition}\label{prop:Ergodic:Fredholm}
    Let $\nu\in C^{2,\alpha}_p(\R^d)$ be the strictly positive density of the unique invariant measure of $(\dot{Z}_t)_{t\geq 0}$ on the torus, and $f\in C^{0,\alpha}_p(\R^d)$ be such that $\nu(f)=0$. Then, there exists a unique $\eta^f\in C_p^{2,\alpha}(\R^d)$ that satisfies
    \begin{equation*}
        \mathcal{L} \eta^f = f,
    \end{equation*}
    and such that $\nu(\eta^f)=0$.
    Furthermore, there exists a positive constant $C=C(d,\|\sigma\|_{0,\alpha}, \|b\|_{0,\alpha})$ such that
    \begin{equation}\label{eq:Schauder:Fredholm}
        \|\eta^f\|_{2,\alpha} \leq C\|f\|_{0,\alpha}.
    \end{equation}
    
\end{proposition}

\begin{proof}
    Let us consider the closed subspace of $C_p^{0,\alpha}(\R^d)$,
    \begin{equation*}
        C_{p,\nu}^{0,\alpha}(\R^d):=\{f\in C_p^{0,\alpha}(\R^d): \nu(f) = 0\}.
    \end{equation*}
    The pair $(C_{p,\nu}^{0,\alpha}(\R^d),\|\cdot\|_{0,\alpha})$ defines a Banach space which we will refer to as $C_{p,\nu}^{0,\alpha}(\R^d)$ for simplicity. 
    
    By \cite[Theorem 4.11.1 (iv),(v)]{MR1326606}, for $f\in C_{p,\nu}^{0,\alpha}(\R^d)$, there exists a unique $p^f \in C^{2,\alpha}_p(\R^d)$ such that $(1-\mathcal{L})p^f=f$, and a constant $C=C(d,\alpha, \theta)>0$, where $\theta$ is the constant defined in \textbf{\ref{AssumptionA:SigmaEllipticity}}, such that 
    \begin{equation}\label{eq:Scahuder:Resolvent}
        \|p^f\|_{2,\alpha}=\|(1-\mathcal{L})^{-1}f\|_{2,\alpha}\leq C\|f\|_{0,\alpha}
    \end{equation}
    We claim that $\nu(p^f)=0$. Integrating by parts, by properties of the adjoint, and since $\mathcal{L}^* \nu = 0$,
    \begin{equation*}
        \begin{split}
            -\nu(f) & =- \int_{\TT^d} f(x)\nu(x)dx = \int_{\TT^d} (\mathcal{L}-1)p^f \nu(x) dx = \int_{\TT^d} \mathcal{L}p^f(x) \nu(x) dx - \nu(p^f)\\
            & = \int_{\TT^d} p^f(x) \mathcal{L}^*\nu(x) dx - \nu(p^f) = -\nu(p^f).
        \end{split}
    \end{equation*}
    The left-hand side of the expression above is equal to zero since $f\in  C_{p,\nu}^{2,\alpha}(\R^d)$, it follows that $p^f$ belongs to $ C_{p,\nu}^{2,\alpha}(\R^d)$. In particular, $p^f=(1-\mathcal{L})^{-1}f \in C^{0,1}_p(\R^d)$. On the other hand, we know that for every $\epsilon>0$, $C^{0,\alpha+\epsilon}_p(\R^d)$ is compactly embedded into $C^{0,\alpha}_p(\R^d)$. Thus, since $\alpha<1$, $(1-\mathcal{L})^{-1}:C_{p,\nu}^{0,\alpha}(\R^d) \to C_{p,\nu}^{0,\alpha}(\R^d)$ is a compact linear operator. 
    
    With this at hand we have everything necessary to apply the Fredholm alternative. First, we show that the only solution $p \in C_{p,\nu}^{0,\alpha}$ to $p - (1-\mathcal{L})^{-1}p = 0$ is the trivial solution $p\equiv 0$. For instance, if $p\in C_{p}^{0,\alpha}(\R^d)$ solves the aforementioned equation, that means that in fact $p\in C_{p}^{2,\alpha}(\R^d)$, and it satisfies $ p - \mathcal{L}p  = p$, from where, $\mathcal{L}p=0$. By the strong maximum principle, the only solution to this equation is $p$ identically a constant. Since $\nu>0$, the only constant $p$ satisfying $\int_{\TT^d}p\nu(x)dx=0$ is $p=0$. Thus, the only solution $p \in C_{p,\nu}^{0,\alpha}(\R^d)$ to $p - (1-\mathcal{L})^{-1}p = 0$ is the trivial solution $p\equiv 0$. Thus, by the Fredholm alternative, for every $f\in C_{p,\nu}^{0,\alpha}(\R^d)$, there exists a unique $\eta^f \in C_{p,\nu}^{0,\alpha}(\R^d)$ such that $\mathcal{L}\eta^f = f $.
    
    Finally, we observe that by standard Schauder estimates (see \cite[Theorem 3.2.8]{MR1326606}), there exists a positive constant $C=C(d,\|\sigma\|_{0,\alpha}, \|b\|_{0,\alpha})$ such that
    \begin{equation*}
        \| \eta^f \|_{2,\alpha} \leq C(\|\eta^f\|_{0,\alpha} + \|f\|_{0,\alpha}).
    \end{equation*}
    On the other hand, observing that $\eta^f = (1-\mathcal{L})^{-1}(\eta^f-f)$, and employing the argument in the proof of \cite[Theorem 4.11.1(v)]{MR1326606}, we observe that $\|\eta^f\|_{0,\alpha}\leq \|\eta^f - f\|_{\infty}\leq \|\eta^f\|_{\infty} + \|f\|_{\infty}$. Then, we obtain $C=C(d,\|\sigma\|_{0,\alpha}, \|b\|_{0,\alpha})$ such that $\|\eta^f\|_{\infty}\leq C \|f\|_{\infty}$ using the exponential ergodicity of diffusions on the torus (see, \cite[Theorem 3.3.2]{MR2839402}). This follows by applying the reasoning in \cite[Lemma 2.1]{addarioberry2025shapetheorembbmperiodic}, where an analogous bound is obtained for the Laplacian. The bound \eqref{eq:Schauder:Fredholm} follows since $\|f\|_\infty \leq \|f\|_{0,\alpha}$.
\end{proof}

\begin{remark}
    The solution $\eta^f$ to the equation in the proposition above is unique up to an additive constant. To see this, consider two solutions $\eta^1$ and $\eta^2$ of the equation $\mathcal{L}\eta=f$. Since $\mathcal{L}\eta$ only involves the derivatives of $\eta$, $\mathcal{L}(\eta^2-\nu(\eta^2)) =f $ and $\mathcal{L}(\eta^1-\nu(\eta^1)) =f $. Clearly $\nu(\eta^1-\nu(\eta^1))=0=\nu(\eta^2-\nu(\eta^2))$, so the previous proposition implies that $\eta^1-\nu(\eta^1) = \eta^2-\nu(\eta^2)$. This shows that $\eta_1$ and $\eta_2$ only differ by a constant.
\end{remark}

The following ergodic result generalizes \cite[Lemma 2.1]{addarioberry2025shapetheorembbmperiodic}. The proof is identical, except for the step involving the existence of the solution for the Poisson equation $\tfrac{1}{2}\Delta \eta = f$, which is replaced by Proposition \ref{prop:Ergodic:Fredholm} above.

\begin{lemma}\label{lemma:Ergodic:WeakErgodic}
    Let $\nu\in C^{2,\alpha}_p(\R^d)$ be the density of the unique invariant measure of $(\dot{Z}_t)_{t\geq 0}$ on $
    \TT^d$ and $f\in C^{0,\alpha}_p(\R^d)$ be such that $\nu(f)=0$. Fix $\epsilon>0$. There exists a constant $C=C(d,\|f\|_{0,\alpha})>0$ such that, for all $t$ sufficiently large (depending on $d$, $\|f\|_{0,\alpha}$, and $\epsilon$),
    \begin{equation}
        \sup_{x\in \R^d} \probbb{x}{}{\frac{1}{t}\Big|\int_0^t f(Z_s)ds \Big|>\epsilon}\leq C\epsilon^{-2}t^{-1}.
    \end{equation}
\end{lemma}
\begin{proof}[Proof of Lemma \ref{lemma:Ergodic:WeakErgodic}]
    By Proposition \ref{prop:Ergodic:Fredholm}, up to an additive constant, there exists a unique $\eta^f\in C^{2,\alpha}_p(\R^d)$ satisfying $\mathcal{L} \eta^f= f$. By It\^o's formula,
    \begin{equation}
        \begin{split}
            \eta^f(Z_t)-\eta^f(Z_0)& =\int_0^t \mathcal{L} \eta^f(Z_s)ds + \sum_{i=1}^d \int_0^t \frac{\partial \eta^f}{\partial x_i}(Z_s) dB^i_s =\int_0^t f(Z_s)ds + \sum_{i=1}^d \int_0^t \frac{\partial \eta^f}{\partial x_i}(Z_s) dB^i_s.
        \end{split}
    \end{equation}
     Since $\eta^f$ is bounded, for $t \ge 4\|\eta^f\|_\infty/\epsilon$ we have 
    \begin{equation}
        \begin{split}
            \mathbb{P}_{x}{}\left[{\frac{1}{t}\bigg|\int_0^t f(Z_s) ds\bigg|>\epsilon}\right] & = \mathbb{P}_{x}{}\left[{\frac{1}{t}\bigg|\eta^f(Z_t)-\eta^f(Z_0) - \sum_{i=1}^d \int_0^t \frac{\partial \eta^f}{\partial x_i}(Z_s) dB^i_s\bigg|>\epsilon}\right]\\
            & \leq \mathbb{P}_{x}{}\left[{\frac{1}{t} \bigg|\sum_{i=1}^d \int_0^t \frac{\partial \eta^f}{\partial x_i}(Z_s) dB^i_s\bigg|>\epsilon/2}\right]\\
            & \leq \mathbb{P}_{x}{}\left[{\sup_{\ell\in [0,t]}\bigg| \sum_{i=1}^d \int_0^\ell \frac{\partial \eta^f}{\partial x_i}(Z_s) dB^i_s\bigg|>t\epsilon/2}\right].
        \end{split}
    \end{equation} 
    Viewed as a function of $\ell$, the inner sum in the final term of the estimate is a continuous martingale with finite second moment. Thus, by Doob's maximal inequality and It\^o's isometry,
    \begin{equation}\label{eq:Ergodic:DoobsMaximal}
        \begin{split}
            \mathbb{P}_{x}{}\left[{\frac{1}{t}\bigg|\int_0^t f(Z_s) ds\bigg|>\epsilon}\right] & \leq \frac{4}{t^2\epsilon^2}\expecdddd{x}{}{\bigg| \sum_{i=1}^d \int_0^t \frac{\partial \eta^f}{\partial x_i}(Z_s) dB^i_s\bigg|^2} = \frac{4}{t^2\epsilon^2}\sum_{i=1}^d\expecdddd{x}{}{  \int_0^t \Bigl|\frac{\partial \eta^f}{\partial x_i}(Z_s) \Bigl|^2 ds}\\
            & \leq \frac{4td\|\nabla\eta^f\|_{\infty}^2}{t^2\epsilon^2}= \frac{4d\|\nabla\eta^f\|_{\infty}^2}{t\epsilon^2}.
        \end{split}
    \end{equation}
    Finally, from Proposition \ref{prop:Ergodic:Fredholm}, there exists $C>0$ such that $\|\eta^f\|_{2,\alpha}\leq C\|f\|_{0,\alpha}$, which in particular implies that $\|\nabla\eta^f\|_{\infty}\leq C\|f\|_{0,\alpha}$, and the result follows.
\end{proof}

The following corollary follows immediately by applying the previous result to the function $f(\cdot)-\nu(f)$.

\begin{corollary}\label{corollary:Ergodic:QuadVariation}
    Let $f\in C^{0,\alpha}_p(\R^d)$ be a non-negative function with $f\not \equiv 0$. Then,
    \begin{equation}
        \inf_{x\in \R^d}\probbb{x}{}{\int_0^t f(Z_s)ds \to +\infty\ \mathrm{ as }\ t\to +\infty}=1.
    \end{equation}
\end{corollary}
We proceed to improve the bound in Lemma \ref{lemma:Ergodic:WeakErgodic}. Let $\nu$ be the invariant measure of $(\dot{Z}_t)_{t\geq 0}$ on $\TT^d$. By Proposition \ref{prop:Ergodic:Fredholm}, for $f\not \equiv 0$ with $\nu(f)=0$,  let $\eta^f \in C^2_p(\R^d)$ be a solution of $\mathcal{L}\eta^f = f$. Consider the continuous martingale
\begin{equation}
    M_t^f:= \eta^f(Z_t) - \eta^f(Z_0) -\int_0^t \mathcal{L}\eta^f(Z_s) ds  = \sum_{i=1}^d \int_0^t \frac{\partial \eta^f}{\partial x_i}(Z_s) dB^i_s,
\end{equation}
with quadratic variation
\begin{equation}
    \langle M^f \rangle_t := \int_0^t |\nabla \eta^f(Z_s)|^2 ds.
\end{equation}
Since $f\not \equiv 0$, by the equation $\mathcal{L}\eta^f = f$, $\nabla \eta^f \not\equiv 0$. Furthermore, observe that $|\nabla \eta^f|^2 \in C^{0,\alpha}_p(\R^d)$ since $\nabla \eta^f$ is Lipschitz, and for $x\neq y$,
\begin{equation*}
    \frac{\Big||\nabla \eta^f (x)|^2 - |\nabla \eta^f(y)|^2\Big |}{|x-y|^\alpha} \leq 2d \| \nabla \eta^f  \|_{\infty}\frac{\Big||\nabla \eta^f (x)| - |\nabla \eta^f(y)|\Big |}{|x-y|^\alpha}.
\end{equation*}
Hence, by Corollary \ref{corollary:Ergodic:QuadVariation}, for every $x\in \R^d$, $\prob{x}{}{\langle M^f \rangle_t \uparrow \infty}=1$. Thus, since $M_t^f$ is a continuous martingale with $\prob{x}{}{\langle M^f \rangle_t \uparrow \infty}=1$, we can apply the Dambis-Dubins-Schwarz theorem (see, for example, \cite[Theorem V.1.6]{MR1725357}), which implies that $M_t^f$ is a time change of a Brownian motion. Rigorously, set for every $t\geq 0$,
\begin{equation}\label{eq:Ergodic:DambisDubinsTimeChange}
    T^f_t:=\inf\{s\geq 0: \langle M^f \rangle_s>t\},
\end{equation}
which is an almost-surely finite stopping time. Then, $B^f_t:=M^f_{T_t^f}$ is a $\cF_{T_t^f}$-Brownian motion and $M^f_t=B^f_{\langle M^f \rangle_t}$. Now we improve the bound in Lemma \ref{lemma:Ergodic:WeakErgodic}.
\begin{proposition}\label{prop:Ergodic:ErgodicStrong}
    Let $\nu\in C^{2,\alpha}_p(\R^d)$ be the density of the unique invariant measure of $(\dot{Z}_t)_{t\geq 0}$ on $\TT^d$ and $f\in C^{0,\alpha}_p(\R^d)$ be such that $\nu(f)=0$. Fix $\epsilon>0$. There exist positive constants $C_1,C_2$ depending on $d$, $\|\sigma\|_{0,\alpha}$ and $\|b\|_{0,\alpha}$ such that, for all $t$ sufficiently large (depending on $d$, $\|\sigma\|_{0,\alpha}$, $\|b\|_{0,\alpha}$, $\|f\|_{0,\alpha}$ and $\epsilon$),
    \begin{equation}\label{eq:Ergodic:StrongErgodic}
        \sup_{x\in \R^d} \probbb{x}{}{\sup_{\ell \in [0,t]}\bigg|\int_0^\ell f(Z_s)ds \bigg|>t\epsilon}\leq C_1 \frac{\|f \|_{0,\alpha}}{\epsilon t^{1/2}} e^{-C_2t\epsilon^2/\|f\|_{0,\alpha}^2}.
    \end{equation}
\end{proposition}

\begin{proof}
    As in the proof of Lemma \ref{lemma:Ergodic:WeakErgodic}, for $t \ge 4\|\eta^f\|_\infty/\epsilon$,
    \begin{equation*}
        \mathbb{P}_{x}{}\left[\sup_{\ell \in [0,t]}{\bigg|\int_0^\ell f(Z_s) ds\bigg|>t\epsilon}\right] \leq \probbb{x}{}{\sup_{\ell\in [0,t]}|M_\ell^f|>\frac{t \epsilon}{2}}.
    \end{equation*}
    By the time change in \eqref{eq:Ergodic:DambisDubinsTimeChange}, followed by a standard estimate for the Brownian motion, we obtain positive constants $C_1$ and $C_2$ only depending on $d$ such that
    \begin{equation}\label{eq:Ergodic:Strong2}
        \begin{split}
            \probbb{x}{}{\sup_{\ell\in [0,t]}|M_\ell^f|>\frac{t \epsilon}{2}} & = \probbb{x}{}{\sup_{\ell\in [0,t]}|B^f_{\langle M^f \rangle_\ell}|> \frac{t \epsilon}{2}}\leq \probbb{x}{}{\sup_{\ell\in [0,t]}|B^f_{\ell \| \ |\nabla \eta^f|^2\|_\infty}|>\frac{t \epsilon}{2}}\\ &\leq C_1 \frac{\|\nabla \eta^f \|_\infty}{\epsilon t^{1/2}} e^{-C_2t\epsilon^2/\|\nabla \eta^f\|_\infty^2}.
        \end{split}
    \end{equation}
    Finally, from Proposition \ref{prop:Ergodic:Fredholm}, there exists a positive constant $C=C(d,\|\sigma\|_{0,\alpha}, \|b\|_{0,\alpha})$ such that $\|\nabla \eta^f \|_\infty\leq C \|f\|_{0,\alpha}$, which gives the bound in \eqref{eq:Ergodic:StrongErgodic}.
\end{proof}

We finish this section with two applications, the first gives an asymptotic result for the periodic diffusion. For $i\in [d]$, let $b_i:\R^d \to \R$ be the $i$-th entry of the drift $b$. Let $\nu$ be the density of the invariant measure of $\dot{Z}_t$. For each $i$, let $\eta^{b_i}:\R^d\to \R$ be the a periodic solution of 
\begin{equation*}
    \mathcal{L} \eta^{b_i}(x) = b_i(x) - \nu(b_i),
\end{equation*}
which exists by Proposition \ref{prop:Ergodic:Fredholm}, and consider the vector-valued function $\eta^b := (\eta^{b_1},...,\eta^{b_d})$.

\begin{corollary}\label{corollary:Ergodic:AsymptoticDriftDiffusion}
    Fix $\epsilon>0$. There exists positive constants $C_1,C_2$ depending on $d$, $\|b\|_{0,\alpha}$ and $\|\sigma\|_{0,\alpha}$ such that, for all $t$ sufficiently large (depending on $d$, $\|b\|_{0,\alpha}$, $\|\sigma\|_{0,\alpha}$ and $\epsilon$),
    \begin{equation}
        \sup_{x\in \R^d} \mathbb{P}_{x}\bigg[\sup_{\ell \in [0,t]}\bigg| Z_\ell - Z_0 - \ell\nu(b)   \bigg|>t\epsilon\bigg]\leq C_1 t^{-1/2}e^{-C_2t\epsilon^2}.
    \end{equation}
\end{corollary}

\begin{proof}
    Notice that $Z_t - Z_0 -  t\nu(b)=  \int_0^t (b(Z_s) - \nu(b))ds + \int_0^t \sigma(Z_s)dB_s$. Hence,
    \begin{equation*}
        \begin{split}
            \mathbb{P}_{x}\bigg[\sup_{\ell \in [0,t]}\bigg| Z_\ell - Z_0 -& \ell\nu(b)  \bigg|>t\epsilon\bigg] \\ &\leq \mathbb{P}_{x}\bigg[\sup_{\ell \in [0,t]}\bigg| \int_0^\ell (b(Z_s) - \nu(b))ds \bigg|>\frac{t \epsilon}{2}\bigg] + \mathbb{P}_{x}\bigg[\sup_{\ell \in [0,t]}\bigg| \int_0^\ell \sigma(Z_s)dB_s \bigg|>\frac{t \epsilon}{2}\bigg]\\
            & \leq \sum_{i=1}^d \mathbb{P}_{x}\bigg[\sup_{\ell \in [0,t]}\bigg| \int_0^\ell (b_i(Z_s) - \nu(b_i))ds \bigg|>\frac{t \epsilon}{2d}\bigg]+ \sum_{i=1}^d\mathbb{P}_{x}\bigg[\sup_{\ell \in [0,t]}| M^i_s |>\frac{t \epsilon}{2d}\bigg],
        \end{split}
    \end{equation*}
    where we let $M^i_\ell$ be the $i$-th entry of $\int_0^\ell \sigma(Z_s)dB_s$. Next, we apply Proposition \ref{prop:Ergodic:ErgodicStrong} to bound each term of the first sum on the right-hand side of the expression above. For each $i\in [d]$, for $t \ge 8d\|\eta^{b_i}\|_\infty/\epsilon$, and for some positive constants $C_1$ and $C_2$ depending on $d$, $\|b\|_{0,\alpha}$ and $\|\sigma\|_{0,\alpha}$,
    \begin{equation*}
        \mathbb{P}_{x}\bigg[\sup_{\ell \in [0,t]}\bigg| \int_0^\ell (b_i(Z_s) - \nu(b_i))ds \bigg|>\frac{t \epsilon}{2d}\bigg] \leq  C_1 \frac{2d\|b_i \|_{0,\alpha}}{\epsilon t^{1/2}} e^{-C_2t\epsilon^2/(4d^2\| b_i\|_{0,\alpha}^2)}.
    \end{equation*}
    To bound each term of the second sum, observe that the quadratic variation of each of the martingales $M^i$ is given by
    \begin{equation*}
        \langle M^i \rangle_t = \sum_{k=1}^d \int_0^t \sigma^2_{ik}(Z_s) ds.
    \end{equation*}
    By \textbf{\ref{AssumptionA:SigmaEllipticity}}, $\sigma$ does not have a row which is identically zero, so $\sum_{k=1}^d \sigma^2_{ik} \not\equiv 0$. Since $M^i$ is a continuous martingale, and by Corollary \ref{corollary:Ergodic:QuadVariation}, we can employ the Dambis-Dubins-Schwarz theorem. Employing the time change introduced in \eqref{eq:Ergodic:DambisDubinsTimeChange}, as in \eqref{eq:Ergodic:Strong2}, we obtain different positive constants $C_1$ and $C_2$ depending on $d$, such that, for every $i\in [d]$ and $t>0$,
    \begin{equation*}
        \mathbb{P}_{x}\bigg[\sup_{\ell \in [0,t]}| M^i_s |>\frac{t \epsilon}{2d}\bigg] \leq C_1 \frac{\|\sigma\|_{\infty}}{\epsilon t^{1/2}} e^{-C_2t\epsilon^2/\|\sigma\|_{\infty}^2}.
    \end{equation*}
    Plugging these bounds into the first inequality the conclusion of the lemma follows.
\end{proof}

The second application is the following bound for the tail of the branching time $\tau_\emptyset$. This can be proved using the same argument given in \cite[Lemma 2.2]{addarioberry2025shapetheorembbmperiodic} replacing the ergodic result for the Brownian motion in the torus by Proposition \ref{prop:Ergodic:ErgodicStrong} so we omit the details.

\begin{corollary}\label{corollary:tailBranchingTime}
Let $\tau_\emptyset$ the first branching time of the periodic branching diffusion. There exist positive constants $\Gamma$ and $\Theta$ depending on $\sigma$, $b$ and $g$, such that, for every $t\geq 0$,
    \begin{equation*}
        \Theta^{-1}e^{-t\Gamma}\leq \inf_{x\in \TT^d}\proba{x}{}{\tau_\emptyset>t}\leq \sup_{x\in \TT^d}\proba{x}{}{\tau_\emptyset>t} \leq \Theta e^{-t\Gamma}.
    \end{equation*}
\end{corollary}

\subsection{The principal eigenvalue}\label{sec:EigenvalueProblem}

In this section we list properties of the principal eigenvalue $\gamma(\zeta)$ of \eqref{eq:Intro:EigenvalueProblem}, its associated eigenfunction $\psi(\cdot; \zeta)$, and of $\css(e)$ defined in \eqref{eq:Intro:SpeedRelativeTo} which will be used through this work. The following proposition justifies the definition of the principal eigenvalue given in \eqref{eq:Intro:EigenvalueProblem}.

\begin{proposition}\cite[Theorem 4.11.1, parts (vi) and (viii)]{MR1326606}\label{prop:PrincipalEigenEx}
    Let $R\in C^{0,\alpha}_p(\R^d)$. Fix $\zeta\in \R^d$, and consider the eigenvalue problem \eqref{eq:Intro:EigenvalueProblem}. Then, the principal eigenvalue $\gamma(\zeta)$ exists, that is, there is a maximal real number $\gamma(\zeta)$, for which we can find a positive principal eigenfunction $\psi(\cdot ;\zeta)\in C^2_p(\R^d)$ satisfying \eqref{eq:Intro:EigenvalueProblem}. Moreover, $\psi(\cdot ;\zeta)\in C_p^{2,\alpha}(\R^d)$ and the principal eigenfunction is unique up to rescaling by a multiplicative constant. 
\end{proposition}

\begin{remark}
    The argument given in \cite[Theorem 4.11.1]{MR1326606} can be generalized to show that, if $R\in C^{k,\alpha}_p(\R^d)$ for a given integer $k\geq 0$, $\psi(\cdot ;\zeta)\in C_p^{k+2,\alpha}(\R^d)$.
\end{remark}

For fixed $\zeta \in \R^d$, fix a positive principal  eigenfunction $\psi(\cdot ;\zeta)$ as in the proposition above. Since $\psi(\cdot ;\zeta)$ is $\Z^d$-periodic and twice differentiable, it achieves its minimum and maximum. In particular, we can rescale in such a way that
    \begin{equation}\label{eq:BoundEigenfunction}
        \int_{\TT^d}\psi(x;\zeta)dx =1 , \ \ \inf_{x\in \R^d}\psi(x ;\zeta) > 0 \ \mathrm{and} \ \sup_{x\in \R^d}\psi(x;\zeta) < +\infty.  
    \end{equation}
    Furthermore, for the same strictly positive eigenfunction $\psi(\cdot\ ;\zeta)$, 
    \begin{equation}\label{eq:BoundEigenfunctionQuotient}
        \inf_{x,y\in \R^d}\frac{\psi(x ;\zeta)}{\psi(y ;\zeta)} \geq \frac{\inf_{x\in \R^d}\psi(x ;\zeta)}{\sup_{y\in \R^d}\psi(y ;\zeta)} > 0 \ \ \mathrm{and} \ \sup_{x,y\in \R^d}\frac{\psi(x ;\zeta)}{\psi(y ;\zeta)} \leq \frac{\sup_{x\in \R^d}\psi(x ;\zeta)}{\inf_{y\in \R^d}\psi(y ;\zeta)} < +\infty.  
    \end{equation}

Now, we recall the following properties of $\gamma$. The proofs  essentially follows the same logic as the one given in \cite[Proposition 2.6]{addarioberry2025shapetheorembbmperiodic}, and we provide the necessary modifications.

\begin{proposition}\label{Prop:PropertiesGamma}
    Fix $\alpha \in (0,1)$ and let the effective branching rate $R\in C^{0,\alpha}_p(\R^d)$ defined in \eqref{eq:Intro:EffectiveBranching} be such that $R \ge 0$ and $R\not\equiv 0$. Then $\gamma(\zeta)=\gamma(\zeta;b,\sigma,R)$ satisfies the following properties.
    \begin{enumerate}[(i)]
        \item The function $\gs$ belongs to $C^2(\R^d)$, is strictly convex, and for $\zeta\in \R^d$,
        \begin{equation}\label{eq:QuadraticBoundsGamma}
             \gamma(\zeta) \in \left[\Big(\min_{x\in \R^d}R(x)\Big) + \frac{\theta^{-1} |\zeta|^2}{2} - |\zeta| \| b\|_\infty ,  \Big(\max_{x\in \R^d}R(x)\Big) + \frac{\theta |\zeta|^2}{2} + |\zeta| \| b\|_\infty  \right],
        \end{equation}
        where we recall that $\theta$ is a positive constant such that $0<\theta^{-1} \mathrm{Id}<\sigma \sigma^T<\theta \mathrm{Id}$ (see \textbf{\ref{AssumptionA:SigmaEllipticity}}).
        \item Let $\nu(\cdot;0):\TT^d\to \R$ be the density on the torus of the invariant measure of a periodic diffusion with generator $\mathcal{L}^0$ as defined in \eqref{eq:Intro:GeneratorExponential}. Recall that $\nu(R;0)$ denotes the integral $\int_{\TT^d}R(x)\nu(x;0)dx$. We have that $\gamma(0)\geq \nu(R;0)>0$, and for $e\in S^{d-1}$,
        \begin{equation}\label{eq:limGamma}
            \begin{split}
                &\lim_{\lambda \rightarrow 0^+ }\frac{\gs(\lambda e)}{\lambda}=\lim_{\lambda \rightarrow + \infty }\frac{\gs(\lambda e)}{\lambda}=+\infty, \ \mathrm{ and } \ \lim_{\lambda \rightarrow 0^- }\frac{\gs(\lambda e)}{\lambda}= \lim_{\lambda \rightarrow - \infty }\frac{\gs(\lambda e)}{\lambda}=-\infty.
            \end{split}
        \end{equation}
        \item For all $e\in S^{d-1}$, there exists a unique $\lambda_e\in (0,\infty)$ such that 
        \begin{equation*}
            \inf_{\lambda>0} \frac{\gs(\lambda e)}{\lambda}=\frac{\gs(\lambda_e e)}{\lambda_e}.
        \end{equation*}
        \item At every $\zeta\in \R^d$, the Hessian $\nabla^2 \gamma (\zeta)$ is invertible.
    \end{enumerate}
\end{proposition}

\begin{proof}
        \textit{(i)} The fact that $\gamma(\cdot)$ belongs to $C^2(\R^d)$ and is strictly convex is a direct consequence of \cite[Theorem 8.2.10 (i) (iii)]{MR1326606}. For \eqref{eq:QuadraticBoundsGamma}, fix $\zeta\in \R^d$. Let $Z^\zeta_t$ be a periodic diffusion with generator $\mathcal{L}^\zeta$ as in \eqref{eq:Intro:GeneratorExponential}. Let $\gamma(\zeta)$ be the principal eigenvalue of \eqref{eq:Intro:EigenvalueProblem}, and $\psi(\cdot\ ;\zeta)$ the strictly positive principal eigenfunction satisfying \eqref{eq:BoundEigenfunction}. By the Feynman-Kac formula, for every $x\in \R^d$ and $t>0$,
        \begin{equation}\label{eq:FeynmanKacApp1}
            \psi(x;\zeta) = \mathbb{E}_x\left[ \psi(Z^\zeta_t;\zeta) e^{\int_0^t K^\zeta(Z_s^\zeta) ds - \gamma(\zeta)t} \right],
        \end{equation}
        where $K^\zeta$ is defined in \eqref{eq:Intro:GeneratorExponential}. From the definition of $K^\zeta$, for every $y\in \R^d$,
        \begin{equation*}
            \frac{1}{2}\theta^{-1}|\zeta|^2 - |\zeta|\|b\|_\infty + \Big(\min_{x\in \R^d}R(x)\Big) \leq K^\zeta(y) \leq \frac{1}{2}\theta|\zeta|^2 + |\zeta|\|b\|_\infty + \Big(\max_{x\in \R^d}R(x)\Big).
        \end{equation*}
        With this at hand, the bounds for $\gamma$ are obtained by a contradiction argument as in the proof of \cite[Proposition 2.6(i)]{addarioberry2025shapetheorembbmperiodic} using Lemma \ref{lemma:Ergodic:WeakErgodic}, and we omit the details.
        
\textit{(ii)} Let $Z^0_t$ be a periodic diffusion with generator $\mathcal{L}^0$ as in \eqref{eq:Intro:GeneratorExponential} and $\dot{Z}^0_t$ its projection into $\TT^d$. By Proposition \ref{prop:Ergodic:InvariantMeasure}, there exists a unique invariant measure for $\dot{Z}^0_t$ with a strictly positive density $\nu(\cdot;0)\in C^2_p(\R^d)$. By Proposition~\ref{prop:PrincipalEigenEx}, there exists a strictly positive function $\psi(\cdot\ ;0)\in C^{2,\alpha}_p(\R^d)$ such that $\mathcal{L}^0 \psi + R \psi = \gamma(0)\psi$, and by the Feynman-Kac formula,
    \begin{equation*}
        \psi(x;0)= \expecdd{x}{}{\psi(Z_t^0;0)e^{\int_0^t R(Z_s^0)ds -t\gamma(0) }}.
    \end{equation*}
The proof that $\gamma(0)\geq \nu(R,0)>0$ then follows via a contradiction argument, employing Proposition \ref{prop:Ergodic:ErgodicStrong} (this is exactly the same contradiction argument given in \cite[Proposition 2.6(ii)]{addarioberry2025shapetheorembbmperiodic}). The limits in \eqref{eq:limGamma} follow directly from the fact that $\gamma(0)\geq \nu(R;0)>0$ and from \eqref{eq:QuadraticBoundsGamma}.

\textit{(iii)} This follows from the fact that $\gamma(0)\geq \nu(R;0)>0$, the strict convexity of $\gamma$, and by \eqref{eq:QuadraticBoundsGamma}.
        
    \textit{(iv)} To see that the Hessian $\nabla^2 \gamma$ of $\gamma$ is invertible, let us fix $\zeta\in \R^d$ and $\lambda\neq 0\in \R^d$. By \cite[Theorem 8.2.10 (iii)]{MR1326606}, $\lambda ^T \nabla^2 \gamma(\zeta)  \lambda>0$. Since $\lambda\neq 0\in \R^d$ is arbitrary, it follows that $\nabla^2 \gamma(\zeta)$ is invertible.
     \end{proof}

\begin{remark}\label{r:derivative_cs}
    Proposition \ref{Prop:PropertiesGamma}(i) justifies the definition \eqref{eq:Intro:EffectiveDrift} of the effective drift. Given the strict convexity and differentiability of $\gamma$, Proposition \ref{Prop:PropertiesGamma}(iii) implies that \begin{equation}\label{eq:EigenvalueProblem:DirectionalDerivative}
         \tfrac{\gamma(\lambda_e e)}{\lambda_e} = \nabla \gamma(\lambda_e e) \cdot e = \partial_e \gamma(\lambda_e e) .
    \end{equation}
    Furthermore, by the strict convexity of $\gamma$, for every $e\in S^{d-1}$, $\css(e)= \nabla \gamma(\lambda_e e) \cdot e -  \nabla \gamma(0)\cdot e >0$, which justifies \eqref{eq:Intro:SpeedRelativeTo}.
\end{remark}

The function $\gamma$ is in fact analytic. The proof of the following proposition is provided in Appendix \ref{append:PrincipalEigenvalue}.

\begin{proposition}\label{prop:Eigenvalue:Analytic}
    The function $\gamma$ is analytic.
\end{proposition}

In light of Proposition \ref{Prop:PropertiesGamma}, for the effective branching rate $R$, we can define a function $\cs:S^{d-1}\rightarrow \R$, by setting, for every $e\in S^{d-1}$, 
\begin{equation}\label{eq:prelimCriticalSpeed}
    \cs(e)=\cs(e;b,\sigma,R):= \min_{\lambda>0} \frac{\gs(\lambda e)}{\lambda}=\frac{\gs(\lambda_e e)}{\lambda_e}.
\end{equation}

\begin{lemma}\label{lemma:DifferentiabilityLambda_e}
    Let $\Lambda:=\{\lambda_e e: e\in S^{d-1}\} \subset \R^d $. Then $\Lambda$ is of class $C^1$, meaning that locally, up to a change of coordinates,  $\Lambda$ is the graph of a $C^1$ function on some open subset of $\R^{d-1}$. In particular $\inf_{e\in S^{d-1}}\lambda_e >0 $ and $\sup_{e\in S^{d-1}}\lambda_e<+\infty$. Furthermore, $\cs,\css:S^{d-1}\to \R$ are continuous, and, $\inf_{e\in S^{d-1}}\css(e) >0 $ and $\sup_{e\in S^{d-1}}\{|\cs(e)|,|\css(e)|\}<+\infty$.
\end{lemma}
\begin{proof}
    Consider the function $F:\R^d \to \R$ defined by $F(\zeta)=\nabla \gamma(\zeta)\cdot \zeta-\gamma(\zeta)$. Since $\gamma \in C^{2}(\R^d)$, $F\in C^1(\R^d)$. We claim that $\Lambda = \{\zeta\in \R^d: F(\zeta)=0\}$. By \eqref{eq:EigenvalueProblem:DirectionalDerivative}, for every $e\in S^{d-1}$, $F(\lambda_e e) = \nabla \gamma(\lambda_e e) \cdot (\lambda_e e) - \gamma(\lambda_e e)=0$, so $\Lambda \subset \{\zeta\in \R^d: F(\zeta)=0\}$. For fixed $e\in S^{d-1}$, by Proposition \ref{Prop:PropertiesGamma}(iv),
    \begin{equation*}
        \frac{\partial }{\partial \lambda} F(\lambda e) =  \frac{\partial }{\partial \lambda} (\nabla \gamma(\lambda e)\cdot (\lambda e)-\gamma(\lambda e)) = \nabla \gamma(\lambda e) \cdot e + \lambda e^T \nabla^2 \gamma(\lambda e) e - \nabla \gamma(\lambda e) \cdot e>0.
    \end{equation*}
    It follows that there is at most one $\lambda>0$ such that $F(\lambda e)=0$. Since $\lambda_e>0$ satisfies $F(\lambda_e e)=0$, we conclude that $\Lambda \supset \{\zeta\in \R^d: F(\zeta)=0\}$, and thus $\Lambda = \{\zeta\in \R^d: F(\zeta)=0\}$.

    With this at hand, we apply the implicit function theorem. We need to verify that for every $\zeta \in \Lambda$, $\nabla F(\zeta)\neq 0$. By Proposition \ref{Prop:PropertiesGamma}(iv), for $\zeta\neq 0$, $\nabla F(\zeta) = \nabla \gamma(\zeta) + \nabla^2 \gamma(\zeta) \zeta - \nabla \gamma(\zeta) \neq 0$. Since $\lambda_e>0$, $\nabla F(\lambda_e e)\neq 0$, and we conclude that $\Lambda$ is of class $C^1$ by applying the implicit function theorem to the zero-level set of $F$.

    From this, it follows that the map $e\mapsto \lambda_e$ is continuous, and since $\lambda_e>0$ for every $e\in S^{d-1}$, it follows that $\lambda_e$ is bounded away from zero. The continuity of $\cs$ and $\css$ follow directly from \eqref{eq:prelimCriticalSpeed} and \eqref{eq:Intro:SpeedRelativeTo}, as $\cs:S^{d-1}\to \R$ is the quotient of two continuous functions whose denominator does not vanish.
\end{proof}

\begin{remark}
    The argument above shows that, if $\gamma \in C^k(\R^d)$ for some $k \geq 2$, $\Lambda$ is of class $C^{k-1}$.
\end{remark}

\subsection{Many-to-one lemma}\label{sec:ManyToOne}

We introduce the many-to-one lemma corresponding to the periodic branching diffusion. We will express additive functionals of the periodic branching diffusion in terms of the functional of one trajectory of another periodic diffusion $(Y_t)_{t\geq 0}$ which we proceed to introduce.

Let us fix $\zeta \in \R^d$, and Let $\phi:\R^d\to \R$ be defined by 
\begin{equation}\label{eq:ManyToOne:PhiPotential}
    \phi(x)=\phi(x;\zeta):=\log \psi(x;\zeta) + x\cdot \zeta,
\end{equation}
where $\psi(\cdot ;\zeta)$ is the strictly positive principal eigenfunction introduced in Proposition \ref{prop:PrincipalEigenEx}. Observe that the gradient of $\phi$ satisfies
\begin{equation}\label{eq:GradPhidef}
   \nabla\phi(x)=\nabla\phi(x;\zeta)=(\partial_1\phi(x),...,\partial_d \phi(x))= \frac{\nabla \psi(x;\zeta)}{\psi(x;\zeta)} + \zeta,
\end{equation}
where by Proposition \ref{prop:PrincipalEigenEx}, each of the entries of $\nabla \phi$ belongs to $C^{1,\alpha}_p(\R^d)$.

Consider the formal stochastic differential equation with $\Z^d$-periodic coefficients
\begin{equation}\label{eq:SDETiling2.0}
    dY_t = (b(Y_t)+A(Y_t)\nabla\phi(Y_t;\zeta))dt + \sigma(Y_t)dW_t,
\end{equation}
where $(W_t)_{t\geq 0}$ denotes a $d$-dimensional Brownian motion. For every $x\in \R^d$, and any weak solution $(Y_t)_{t\geq 0}$ of \eqref{eq:SDETiling2.0} with $Y_0=x$, the generator $\mathcal{K}^\zeta$ of $(Y_t)_{t\geq 0}$ is defined for every function $q\in C^2(\R^d)$ by
\begin{equation}\label{eq:GeneratorTilting}
    \mathcal{K}^\zeta q(y)= \frac{1}{2}\sum_{i=1}^d \sum_{j=1}^d a_{ij}(y)\partial_{ij}q(y) + (b(y)+A(y)\nabla \phi(y;\zeta))\cdot \nabla q(y), \ \forall y\in \R^d.
\end{equation}

Consider an abstract filtered probability space $(\hat{\Omega}, \hat{\cF}, \{\hat{\cF_t}\}_{t\geq 0}, \bP^{\zeta})$, with $ \{\hat{\cF_t}\}_{t\geq 0}$ right-continuous and complete. Suppose that for every $x\in \R^d$ we are given a measure $\mathbb{P}^\zeta_x$ on $(\hat{\Omega}, \hat{\cF})$  and a $\{\hat{\cF_t}\}_{t\geq 0}$-adapted continuous process $(Y_t)_{t\geq 0}$ such that under the measure $\mathbb{P}^\zeta_x$, $(Y_t)_{t\geq 0}$ is a weak solution of \eqref{eq:SDETiling2.0} with $\mathbb{P}^\zeta_x[Y_0=x]=1$. Let us denote by $\mathbb{E}^\zeta_x$ the expectation operator associated to $\mathbb{P}^\zeta_x$.

The following lemma is the main result of this section. This many-to-one lemma is used in \cite{MR4162842} to obtain exact asymptotics of the transition kernel of the periodic branching diffusion. A one-dimensional version can be found in \cite[Lemma 2.3]{MR4492971}.

\begin{lemma}\label{lemma:ManyToOneLemma}
    Let $(Y_t)_{t\geq 0}$ be the weak solution introduced above, let $\cF^Y_t=\sigma(Y_s; s\in [0,t])$, and $F:C([0,t])\rightarrow [0,\infty)$ be a given function. Recall that for $\zeta \in \R^d$, $\gamma(\zeta)$ denotes the principal eigenvalue of the eigenvalue problem \eqref{eq:Intro:EigenvalueProblem}, and $\psi(\cdot;\zeta)$ its associated eigenfunction. Recall that $X_s(v)$ denotes the position of the particle $v$ at time $s$. Then, if $F((Y_s)_{s \in [0,t]})$ is $\cF^Y_t$-measurable,
    \begin{equation}\label{eq:manyTO}
            \begin{split}
                \mathbf{E}_{x}\left[ \sum_{v\in N_t} F\big((X_s(v))_{s \in [0,t]}\big)   \right] = \expecddd{x}{\zeta}{\frac{\psi(x;\zeta)}{\psi(Y_t;\zeta)}e^{-\zeta\cdot (Y_t-Y_0)+t\gs(\zeta) }F((Y_s)_{s\in [0,t]})}.
            \end{split}
    \end{equation}
\end{lemma}

This result is a consequence of the general many-to-one lemma proven in \cite[Section 4.1]{MR3606740}. Let $C[0,t]^d$ denote the space of continuous functions $f:[0,t]\to \R^d$. In general, to prove a many-to-one lemma as in \cite[Section 4.1]{MR3606740}, we need two elements. First, a functional $\mathcal{E}:(C[0,T])^d \times [0,T]\to \R$ such that $t\mapsto \mathcal{E}( (f_s)_{s\in [0,t]} ,t)$ is a non-negative martingale with $\mathcal{E}( f_0 ,0)=1$ when $(f_s)_{s\geq 0}$ is a weak solution of \eqref{eq:Intro:Diffusion}. Second, we need to characterize the dynamics of $Z_t$ under the change of measure with respect to the aforementioned martingale. Let us start by introducing the functional $\mathcal{E}$.

For a given $T>0$ and $\zeta \in \R^d$, let $\mathcal{E}^\zeta:(C[0,T))^d\times [0,T) \to [0,T)$ be the functional defined for $(f_s)_{s\in [0,T]}\in (C[0,T])^d$ by
\begin{equation}\label{eq:ManyToOne:ExponentialMartingale2.0}
    \mathcal{E}^{\zeta}((f_s)_{s\in [0,t]},t) :=\exp  \left(\log\frac{\psi(f_t,\zeta)}{\psi(f_0,\zeta)} + \zeta\cdot(f_t-f_0) - t \gs(\zeta) + \int_0^t R(f_s) \ ds\right).
\end{equation}
Now we show that when $f_s$ is a weak solution of \eqref{eq:Intro:Diffusion}, \eqref{eq:ManyToOne:ExponentialMartingale2.0} is a martingale. Recall that we considered a general weak solution $Z_t$ of \eqref{eq:Intro:Diffusion} defined on a filtered probability space $(\Omega,\mathcal{F},\{\cF_t\}_{t \ge 0},\bP)$, where the filtration $\cF_t$ is right-continuous and complete. For $x\in \R^d$, we denoted by $\bP_x$ the measure on $(\Omega,\cF)$ under which $Z_t$ is a weak solution of \eqref{eq:Intro:Diffusion} with $Z_0 =x$, and by $\bE_x$ the expectation operator associated to $\bP_x$.

We will see that the functional $\mathcal{E}( (Z_s)_{s\in [0,t]} ,t)$ is the exponential martingale associated to another martingale. Recall $\mathcal{L}$, the generator of $Z_t$ introduced in \eqref{eq:Prere:Generator}. First, consider the mean-zero real-valued martingale,
\begin{equation}\label{eq:ManyToOne:Martingale}
    \begin{split}
        M^\zeta_t  :&= \phi(Z_t) - \phi(Z_0) - \int_0^t \mathcal{L}\phi(Z_s)ds = \int_0^t \nabla \phi(Z_s)^T \sigma(Z_s) dB_s,
    \end{split}
\end{equation}
where the second identity follows from a direct application of It\^o's formula to $\phi(Z_t)$.
This martingale has quadratic variation $\langle M^\zeta \rangle_t$ satisfying
\begin{equation}\label{eq:ManyToOne:QuadraticVariation}
    \langle M^\zeta \rangle_t = \int_0^t \nabla \phi(Z_s;\zeta)^T A(Z_s) \ \nabla \phi(Z_s;\zeta) ds \leq t \theta \sup_{x\in \R^d}|\nabla \phi(x;\zeta)|^2.
\end{equation}
\begin{lemma}\label{lemma:ManyToOne:ExponentialMartingale2}
    Let $(M^\zeta_t)_{t\geq 0}$ be the martingale introduced above, and $\mathbb{P}_x$ the underlying probability measure for $(Z_t)_{t\geq 0}$ in \eqref{eq:Intro:Diffusion}. Then, $\mathbb{P}_x$-almost surely,
    \begin{equation}\label{eq:claim}
        M^\zeta_t-\frac{1}{2} \langle M^\zeta \rangle_t= \log\frac{\psi(Z_t;\zeta)}{\psi(Z_0;\zeta)} + \zeta\cdot(Z_t-Z_0) - t \gs(\zeta) + \int_0^t R(Z_s) \ ds.
    \end{equation}
\end{lemma}
\begin{proof}
    The proof follows the exact same logic as the proof of \cite[Lemma 2.13]{addarioberry2025shapetheorembbmperiodic}, applying It\^o's formula to $\log \tfrac{\psi(Z_t,\zeta)}{\psi(Z_0,\zeta)}$, combining with \eqref{eq:ManyToOne:Martingale} and \eqref{eq:ManyToOne:QuadraticVariation}, and using the fact that $\psi(\cdot;\zeta)$ is the eigenfunction associated to the principal eigenfunction $\gamma(\zeta)$ of the eigenvalue problem \eqref{eq:Intro:EigenvalueProblem}.
\end{proof}

Since $\nabla\phi$ is bounded, by \eqref{eq:ManyToOne:QuadraticVariation} we have that for every $t>0$, $\mathbb{E}_x[e^{\langle M^\zeta \rangle_t /2}]<\infty$. Thus, by Novikov's theorem (see \cite[Corollary 3.5.13]{MR1121940}), $\exp\parentesiss{M^\zeta_t - \frac{1}{2}\langle M^\zeta \rangle_t }$ is a non-negative martingale with $\exp\parentesiss{M^\zeta_0 - \frac{1}{2}\langle M^\zeta \rangle_0 } = 1$. Hence, by the lemma above and \eqref{eq:ManyToOne:ExponentialMartingale2.0}, $(\mathcal{E}^{\zeta}((Z_s)_{s\in [0,t]},t))_{t\geq 0}$ is a non-negative martingale with $\mathcal{E}^{\zeta}(Z_0,0)=1$. This is the martingale we will employ to apply the general many-to-one lemma in \cite[Section 4.1]{MR3606740}.

In addition to $(\mathcal{E}^{\zeta}((Z_s)_{s\in [0,t]},t))_{t\geq 0}$, we also characterize the dynamics of $Z_t$ under the Girsanov change of measure with respect to $(\mathcal{E}^{\zeta}((Z_s)_{s\in [0,t]},t))_{t\geq 0}$. Consider the completion of the natural filtration of $(Z_t)_{t\geq 0}$. Let us denote this filtration by $(\cF^Z_{t})_{t\geq 0}$, so $\cF^Z_t$ is the completion of $\sigma(Z_s;s\in [0,t])$. Since $\nabla\phi$ is continuous, one can show that there exists a measure $\bQ^\zeta_x$ on $\cF_\infty^Z:=\sigma(Z_s;s\geq 0)$ such that, for every $t\geq 0$ and $x\in \R^d$, the Radon-Nikodym derivative of the restriction of $\bQ^\zeta_x$ and $\bP_x$ to $\cF^Z_t$ satisfies
\begin{equation}\label{eq:measuretilting}
    \frac{d\bQ^\zeta_{x}}{d\bP_{x}}\Big |_{\cF^Z_t}=\mathcal{E}^{\zeta}((Z_s)_{s\in [0,t]},t)
\end{equation}
(see, for example, \cite[eq. (5.7) in the discussion before Corollary 3.5.2]{MR1121940}). To avoid introducing unnecessary notation, we will make a slight abuse of notation using the same symbol $\bQ^\zeta_{x}$ to denote the expectation operator associated to the measure $\bQ^\zeta_{x}$.

\begin{lemma}\label{lemma:ManyToOne:CarreDuChamp}
    Let $\mathbb{Q}^\zeta_x$ be the measure defined above. Under $\mathbb{Q}^\zeta_x$, $(Z_t)_{t\geq 0}$ is a weak solution of \eqref{eq:SDETiling2.0} with $Z_0=x$.
\end{lemma}

\begin{proof}
    By \cite[Proposition VIII.3.1]{MR1725357}), $(Z_t)_{t\geq 0}$ is the weak solution of an SDE under $\mathbb{Q}^\zeta_x$. Now we determine the coefficients of the corresponding SDE. To do so, we apply \cite[Proposition VIII.3.4]{MR1725357}, which states that the generator of $(Z_t)_{t\geq 0}$ under $\bQ^\zeta$ is defined for every function $f\in C^2(\R^d)$ by $ \mathcal{L}f + \Gamma(\phi,f)$, where $\Gamma$ is the \textit{op\'erateur carr\'e du champ} defined on $C^2(\R^d)\times C^2(\R^d)$ by
\begin{equation*}
    \Gamma(f,h)= \mathcal{L}(fh)-f\mathcal{L}h -h\mathcal{L}f. 
\end{equation*}
For $f\in C^2(\R^d)$, by the product rule and the symmetry of $A$,
\begin{equation*}
    \begin{split}
        \Gamma(\phi, f) & = \frac{1}{2}\sum_{i=1}^d \sum_{j=1}^d a_{ij}\partial_{ij}(\phi f) +  b\cdot ( f\nabla \phi  + \phi\nabla f ) \\
        &\ \ \ \ \ \ \ \ \ \ -  \frac{f}{2}\sum_{i=1}^d \sum_{j=1}^d a_{ij}\partial_{ij}\phi - f b\cdot \nabla \phi  -  \frac{\phi}{2}\sum_{i=1}^d \sum_{j=1}^d a_{ij}\partial_{ij}f - \phi b\cdot \nabla f\\
        & = \frac{1}{2}\sum_{i=1}^d \sum_{j=1}^d a_{ij}(\partial_i \phi \partial_j f + \partial_i f \partial_j \phi )\\
        & =(A \nabla \phi) \cdot \nabla f        
    \end{split}
\end{equation*}
Thus, under $\mathbb{Q}^\zeta_x$, $(Z_t)_{t\geq 0}$ is a weak solution of \eqref{eq:SDETiling2.0} with $Z_0=x$. 
\end{proof}

\begin{remark}
    The proofs above did not use any particular information about the weak solution $Z_t$ of \eqref{eq:Intro:Diffusion} and they hold for any weak solution of \eqref{eq:Intro:Diffusion}.
\end{remark}

With this at hand, the proof of Lemma \ref{lemma:ManyToOneLemma} follows.

\begin{proof}[\textbf{Proof of Lemma \ref{lemma:ManyToOneLemma}}]
    Recall the effective branching rate $R$ in \eqref{eq:Intro:EffectiveBranching}. Since by Lemma \ref{lemma:ManyToOne:ExponentialMartingale2}, $(\mathcal{E}^{\zeta}((Z_s)_{s\in [0,t]},t))_{t\geq 0}$ is a non-negative martingale with $\mathcal{E}^{\zeta}(Z_0,0)=1$ under $\mathbb{P}_x$, by the many-to-one lemma in \cite[Section 4.1]{MR3606740} and plugging in the definition \eqref{eq:ManyToOne:ExponentialMartingale2.0} of $\mathcal{E}^\zeta$, we get 
    \begin{equation*}
        \begin{split}
            \mathbf{E}_{x}\left[ \sum_{v\in N_t} F\big((X_s(v))_{s \in [0,t]}\big)   \right] &= \mathbb{Q}^\zeta_x\left [\frac{1}{\mathcal{E}^\zeta((Z_s)_{s\geq 0},t)}e^{\int_{0}^t R(Z_s) ds}F\big((Z_s)_{s \in [0,t]}\big)\right]\\
            &= \mathbb{Q}^\zeta_x\left [\frac{\psi(x;\zeta)}{\psi(\xi_t;\zeta)}e^{-\zeta\cdot (Z_t-Z_0)+t\gs(\zeta) }F\big((Z_s)_{s \in [0,t]}\big)\right],
        \end{split}
    \end{equation*}
    where we employ the abuse of notation for $\bQ^\zeta_x$ mentioned after \eqref{eq:measuretilting}. By Lemma \ref{lemma:ManyToOne:CarreDuChamp},  under the measure $\mathbb{Q}^\zeta_x$, $(Z_s)_{s\geq 0}$ is a weak solution of \eqref{eq:SDETiling2.0}. Since the right-hand side of the previous expression is $\mathcal{F}^{Z}_t$-measurable, by uniqueness of weak solutions of \eqref{eq:SDETiling2.0}, the right side of the equation above does not depend on the particular measure under which we are taking the weak solution of \eqref{eq:SDETiling2.0}, so
    \begin{equation*}
        \begin{split}
            \mathbf{E}_{x}\left[ \sum_{v\in N_t} F\big((X_s(v))_{s \in [0,t]}\big)   \right] = \mathbb{E}^\zeta_x\left [\frac{\psi(x;\zeta)}{\psi(Y_t;\zeta)}e^{-\zeta\cdot (Y_t-Y_0)+t\gs(\zeta) }F\big((Y_s)_{s \in [0,t]}\big)\right],
        \end{split}
    \end{equation*}
    which finishes the proof.
\end{proof}

\subsection{Interpolation result}\label{sec:interpolation}
In this subsection, we present a general ``interpolation lemma'' which controls the probability of the periodic branching diffusion making ballistic displacements on intervals of order $O(1)$, after correcting by the effective drift.

For $E\subset \R^d$ and $\epsilon>0$, let $E_\epsilon:=\cup_{x\in E}B(x,\epsilon)$. Given two sets $E,D\subset \R^d$, recall that the {\em Hausdorff distance} between $E$ and $D$ is given by
\begin{equation*}
    d_\mathrm{H}(E,D):=\inf\{ \epsilon>0 :\ E\subset D_\epsilon \text{ and } D\subset E_\epsilon \}.
\end{equation*}
The following is the main result of this section, we will later refer to it as the interpolation lemma. 

\begin{lemma}\label{lemma:interpolation}
    For $t \geq 0$, let $\mathcal{X}_t:=\{X_t(v):v\in \cN_t\}$ and $\vb$ the effective drift. Fix $T>0$ and $\kappa>0$. Consider the event
    \begin{equation*}
        A_n:=\Bigl\{\sup_{\ell \in [nT,(n+1)T]} d_\mathrm{H}(\mathcal{X}_{nT}- nT \vb,\mathcal{X}_\ell-\ell\vb)> \kappa nT \Bigl \}, 
    \end{equation*}
    then, $\sum_{n=1}^\infty \left(\sup_{x\in \R^d} \proba{x}{}{A_n}\right) < + \infty$.    
\end{lemma}

Before proceeding to the proof of this lemma, we prove the following known probabilistic characterization of the effective drift $\vb$. 

\begin{lemma}\cite[Theorem 8.2.10 (ii)]{MR1326606}\label{lemma:Interpolation:ProbaEffectiveDrift}
    Let $Y^0_t$ be a weak solution of \eqref{eq:SDETiling2.0} taking $\zeta=0$. Let $\dot{Y}^0_t$ be its projection into $\TT^d$ introduced in Proposition \ref{prop:Ergodic:InvariantMeasure}. Let $\nu(\cdot;0)$ be the density of the invariant measure of $\dot{Y}^0_t$ on $\TT^d$. Let $\vb$ be the effective drift as defined in \eqref{eq:Intro:EffectiveDrift} and $\psi(\cdot;0)$ the principal eigenfunction associated to the principal eigenvalue $\gamma(0)$ of \eqref{eq:Intro:EigenvalueProblem}. Then,
    \begin{equation}
        \vb  = \int_{\TT^d}\Big(b(x) + A(x)\frac{\nabla \psi(x;0)}{\psi(x;0)} \Big)\nu(x;0)dx.
    \end{equation}
\end{lemma}
\begin{proof}
    Since $0$ is fixed in this proof, let us set $\psi(x):=\psi(x;0)$. The result \cite[Theorem 8.2.10 (ii)]{MR1326606} applied to the operator $\mathcal{L}$ \eqref{eq:Generator:DivergenceForm} in divergence-form gives
    \begin{equation*}
        \vb = \int_{\TT^d}\Big(b(x) + A(x)\frac{\nabla \psi(x)}{\psi(x)} \Big)\psi(x)\Tilde{\psi}(x) dx,
    \end{equation*}
    where  $\Tilde{\psi}(\cdot )>0$ satisfies the adjoint equation $(\mathcal{L}^0 + K^0)^*\Tilde{\psi} = \gamma(0)\Tilde{\psi}$ for $\mathcal{L}^0$ and $K^0$ as defined in \eqref{eq:Intro:GeneratorExponential}, and the $\Z^d$-periodic functions $\psi$ and $\Tilde{\psi}$ are normalized so that $\int_{\TT^d}\psi(x)\Tilde{\psi}(x)dx =1$.
    What remains is to show that $\psi\Tilde{\psi}$ is the density of the invariant measure of $(\dot{Y}_t^0)_{t\geq 0}$ on the torus. Let us prove this fact using Proposition \ref{prop:Ergodic:InvariantMeasure}, which characterizes $\nu(\cdot;0)$ as a periodic solution of the adjoint equation $(\mathcal{K}^0)^* \nu(\cdot ; 0) = 0$, where $\mathcal{K}^0$ was introduced in \eqref{eq:GeneratorTilting}. Recall $\kappa(x) = b(x) - \frac{1}{2}\sum_{j=1}^d \partial_j a_j(x)$ introduced in \eqref{eq:Ergodic:DivDrift}. Since $\phi(y;0) = A(y)\tfrac{\nabla \psi(y)}{\psi(y)}$, expressed in divergence-form, $\mathcal{K}^0$ takes the form
    \begin{equation}
        \begin{split}
            (\mathcal{K}^0)q(y) &= \frac{1}{2}\mathrm{div}\big(A(y)\nabla q(y)\big) + \Big( \kappa(y)+A(y)\frac{\nabla \psi(y)}{\psi(y)}\Big)\cdot \nabla q(y).
        \end{split}
    \end{equation}
    Thus, applying the adjoint of $\mathcal{K}^0$ to the product $\psi \Tilde{\psi}$ we obtain
    \begin{equation*}
        \begin{split}
            (\mathcal{K}^0)^*(\psi\Tilde{\psi}) & =  \frac{1}{2}\mathrm{div}\big(A\nabla (\psi \Tilde{\psi}  )\big) - \Big(\kappa + A\frac{\nabla \psi}{\psi}\Big)\cdot \nabla( \psi \Tilde{\psi}  ) - \mathrm{div}\Big(\kappa + A\frac{\nabla \psi}{\psi}\Big) (\psi \Tilde{\psi})\\
            & = \frac{1}{2}\mathrm{div}\big(A(\nabla \psi \Tilde{\psi} +  \psi \nabla\Tilde{\psi} )\big) - \Big(\kappa + A\frac{\nabla \psi}{\psi}\Big)\cdot (\nabla \psi \Tilde{\psi} +  \psi \nabla\Tilde{\psi} ) - \mathrm{div}\Big(\kappa + A\frac{\nabla \psi}{\psi}\Big) (\psi \Tilde{\psi}).
        \end{split}
    \end{equation*}
    Now we expand each of the terms above to get
    \begin{equation*}
        \begin{split}
            \frac{1}{2}\mathrm{div}(A(\nabla \psi \Tilde{\psi} +  \psi \nabla\Tilde{\psi} )) &= \nabla \Tilde{\psi} A \nabla \psi + \frac{1}{2} \psi \mathrm{div}(A \nabla \Tilde{\psi}) + \frac{1}{2} \Tilde{\psi} \mathrm{div}(A \nabla \psi),\\
            - \Big(\kappa + A\frac{\nabla \psi}{\psi}\Big)\cdot (\nabla \psi \Tilde{\psi} +  \psi \nabla\Tilde{\psi} ) &= - \Tilde{\psi}\kappa \cdot \nabla \psi - \frac{\Tilde{\psi}}{\psi} \nabla \psi^T A \nabla \psi - \psi\kappa \cdot \nabla \Tilde{\psi} - \nabla \Tilde{\psi}^T A \nabla \psi,\\
            - \mathrm{div}\Big(\kappa + A\frac{\nabla \psi}{\psi}\Big) \psi \Tilde{\psi} &=  - \mathrm{div}(\kappa) \psi \Tilde{\psi} + \frac{\Tilde{\psi}}{\psi}\nabla \psi^T A \nabla \psi-\mathrm{div}(A\nabla \psi) \Tilde{\psi}.
        \end{split}
    \end{equation*}
    Recall that $\Tilde{\psi}$ satisfies the adjoint equation $(\mathcal{L}^0+K^0)^*\Tilde{\psi} = \gamma(0)\Tilde{\psi}$ for $\mathcal{L}^0$ and $K^0$ introduced in \eqref{eq:Intro:GeneratorExponential}, and that $\psi$ solves the eigenvalue problem $(\mathcal{L}^0+K^0)\psi = \gamma(0)\psi$ in \eqref{eq:Intro:EigenvalueProblem}. Plugging the previous expressions back into the original identity for $(\mathcal{K}^0)^*(\psi\Tilde{\psi})$, and grouping the terms conveniently, we obtain
    \begin{equation*}
        \begin{split}
            (\mathcal{K}^0)^*(\psi\Tilde{\psi}) &= \psi \bigg( \frac{1}{2} \mathrm{div}(A \nabla \Tilde{\psi}) - \kappa \cdot \nabla \Tilde{\psi } - \mathrm{div}(\kappa) \Tilde{\psi} \bigg) + \Tilde{\psi}\bigg( -\frac{1}{2} \mathrm{div}(A \nabla \psi) - \kappa \cdot \nabla \psi \bigg)\\
            & = \psi \big((\mathcal{L}^0)^*\Tilde{\psi}\big) + \Tilde{\psi}(-\mathcal{L}^0\psi) = \psi (\gamma(0)\Tilde{\psi} - K^0 \Tilde{\psi}) + \Tilde{\psi}(-\gamma(0)\psi+K^0\psi) = 0.
        \end{split}
    \end{equation*}
    Since $\int_{\TT^d}\psi(x)\Tilde{\psi}(x)dx=1$, by Proposition \ref{prop:Ergodic:InvariantMeasure}, we conclude that $\nu(x;0)=\psi(x)\Tilde{\psi}(x)$, which finishes the proof.
\end{proof}

\begin{proof}[Proof of Lemma \ref{lemma:interpolation}]
    Notice that if the event $A_n$ occurs, there must be one living particle $v\in \cN_{(n+1)T}$, and $\ell\in [nT,(n+1)T]$ such that $|X_v(\ell)-X_v(nT)+(nT-\ell)\vb| > \kappa(nT)$. Let $(Y_t)_{t\geq 0}$ be a weak solution of \eqref{eq:SDETiling2.0} with $\zeta = 0$.  By the Markov inequality and the many-to-one lemma with $\zeta=0$ (Lemma \ref{lemma:ManyToOneLemma}),
    \begin{equation}\label{eq:Interpolation}
        \begin{split}
            \proba{x}{}{A_n} &\leq \probaaa{x}{}{\exists v\in \cN_{(n+1)T}: \sup_{\ell \in [nT,(n+1)T]} |X_v(\ell)-X_v(nT)+(nT-\ell)\vb|> \kappa nT  }\\
            & \leq \expecteddd{x}{}{\sum_{v\in \cN_{(n+1)T}} \indc_{\{\sup_{\ell \in [nT,(n+1)T]} |X_v(\ell)-X_v(nT)+(nT-\ell)\vb|> \kappa nT\}}  }\\
            &=\expecddd{x}{0}{\tfrac{\psi(x;0)}{\psi(Y_t;0)}e^{-0\cdot(Y_t-Y_0)+(n+1)T\gs(0) }\indc_{\{\sup_{\ell \in [nT,(n+1)T]} |Y_\ell-Y_{nT}+(nT-\ell)\vb|> \kappa nT\}}} \\
            & \leq \tfrac{\max_{y\in \R^d}\psi(y;0) }{\min_{y\in \R^d}\psi(y;0)} e^{(n+1)T\gamma(0)}\sup_{y\in\R^d}\probbb{y}{0}{\sup_{\ell \in [0,T]} |Y_\ell-Y_0-\ell \vb|> \kappa nT}.
        \end{split}
    \end{equation}
    On the other hand, by Lemma \ref{lemma:Interpolation:ProbaEffectiveDrift}, we know that
    \begin{equation*}
        \vb = \int_{\TT^d} \Big(b(x) + A(x)\frac{\nabla \psi(x;0)}{\psi(x;0)} \Big)\nu(x;0)  dx,
    \end{equation*}
    where $\nu(x;0)$ is the density of the invariant measure of the projection of $Y_t$ into the torus. Observe from \eqref{eq:SDETiling2.0} that $b(x) + A(x)\frac{\nabla \psi(x;0)}{\psi(x;0)}$ is exactly the drift of $Y_t$. Thus, Corollary \ref{corollary:Ergodic:AsymptoticDriftDiffusion} gives a uniform bound on the probability in the right hand side of \eqref{eq:Interpolation}. Namely, there exist positive constants $C_1$ and $C_2$ depending on $d$, $\|b\|_{0,\alpha}$ and $\|\sigma\|_{0,\alpha}$ such that
    \begin{equation*}
        \sup_{y\in\R^d}\probbb{y}{}{\sup_{\ell \in [0,T]} |Y_\ell-Y_0-\ell \vb|> \kappa nT} \leq C_1 e^{-C_2 \kappa^2 n^2 T }.
    \end{equation*}
    Plugging this back into \eqref{eq:Interpolation}, the conclusion of the lemma follows.
    \end{proof}

\section{Proof of the shape theorem}\label{sec:TheProof}
In this section, we present the statements and modifications needed for the proof of Theorem \ref{thm:Intro:ShapeHauss}. Most of the results in this section follow using the exact same reasoning as corresponding results in \cite{addarioberry2025shapetheorembbmperiodic}, and for this reason, we omit most of the details. We will only describe the changes in notation, and the changes needed to obtain the corresponding result. Let us recall that in this section, we are always working under assumptions \textbf{\ref{AssumptionA:bSigmaRegularity}}, \textbf{\ref{AssumptionA:SigmaEllipticity}}, \textbf{\ref{AssumptionB:g}}, \textbf{\ref{AssumptionC:measurability}}, 
\textbf{\ref{AssumptionC:NoKillingPureBranching}}, \textbf{\ref{AssumptionC:HolderContinuity}}, and \textbf{\ref{AssumptionC:StochasticDomination}}. 

\subsection{Half-space estimates}\label{sec:Half-Space}

Most of our efforts will focus on proving Proposition \ref{prop:Intro:DirectionalBounds}, from which Proposition \ref{prop:Intro:QuantShape} will be a straightforward consequence. The upper half-space estimate \eqref{eq:Intro:UpperEstimate} follows directly from the following proposition, which provides a refinement of the upper estimate \eqref{eq:Intro:UpperEstimate}.

\begin{proposition}\label{lemma:UpperBound}
     Fix $e\in S^{d-1}$ and let $\vb$ be the effective drift introduced in \eqref{eq:Intro:EffectiveDrift}. 
    Let $\lambda_e>0$ be such that $\cs(e)=\frac{\gs(\lambda_e e)}{\lambda_e}$ as in 
     \eqref{eq:prelimCriticalSpeed}, $\css(e)>0$ is the speed relative to $\vb$ defined in \eqref{eq:Intro:SpeedRelativeTo}, and $\zeta_e:=\lambda_e e \in \R^d$. Let $\psi(\cdot ;\zeta_e)$ be the principal eigenfunction of the eigenvalue problem \eqref{eq:Intro:EigenvalueProblem},  normalized so that $\int_{\TT^d} \psi(x;\zeta_e) dx=1$. Then for every $t>0$ and $f:[0,\infty)\to \R$,

    \begin{equation}\label{eq:upper-half-estimate}
        \sup_{x\in \R^{d}}\proba{x}{}{\exists v\in \mathcal{N}_t : (X_t(v)-X_{0}(v)-t\vb)\cdot e> \css(e)t + f(t)} \leq \frac{\max_{y\in \R^d}\psi(y;\zeta_{e}) }{ \min_{y\in \R^d}\psi(y;\zeta_{e}) } e^{-\lambda_e f(t)}.
    \end{equation}
\end{proposition}

\begin{remark}
    For $\epsilon\in (0,1)$, taking $f(t)=\epsilon \cs(e) t$ in the upper half-space estimate, we obtain 
    \begin{equation}
        \sup_{x\in \R^{d}}\proba{x}{}{\exists v\in \mathcal{N}_t : (X_t(v)-X_{0}(v)-t\vb)\cdot e> (1+\epsilon)\css(e)t} \leq Ce^{-t\epsilon \lambda_e\css(e)}.
    \end{equation} 
\end{remark}

The upper half-space estimate is proved via a many-to-one lemma as in \cite[Proposition 3.1]{addarioberry2025shapetheorembbmperiodic}.

\begin{proof}[Proof of Proposition \ref{lemma:UpperBound}]
    Notice that Proposition \ref{lemma:UpperBound} only differs from \cite[Proposition 3.1]{addarioberry2025shapetheorembbmperiodic} in that the position of the particles is corrected by the effective drift $t\vb$, and $\css(e)$ plays the role of $\cs(e)$. The proof follows identically, applying Lemma \ref{lemma:ManyToOneLemma} instead of \cite[Lemma 2.9, Lemma 2.10]{addarioberry2025shapetheorembbmperiodic}. Here, $(Y_t)_{t\geq 0}$ denotes the weak solution of \eqref{eq:SDETiling2.0} under the measure $\mathbb{P}_x^{\zeta_e}$. We take the principal eigenvalue $\gamma(\zeta_e)$, and we use the fact that $\css(e)>0$ introduced in \eqref{eq:Intro:SpeedRelativeTo}.
\end{proof}

\begin{remark}
    For $f(t)$ having the form $f(t)=at^\beta$ with $\beta>1/2$ for a positive constant $a$, the estimate \eqref{eq:upper-half-estimate} can be improved by involving the measure $\mathbb{P}^{\zeta_e}$ introduced in Lemma \ref{lemma:ManyToOneLemma}, although the estimate obtained above is enough for our purposes. For instance, employing the many-to-one lemma, and recalling that $Y_t$ is a weak solution of \eqref{eq:SDETiling2.0}, one obtains
    \begin{equation*}
        \begin{split}
            \sup_{x\in \R^{d}}\mathbf{P}_{x}[\exists  v\in \mathcal{N}_t : (X_t(v)-X_{0}(v)&-t\vb)\cdot e> \css(e)t + f(t)]\\ 
            &\leq \frac{\max_{y\in \R^d}\psi(y;\zeta_{e}) }{ \min_{y\in \R^d}\psi(y;\zeta_{e}) } e^{-\lambda_e f(t)}\mathbb{P}^{\zeta_e}[(Y_t-Y_0-t\vb)\cdot e > \Tilde{c}(e)t + f(t)].
        \end{split}
    \end{equation*}
    For $f(t)$ having the form $f(t)=t^\beta$ with $\beta>1/2$, we can employ the asymptotics of $Y_t$ proved in Section \ref{sec:Ergodic} to improve the bound above.
\end{remark}

The lower half-space estimate \eqref{eq:Intro:LowerEstimate} corresponds to \cite[Proposition 3.2]{addarioberry2025shapetheorembbmperiodic}. The proof of this result is more involved, we consider an embedded discrete-time branching process constructed inductively as follows. Fix $\epsilon>0$ and $T_0>0$. Let
\begin{equation}\label{eq:FirstGeneration}
    I_1^{T_0,\epsilon} := \{v\in \mathcal{N}_{T_0}: (X_{T_0}(v)-X_0(v)-T_0\vb)\cdot e \geq (1-\epsilon)\css(e) T_0 \}.
\end{equation}
We let $I_1^{T_0,\epsilon}$ be the first generation of the discrete-time spatial branching process. Observe that $I_1^{T_0,\epsilon}\subset \cN_{T_0}$. Suppose that for $n\geq 1$, we have defined $I_n^{T_0,\epsilon}$, the $n$-th generation of the discrete-time spatial branching process, and that $I_n^{T_0,\epsilon}\subset \cN_{nT_0}$. We will define the $(n+1)$-generation $I_{n+1}^{T_0,\epsilon}$ as a subset of $\cN_{(n+1)T_0}$. For $v\in I_n^{T_0,\epsilon}$, the descendants of $v$ in the discrete-time branching process are the particles $u\in \cN_{(n+1)T_0}$ such that $v\leq u$ and
\begin{equation*}
    (X_{(n+1)T_0}(u)-X_{nT_0}(u)-T_0\vb)\cdot e\geq (1-\epsilon)\css(e) T_0.
\end{equation*}
In other words, recalling that for $v\in \mathcal{N}_t$ and $s<t$ we write $v\in \mathcal{N}_s$ to denote the unique living ancestor of $v$ at time $s$,
\begin{equation}\label{eq:Generation}
    I_{n+1}^{T_0,\epsilon} := \{v \in \mathcal{N}_{(n+1)T_0}: v \in I_{n}^{T_0,\epsilon},  (X_{(n+1)T_0}(v)-X_{nT_0}(v)-T_0\vb)\cdot e\geq (1-\epsilon)\css(e) T_0\}.
\end{equation}
Observe that for $k\in \N$, if $v\in I_{k}^{T_0,\epsilon}$,
\begin{equation}\label{eq:generationEmbed}
    (X_{k T_0}(v)-X_0(v)-kT_0 \vb)\cdot e =\sum_{j=0}^{k-1}(X_{(j+1) T_0}(v)-X_{j T_0}(v) - T_0\vb)\cdot e \geq (1-\epsilon)\css(e) k T_0.
\end{equation}

We obtain the following survival result for $I^{T_0,\epsilon}_n$. The proof is given in Section \ref{sec:SurvivalLargeDeviation}.
\begin{lemma}\label{lemma:SurvivalProbability}
     Fix $\epsilon>0$. There exists $\beta=\beta(\epsilon)\in (0,1)$ such that, for any $T_0>0$ sufficiently large, 
    \begin{equation*}
        \inf_{x\in \R^d}\proba{x}{}{\# I_n^{T_0,\epsilon} \neq 0 \ \forall n\in \N } \geq \beta.
    \end{equation*}
\end{lemma}

This lemma, combined with Lemma \ref{lemma:interpolation}, give the following weak version of Proposition \ref{lemma:lowerestimate}.

\begin{lemma}\label{lemma:Half-space:WeakLowerEstimate}
    Fix $\epsilon\in (0,1)$. There exist $T_0=T_0(\epsilon)>0$, $K=K(\epsilon, d)\in \N$, and $ \eta \in (0,1)$, such that
    \begin{equation}\label{eq:lower1}
        \inf_{x\in \R^d} \mathbf{P}_x[ \forall t>KT_0, \ \sup_{v\in \cN_t} \{(X_t(v)-x-t\vb)\cdot e\}\geq (1-\epsilon)\css(e)t]\geq \eta.
    \end{equation}
\end{lemma}

\begin{proof}
    Up to correcting by the effective drift $t\vb$, and letting $\css(e)$ play the role of $\cs(e)$, the statement of Lemma \ref{lemma:Half-space:WeakLowerEstimate} and \cite[Lemma 3.10]{addarioberry2025shapetheorembbmperiodic} are identical. The proof follows the same reasoning, but we employ Lemma \ref{lemma:SurvivalProbability} instead of \cite[Lemma 3.9]{addarioberry2025shapetheorembbmperiodic}, and to control the term $\mathbf{P}_x[\cup_{n=k}^\infty A_n]$, we employ Lemma \ref{lemma:interpolation}.
\end{proof}

Finally, a cutoff argument, which employs the upper half-space estimate, improves the previous bound into the lower half-space estimate \eqref{eq:Intro:LowerEstimate}.

\begin{proof}[Proof of \eqref{eq:Intro:LowerEstimate}]
    With Proposition \ref{lemma:UpperBound} at hand, and up to correcting by the effective drift $t\vb$, the cutoff argument to prove the lower half-space estimate \eqref{eq:Intro:LowerEstimate} follows exactly as the corresponding proof in \cite[Section 3.2.3]{addarioberry2025shapetheorembbmperiodic}.  The definitions for $M_t, M^{-}_t$ and $R_s^{v,L}$ in \cite[Section 3.2.3]{addarioberry2025shapetheorembbmperiodic} thus become 
    \begin{equation*}\label{eq:MaxDisplacement}
        M_t:=\max_{v\in \cN_t}\{(X_t(v)-X_0(v) - t\vb)\cdot e\}\ \text{ and } \ M_t^-:=\min_{v\in \cN_t}\{(X_t(v)-X_0(v)-t\vb)\cdot e\},
    \end{equation*}
    and, for $L\geq 0$, $v\in \cN_L$, and $s\geq 0$,
    \begin{equation*}\label{eq:MaxDisplacement1}
        R_s^{v,L}:=\text{argmax}_{\{X_{L+s}(u)-X_0(u)-(L+s)\vb \in \R^d:\ u\in \cN_{L+s}, v\leq u\}}\{(X_{L+s}(u)-X_0(u)-(L+s)\vb)\cdot e  \}.
    \end{equation*}
    Once again, $\css$ plays the role of $\cs$, $X_t(v)-X_0(v)-t\vb$ the role of $X_t(v)-X_0(v)$, and $R^{v,L}_{t-L}-X_L(v) - L\vb$ the role of $R^{v,L}_{t-L}-X_L(v)$. Finally, instead of employing \cite[Proposition 3.1, Lemma 3.10 and Lemma 2.2]{addarioberry2025shapetheorembbmperiodic} we use, respectively, Proposition \ref{lemma:UpperBound}, Lemma \ref{lemma:Half-space:WeakLowerEstimate} and Corollary \ref{corollary:tailBranchingTime}. With these changes the final rate function obtained is $\Gamma(\epsilon, e)= \min\{\tfrac{\kappa}{2}\lambda_{-e}\css(-e), \log(\zeta)\tfrac{\kappa}{Ta},-\log(1-\eta)\}$.
\end{proof}

\subsection{A large deviation principle and survival of the embedded branching process}\label{sec:SurvivalLargeDeviation}

Here we prove Lemma \ref{lemma:SurvivalProbability}, which employs a criterion for the survival of general branching processes (see \cite[Theorem 2.4]{maillard2025generalisedprincipaleigenvaluesglobal}). Applied to our setting, the criterion for survival reads as follows.

\begin{lemma}\cite[Theorem 2.4.2]{maillard2025generalisedprincipaleigenvaluesglobal}\label{lemma:criterion_survival}
    Suppose there exists a bounded measurable function $q:\R^d \to \R$, $\delta \in (0,1)$, and $\rho\geq 1$ such that, for every $x\in \R^d$,
    \begin{equation}\label{eq:criterion_survival}
        \mathbb{E}_x\Big[\sum_{v\in I_1^{T_0,\epsilon}}q(X_t(v))\exp\Big( -\delta \sum_{v\in I_1^{T_0,\epsilon}}q(X_t(v)) \Big)\Big] \geq \rho q(x).
    \end{equation}
    Then, if $q(x)>0$, $ \proba{x}{}{\# I_n^{T_0,\epsilon} \neq 0 \ \forall n\in \N } \geq  \proba{x}{}{\liminf_{n\to \infty}\rho^{-n}\# I_n^{T_0,\epsilon}>0  }\geq 1-e^{-\delta q(x)}$.
\end{lemma}
\begin{remark}
    The lower bound with the exponential term above is not explicitly given in the statement of \cite[Theorem 2.4.2]{maillard2025generalisedprincipaleigenvaluesglobal}, but it can be deduced from the last step of the proof.
\end{remark}

We will take $q\equiv 1$ in \eqref{eq:criterion_survival}. To obtain \eqref{eq:criterion_survival} for this choice of $q$, we employ a large deviation principle on $(Y_t)_{t\geq 0}$, the SDE introduced in the many-to-one lemma (see \eqref{eq:SDETiling2.0}).

\begin{proposition}\label{proposition:largeDeviation}
    Fix $x\in \R^d$ and $e\in S^{d-1}$. Let $\lambda_e>0$ be the minimizer in \eqref{eq:prelimCriticalSpeed}, and recall that $\zeta_e = \lambda_e e$. Let $(Y_t)_{t\geq 0}$ be the solution of \eqref{eq:SDETiling2.0} under the measure $\bP_x^{\zeta_e}$. Let $\hat{Y}_t:=(1/t)e\cdot(Y_t-Y_0)$. Let $I_e:\R\rightarrow (-\infty,+\infty]$ be the function defined by
    \begin{equation}\label{eq:goodratefunction}
        I_e(\lambda)=\sup_{\eta \in \R} \{ \eta \lambda - [\gamma((\lambda_e+\eta)e)-\gamma(\zeta_e)]\}=[\gamma((\lambda_{e}+\cdot)e)-\gamma(\zeta_e)]^{*}(\lambda).
    \end{equation}
    For a set $A\subset \R$, let $I_e(A):=\inf_{\lambda\in A}I_e(\lambda)$. The following holds:
    \begin{enumerate}[(a)]
        \item For every closed set $\mathcal{C}\subset \R$, $\limsup_{t\rightarrow \infty}\sup_{x\in \R^d}\frac{1}{t}\log \bP_x^{\zeta_e}[\hat{Y}_t \in \mathcal{C}]\leq -I_e(\mathcal{C})$.
        \item For every open set $\mathcal{O}\subset \R$, $\liminf_{t\rightarrow \infty}\inf_{x\in \R^d}\frac{1}{t}\log  \bP_x^{\zeta_e}[\hat{Y}_t \in \mathcal{O}]\geq -I_e(\mathcal{O})$.
    \end{enumerate}
\end{proposition}

This large deviation principle is a consequence of a uniform version of the G\"artner-Ellis (See \cite[Theorem 3.3]{addarioberry2025shapetheorembbmperiodic}) and the following lemma.

\begin{lemma}\label{lemma:logmoment}
    Fix $x\in \R^d$ and $e\in S^{d-1}$. Let $\lambda_e>0$ be the minimizer in \eqref{eq:prelimCriticalSpeed}, and recall that $\zeta_e = \lambda_e e$. Let $(Y_t)_{t\geq 0}$ be the solution of \eqref{eq:SDETiling2.0} under the measure $\bP_x^{\zeta_e}$. Let $\hat{Y}_t:=(1/t)e\cdot(Y_t-Y_0)$. For every $\eta \in \R$ and $e\in S^{d-1}$, there exist $\theta_1,\theta_2\in \R$ such that, for every $x\in \R^d$, 
    \begin{equation}\label{eq:bound0}
        \frac{\theta_1}{t}\leq \frac{1}{t} \log \expec{x}{\zeta_e}{e^{\eta t \hat{Y}_t}} -\gamma((\lambda_e + \eta)e)+\gamma(\zeta_e) \leq \frac{\theta_2}{t},
    \end{equation}
    where $\gamma$ is the principal eigenvalue introduced in \eqref{eq:Intro:EigenvalueProblem}.
\end{lemma}

\begin{proof}
    Lemma \ref{lemma:logmoment} corresponds to \cite[Lemma 3.4]{addarioberry2025shapetheorembbmperiodic}. The proof is identical, applying Lemma \ref{lemma:ManyToOneLemma} instead of the corresponding \cite[Lemma 2.10]{addarioberry2025shapetheorembbmperiodic}. Thus, the weak solution $Z_t$ of \eqref{eq:Intro:Diffusion} plays the role of $B_t$, the effective branching rate $R$ plays the role of $g$, and we consider weak solutions $(Y_t)_{t\geq 0}$ and $(\Tilde{Y}_t)_{t\geq 0}$ of \eqref{eq:SDETiling2.0} under $\mathbb{P}^{\zeta_e}_x$ and $\mathbb{P}^{(\lambda_e+\eta)e}_x$ respectively.
\end{proof}

\begin{proof}[Proof of Proposition \ref{proposition:largeDeviation}]
     This follows directly from applying the uniform G\"artner-Ellis (\cite[Theorem 3.3]{addarioberry2025shapetheorembbmperiodic}) and Lemma \ref{lemma:logmoment}. We need to verify that the function $\eta  \mapsto \gamma((\lambda_e+\eta)e) - \gamma(\zeta_e)$ for $\eta\in \R$ satisfies the hypothesis of the aforementioned theorem. This is achieved identically as in the proof of \cite[Proposition 3.6]{addarioberry2025shapetheorembbmperiodic}, but instead of employing \cite[Proposition 2.5]{addarioberry2025shapetheorembbmperiodic}, we employ the corresponding Proposition \ref{Prop:PropertiesGamma}. 
\end{proof}

The function $I_e$ has the following properties.

\begin{lemma}\label{lemma:propI} 
The following holds:
\begin{enumerate}[(i)]
    \item The function $I_e(\lambda)$ is finite for every $\lambda\in \R$, and $I_e$ is strictly convex and differentiable. Moreover, for every $\lambda\in \R$, there exists $\lambda^*\in \R$ such that
    \begin{equation*}
        I_e(\lambda)=\sup_{\eta \in \R} \{ \eta \lambda - [\gamma((\lambda_e+\eta)e)-\gamma(\zeta_e)]\}= \lambda \lambda^* - [\gamma((\lambda_e+\lambda^*)e)-\gamma(\zeta_e)]. 
    \end{equation*}
    \item For any $\lambda\in \R$, $I_e(\lambda)\geq 0$ and $I_e(\lambda)=0$ if and only if $\lambda=c^*(e)$, where $c^{*}(e)$ is defined in \eqref{eq:prelimCriticalSpeed}. Furthermore, $I_e$ is strictly decreasing on the interval $(-\infty,c^*(e))$ and strictly increasing on the interval $(c^*(e),+\infty)$.
    \item For any constant $\beta>0$, there exists $\delta=\delta(\beta)\in (0,1)$ such that for any $\kappa\in (0,\delta)$, 
    \begin{equation*}
        I_e(c^*(e) - \kappa ) < \kappa \beta\quad\text{and}\quad I_e(c^*(e) + \kappa ) < \kappa \beta. 
    \end{equation*}
\end{enumerate} 
\end{lemma}

\begin{proof}
    Notice that, while in most of the other propositions $\css(e)$ plays the role of $\cs(e)$, in Lemma \ref{lemma:propI} it does not. The properties of the function $I_e$ follow in the same manner as the corresponding \cite[Lemma 3.7]{addarioberry2025shapetheorembbmperiodic}, invoking Proposition \ref{Prop:PropertiesGamma} instead of \cite[Proposition 2.3]{addarioberry2025shapetheorembbmperiodic}.
\end{proof}

\begin{remark}
   A straightforward application of Proposition \ref{proposition:largeDeviation}, Lemma \ref{lemma:propI}, and the interpolation lemma (Lemma \ref{lemma:interpolation}) shows that $t^{-1}((Y_t-Y_0)\cdot e)$ converges $\mathbb{P}_x^{\zeta_e}$-almost surely to $\cs(e)$ as $t\to \infty$.
\end{remark}

Using the large deviation principle for $\hat{Y}_t$ as in Lemma \ref{lemma:logmoment} and the properties of $I_e$, we obtain the following uniform lower bound on the expected number of particles with average (relative) speed close to $\css(e)$, and thus, on the expected size of $I_1^{T_0,\epsilon}$.

\begin{lemma}\label{lemma:lowerestimate}
    Fix $e\in S^{d-1}$. Let $I_e$ be defined by \eqref{eq:goodratefunction}, $\css(e)$ be defined by \eqref{eq:Intro:SpeedRelativeTo}, and $\vb$ be the effective drift \eqref{eq:Intro:EffectiveDrift}. Fix $\epsilon\in (0,1)$. There exist constants $C=C(e)>0$ and $T=T(e)>0$ such that for every $t>T$ and $\kappa\in (0,1)$,
    \begin{equation}\label{eq:lowerdirectionalestimate}
        \inf_{x\in \R^d}\mathbb{E}_{x}\Big[\sum_{v\in \mathcal{N}_t} \indc_{\{(1-\kappa \epsilon)t\css(e)\geq (X_t(v)-X_0(v)-t\vb)\cdot e\geq (1-\epsilon)t\css(e) \}} \Big] \geq Ce^{ t[\kappa\epsilon\lambda_e \css(e)-2I_{e}(\cs(e)-\kappa\epsilon\css(e))] }.
    \end{equation}
    Furthermore, there is $\epsilon_0=\epsilon_0(e)\in (0,1)$ such that for $\epsilon\in (0,\epsilon_0)$ and for any $\kappa \in (0,1)$, 
    $$\kappa\epsilon\lambda_e \css(e)-2I_{e}(\cs(e)-\kappa\epsilon\css(e)) >0.$$ In particular, for any $\epsilon\in (0,\epsilon_0)$, the expected value in \eqref{eq:lowerdirectionalestimate} grows to $+\infty$ as $t\uparrow +\infty$.
\end{lemma}

\begin{proof}
    This lemma corresponds with \cite[Lemma 3.8]{addarioberry2025shapetheorembbmperiodic}, the differences in the statement are, adjusting the trajectory of the particles by $t\vb$, using $\css(e)$ instead of $\cs(e)$ inside the expectation operator, and the rate of the exponential term in the right-hand side of the inequality is $\kappa\epsilon\lambda_e \css(e)-2I_{e}(\cs(e)-\kappa\epsilon\css(e))$. Once again, instead of \cite[Lemma 2.9 and Lemma 2.10]{addarioberry2025shapetheorembbmperiodic}, we employ Lemma \ref{lemma:ManyToOneLemma}. Instead of using the properties of $I_e$ (\cite[Lemma 3.7]{addarioberry2025shapetheorembbmperiodic}) we employ the corresponding Lemma \ref{lemma:propI}, and finally, instead of applying the large deviation principle for the $g$-BBM (\cite[Proposition 3.6]{addarioberry2025shapetheorembbmperiodic}) we apply Proposition \ref{proposition:largeDeviation}.
\end{proof}

We finish this section with the proof of the survival result.

\begin{proof}[Proof of Lemma \ref{lemma:SurvivalProbability}]
    We apply Lemma \ref{lemma:criterion_survival} to obtain a uniform lower bound on the probability of survival of the discrete branching process. Through this proof $e$ will denote Euler's number. Let us take $\rho>1$, and, by Lemma \ref{lemma:lowerestimate}, $T_0>0$ sufficiently large such that
    \begin{equation}\label{eq:T0}
        \inf_{x\in \R^d}\expected{x}{}{\# I_1^{T_0,\epsilon}}>2e \rho.
    \end{equation}
    To choose $\delta\in (0,1)$, recall we assumed that every offspring distribution $\mu^{(x)}$ is stochastically dominated by the same offspring distribution $\bar{\mu}$ with finite first moment $\bar{m}_1$ (see Assumption \textbf{\ref{AssumptionC:StochasticDomination}}). Let $\bar{N}_{T_0}$ denote the total number of particles at time $T_0$ of a branching diffusion with constant branching rate $\|g\|_\infty$ and offspring distribution $\bar{\mu}$. Since $\bar{N}_{T_0}$ is integrable, and its distribution does not depend on the initial position $x\in \R^d$, let us fix $\delta\in (0,1)$ sufficiently small such that 
    \begin{equation}\label{eq:HomBranching}
        \sup_{x\in \R^d} \mathbb{E}_x\big[  \bar{N}_{T_0}\indc_{\{ \bar{N}_{T_0} > \delta^{-1} \}}\big] \leq e\rho.
    \end{equation}
    With this at hand, observe that for $q\equiv 1$,
    \begin{equation*}
        \begin{split}
            \mathbb{E}_x\Big[\sum_{v\in I_1^{T_0,\epsilon}}q(X_t(v))\exp\Big( -\delta \sum_{v\in I_1^{T_0,\epsilon}}q(X_t(v)) &\Big)\Big]  =  \mathbb{E}_x\big[ (\# I_1^{T_0,\epsilon})\exp (-\delta (\# I_1^{T_0,\epsilon}) ) \big]\\
            &\geq \mathbb{E}_x\big[ (\# I_1^{T_0,\epsilon})\indc_{\{ \delta (\# I_1^{T_0,\epsilon}) < 1 \}}\exp (-\delta (\# I_1^{T_0,\epsilon}) ) \big]\\
            & \geq e^{-1} \mathbb{E}_x\big[ (\# I_1^{T_0,\epsilon})\indc_{\{ \# I_1^{T_0,\epsilon} < \delta^{-1} \}}\big]\\
            & = e^{-1} \big( \mathbb{E}_x\big[ \# I_1^{T_0,\epsilon}\big]-\mathbb{E}_x\big[ (\# I_1^{T_0,\epsilon})\indc_{\{ \# I_1^{T_0,\epsilon} \geq  \delta^{-1} \}}\big] \big).
        \end{split}
    \end{equation*}
    Finally, $\# I_1^{T_0,\epsilon}$ is stochastically dominated by $\bar{N}_{T_0}$, and thus, by \eqref{eq:T0} and \eqref{eq:HomBranching}, 
    \begin{equation*}
        e^{-1} \big( \mathbb{E}_x\big[ \# I_1^{T_0,\epsilon}\big]-\mathbb{E}_x\big[ (\# I_1^{T_0,\epsilon})\indc_{\{ \# I_1^{T_0,\epsilon} \geq  \delta^{-1} \}}\big] \big) \geq e^{-1}\big( 2e\rho -\mathbb{E}_x\big[  \bar{N}_{T_0}\indc_{\{ \bar{N}_{T_0} > \delta^{-1} \}}\big] \big) \geq \rho = \rho q(x).
    \end{equation*}
    Plugging this back into our previous inequality proves that \eqref{eq:criterion_survival} holds with our choice of $q$, $\delta$, and $\rho$. Thus,
    \begin{equation*}
        \proba{x}{}{\# I_n^{T_0,\epsilon} \neq 0 \ \forall n\in \N } \geq 1-e^{-\delta }>0.
    \end{equation*}
    Since the function $q$ and $\delta\in (0,1)$ were chosen independent of the initial position $x\in \R^d$, the bound above is uniform for every initial condition $x\in \R^d$, and the conclusion follows.
\end{proof}

\subsection{The proof of Proposition \ref{prop:Intro:QuantShape} and Theorem \ref{thm:Intro:ShapeHauss}}\label{sec:ProofOfHaussShape}

The proof of Proposition \ref{eq:quantShapeThm} follows from Proposition \ref{prop:Intro:DirectionalBounds} by applying the following approximation result for Wulff shapes.

\begin{proposition}\cite[Proposition 1.6]{addarioberry2025shapetheorembbmperiodic}\label{prop:wulffapprox}
    Fix $a \in (0,1)$, let $\mathfrak{c}^*:S^{d-1}\rightarrow (a,a^{-1})$, and let $\cW(\mathfrak{c}^*)$ be the Wulff shape of $\mathfrak{c}^*$. Then for $\epsilon\in (0,1)$ there exist finite sets $\cR,\cQ\subset S^{d-1}$ such that 
    \begin{enumerate}[(i)]
        \item $ \bigcap_{r\in \cR} \cH_{r,\mathfrak{c}^*(r)}^{-}\subset(1+\epsilon)\cW(\mathfrak{c}^*)$, and
        \item for every compact convex set $K\subset B(0,2a^{-1})$, if for every $q\in \cQ$, $K\cap \cH^+_{q,\mathfrak{c}^*(q)} \neq \emptyset$, then 
        \begin{equation*}
            (1-\epsilon)\cW(\mathfrak{c}^*)\subset K.
        \end{equation*}
    \end{enumerate}
    \end{proposition}

\begin{proof}[Proof of Proposition \ref{prop:Intro:QuantShape}]
    After correcting by the effective drift $t\vb$, and letting $\css$ play the role of $\cs$, the proof follows the exact same reasoning as the proof of \cite[Proposition 1.2]{addarioberry2025shapetheorembbmperiodic}. 
\end{proof}

The following fact about convex bodies, combined with Proposition \ref{prop:Intro:QuantShape}, implies that particles travel near every extreme point of the shape.

\begin{lemma}\cite[Lemma 4.2]{addarioberry2025shapetheorembbmperiodic}\label{lemma:ApproxExtShape}
    Let \( W\subset \R^{d} \) be compact and convex. Fix \( \epsilon \in (0,1) \). There exists \( \delta \in (0,\epsilon) \) such that, for any set \( E \) with 
    \[
    (1 - \delta)W \subset \mathrm{conv}(E) \subset (1 + \delta)W,
    \]
    if $\xi\in W$ is an extreme point, then $B(\xi, \epsilon)\cap E\neq \emptyset$.
\end{lemma}

As in the proof of \cite[Proposition 4.3]{addarioberry2025shapetheorembbmperiodic}, we can use the particles that travel near extreme points of the shape to approximate any other point either on the interior of the shape, or in the boundary of the shape away from extreme points.

\begin{proposition}\label{prop:probabilityParticlesBalls}
    For all $\xi\in \cW(\css)$ and $\epsilon \in (0,1)$, there exist positive constants $C=C(\xi,\epsilon,d), \Gamma=\Gamma(\xi,\epsilon,d)$, and $T_0=T_0(\epsilon, \xi, d, \|R\|_{C^{0,\alpha}})$ such that, for every $t>T_0$,
    \begin{equation}\label{eq:ApproxIntW}
        \sup_{x\in \R^d} \proba{x}{}{(\mathcal{X}_t-x-t\vb)\cap tB(\xi,\epsilon)=\emptyset } \leq Ce^{-\Gamma t}.
    \end{equation}
\end{proposition}
\begin{proof}
    We use Carath\'eodory's theorem for convex hulls and the Markov property. This is the exact same argument given in \cite[Proposition 4.3]{addarioberry2025shapetheorembbmperiodic}, up to correcting by the effective drift $t\vb$, and applying Proposition \ref{eq:quantShapeThm} instead of \cite[Proposition 1.2]{addarioberry2025shapetheorembbmperiodic}.
\end{proof}

The proof of Theorem \ref{thm:Intro:ShapeHauss} follows by the compactness of the shape.
\begin{proof}[Proof of Theorem \ref{thm:Intro:ShapeHauss}]
    The result follows by applying the interpolation lemma (Lemma \ref{lemma:interpolation}) combined with the previous proposition, just as in the proof of \cite[Theorem 1.1]{addarioberry2025shapetheorembbmperiodic}.
\end{proof}

\section{Properties of the shape}\label{sec:PropertiesShape}

In this section we prove a few properties of the shape.

\subsection{The support function}

Given a convex body $K\subset \R^d$, the support function $h_{K}:\R^{d}\setminus \{0\} \to \R$ of $K$ is defined by 
\begin{equation}
    h_{K}(\zeta)=\sup\{ x\cdot \zeta: x\in K \}, \ \forall \zeta \in \R^{d}\setminus \{0\}.
\end{equation}
Notice that the support function of a convex body $K\subset \R^d$ is $1$-homogeneous, meaning that for every $\zeta \in \R^d$, $h_K(\zeta)=|\zeta|h_K(\zeta/|\zeta|)$. We consider the $1$-homogeneous extension of $\css:S^{d-1}\to \R$, defined, for every $\zeta\in \R^d\setminus \{0\}$ by
\begin{equation*}
    \css(\zeta) = |\zeta|\css(\zeta/|\zeta|).
\end{equation*}
Using Proposition \ref{prop:Intro:DirectionalBounds} we show that the $1$-homogeneous extension of $\css$ is the support function of its Wulff shape $\Tilde{\cW}:=\cW(\css)$. 

\begin{proposition}\label{prop:ShapeTheorem:ContinuityCSS}
    Let $\css:S^{d-1}\to (0,\infty)$ be the function defined in \eqref{eq:Intro:SpeedRelativeTo}. Then, for every $e\in S^{d-1}$, $\cH_{e,\css(e)}$ is a supporting hyperplane of $\Tilde{\cW}$. In particular, the $1$-homogeneous extension of $\css$ is the support function of $\Tilde{\cW}$.
\end{proposition}

\begin{remark}\label{remark:SupportFunction}
    Letting $\cW = \vb + \Tilde{\cW}$ and since $\css(e) = \cs(e)-\vb\cdot e$ (see \eqref{eq:prelimCriticalSpeed} and \eqref{eq:Intro:SpeedRelativeTo}), for every $e\in S^{d-1}$, $\cH_{e,\cs(e)}$ is a supporting hyperplane of $\cW$. Similarly, $\cW$ and $\Tilde{\cW}$ have the same regularity, and the $1$-homogeneous extensions of $\css$ and $\cs$ are the support functions of $\Tilde{\cW}$ and $\cW$ respectively. Moreover, the $1$-homogeneous extensions of $\cs$ and $\css$ satisfy the relation
    \begin{equation*}
        \cs(\zeta) = \css(\zeta) - \vb\cdot \zeta,\ \forall \zeta\in \R^d\setminus\{0\},
    \end{equation*}
    so the $1$-homogeneous extensions $\cs$ and $\css$ have the same regularity as each other.
\end{remark}

Before proving the proposition above, we prove the following auxiliary lemma. 

\begin{lemma}\label{lemma:Wulff:InteriorZero}
    Fix $a\in (0,1)$. Let $\mathfrak{c}^*:S^{d-1} \to (a,a^{-1}]$. Suppose that for some $e\in S^{d-1}$, $\cH_{e,\csf(e)}$ is not a supporting hyperplane of $\cW(\csf)$. Then, for every $\epsilon\in (0,1)$, there exists $\delta \in (0,\epsilon)$ and $\cR\subset S^{d-1}$ finite such that
    \begin{equation*}
        \bigcap_{r\in \cR}\cH^-_{r,\csf(r)+\delta} \subset \cH^-_{e,(1-\epsilon)\csf(e)}.
    \end{equation*}
\end{lemma}

\begin{proof}
    Let $e\in S^{d-1}$ as in the statement. Since $\cH_{e,\csf(e)}$ is not a supporting hyperplane of $\cW(\csf)$, and $\csf$ is bounded away from zero, and $0\in \mathrm{int}(\cW(\mathfrak{c}^*))$, there exists $0<c<\csf(e)$ such that $\cH_{e,c}$ is a supporting hyperplane of $\cW(\csf)$. Let $\epsilon', \kappa \in (0,1)$ be sufficiently small such that $(1+\epsilon')(1+\kappa)c < (1-\epsilon)\csf(e)$. By Proposition \ref{prop:wulffapprox}(i), there exists $\cR\subset S^{d-1}$ finite such that
    \begin{equation*}
        \bigcap_{r\in \cR}\cH^-_{r,(1+\kappa)\csf(r)} \subset (1+\epsilon')\cW((1+\kappa)\csf) \subset (1+\epsilon')(1+\kappa)\cH^-_{e,c} \subset (1-\epsilon)\cH^-_{e,\csf(e)}.
    \end{equation*}
    Letting $\delta = \min_{r\in \cR}\{\kappa \csf(r)\}$, it follows that $\bigcap_{r\in \cR}\cH^-_{r,\csf(r) + \delta } \subset (1-\epsilon)\cH^-_{e,\csf(e)}$, which finishes the proof.
\end{proof}

\begin{proof}[Proof of Proposition \ref{prop:ShapeTheorem:ContinuityCSS}]
    For the purpose of contradiction, assume that there exists $e\in S^{d-1}$ such that $\cH_{e,\css(e)}$ is not a supporting hyperplane of $\cW(\css)$. Let us fix $\epsilon\in (0,1)$. Recall that $\css$ is bounded away from zero and bounded from above by Lemma \ref{lemma:DifferentiabilityLambda_e}. Hence, we can apply Lemma \ref{lemma:Wulff:InteriorZero} to find $\delta \in (0,1)$, and a finite subset $\cR\subset S^{d-1}$ such that 
    \begin{equation*}
        \bigcap_{r\in \cR}\cH^-_{r,\css(r)+\delta} \subset \cH^-_{e,(1-\epsilon)\css(e)}.
    \end{equation*}
    For $r\in \cR$, let $\epsilon'(r) = \delta/\css(r)$, so that $\css(r)+\delta = (1+\epsilon'(r))\css(r)$. Thus, by \eqref{eq:Intro:UpperEstimate}, there exists $\Gamma = \Gamma(r,\delta)>0 $ such that 
    \begin{equation}
        \begin{split}
            \mathbf{P}_0\Big[t^{-1}(\mathcal{X}_t-t\vb)\subset \cH^-_{e,(1-\epsilon)\css(e)}\Big]& \geq  \mathbf{P}_0\Big[t^{-1}(\mathcal{X}_t-t\vb)\subset \bigcap_{r\in \cR}\cH^-_{r,\css(r)+\delta}\Big] \\
            & = 1 - \mathbf{P}_0[\exists r\in \cR \ \mathrm{s.t}\ \cH^+_{r,\css(r)+\delta}\cap t^{-1}(\mathcal{X}_t-t\vb)\neq \emptyset ] \\
            &\geq 1 - \sum_{r\in \cR}\mathbf{P}_0[ \cH^+_{r,(1+\epsilon'(r))\css(r)}\cap t^{-1}(\mathcal{X}_t-t\vb)\neq \emptyset ] \\
         & \geq 1 - (\# \cR) e^{-t  \min_{r\in \mathcal{R}}\Gamma(r,\delta)  }.
        \end{split}
    \end{equation}
    The right-hand side of the expression above tends to $1$ as $t\to +\infty$, but by \eqref{eq:Intro:LowerEstimate}, the left-hand side tends to zero. This gives a contradiction, and the lemma follows.

    The last claim follows by observing that the for $e\in S^{d-1}$, the above means that
    \begin{equation*}
        \css(e)=\inf\{c\in \R:\Tilde{\cW} \subset \cH^{-}_{e,c} \} = \sup\{ x\cdot e: x\in \Tilde{\cW} \}.
    \end{equation*}
    Hence, for $\zeta\in \R^d\setminus \{0\}$, the $1$-homogeneous extension $\css$ satisfies
    \begin{equation*}
        \css(\zeta)= |\zeta|\sup\{ (x\cdot \zeta)/|\zeta|: x\in \Tilde{\cW} \} = \sup\{ x\cdot \zeta: x\in \Tilde{\cW} \},
    \end{equation*}
    so it is the support function of $\Tilde{\cW}$.
\end{proof}

\subsection{Regularity of the shape and the support function}

Recall the definition of the principal eigenvalue $\gamma:\R^d \to \R$ of the family of eigenvalue problems \eqref{eq:Intro:EigenvalueProblem}. We consider the Legendre transform of $\gamma$ defined by
\begin{equation}\label{eq:PropertiesShape:Legendre}
    \Phi(\eta):= \inf_{\zeta\in \R^d}\{\zeta\cdot \eta -\gamma(\zeta)\}.
\end{equation}
The function $\Phi$ is employed in \cite[Chapter 7.2]{Freidlin+1985} to describe the spreading of solutions of the F-KPP equations with compactly supported initial condition. It also plays a key role in describing the exact asymptotics of the transition kernel of the periodic branching diffusion in \cite{MR4162842}. We list the following known properties of $\Phi$.

\begin{lemma}\label{lemma:LegendreTransformGamma}
    The following holds:
    \begin{enumerate}[(i)]
        \item The function $\Phi:\R^d \to \R$ is strictly convex, belongs to $C^2(\R^d)$ and it is co-finite, meaning that, for every $x\in \R^d$, $\lim_{\lambda\to \infty}\Phi(\lambda x)/\lambda =+\infty$.
        \item The vector-valued function $\nabla \gamma:\R^d \to \R^d$ is bijective, and $\nabla \Phi = (\nabla \gamma )^{-1}$.
        \item At every $\zeta\in \R^d$, the Hessian $\nabla^2 \gamma (\zeta)$ is invertible. The Hessian of $\Phi$ satisfies $\nabla^2 \Phi(\cdot)|_{\nabla \gamma(\zeta)} = \nabla^2 \gamma (\zeta)^{-1}$.
        \item For every $\alpha\in \R$, the level set $\Phi_\alpha:=\{y\in \R^d:\Phi(y)\leq \alpha\}$ is compact, convex, and its topological boundary satisfies $\partial \Phi_\alpha = \{y\in \R^d:\Phi(y)= \alpha\}$. 
    \end{enumerate}
\end{lemma}
\begin{proof}
    First, we show a slightly weaker claim than \textit{(i)}, that $\Phi$ is strictly convex, belongs to $C^1(\R^d)$ and it is co-finite. To do so, we employ \cite[Theorem 26.6]{Rockafellar+1970}, which states that if $\gamma \in C^1(\R^d)$ is strictly convex and co-finite, then its Legendre transform $\Phi$ is strictly convex, belongs to $C^1(\R^d)$, and is co-finite. Thus, the properties of $\gamma$ proved in Proposition \ref{Prop:PropertiesGamma} imply that $\gamma$ satisfies the required hypotheses to apply \cite[Theorem 26.6]{Rockafellar+1970} and the slightly weaker claim follows.

    Now we prove \textit{(ii)}. Under the hypotheses mentioned above, another conclusion of \cite[Theorem 26.6]{Rockafellar+1970} is that $\nabla \gamma:\R^d \to \R^d$ is a bijection. The fact that $\nabla \Phi = (\nabla \gamma)^{-1}$ then follows from \cite[Theorem 26.5]{Rockafellar+1970}, which gives this formula for the gradient of the Legendre transform.
    
    Now we prove that $\Phi \in C^2(\R^d)$ and \textit{(iii)}. Let us fix $\zeta\in \R^d$, first we prove that the Hessian $\nabla \gamma^2(\zeta)$ is invertible. Recall $\mathcal{K}^\zeta$ introduced in \eqref{eq:GeneratorTilting}. Let $\nu(\cdot;\zeta)$ be the density of the invariant measure of the diffusion with generator $\mathcal{K}^\zeta$ taken as a diffusion on the torus. Fix $\lambda \in \R^d$, and let $\eta^\lambda$ be a periodic solution of $\mathcal{K}^\zeta \eta^\lambda(x) = \nu(\nabla\phi(\cdot;\zeta);\zeta)\cdot \lambda - \nabla\phi(x;\zeta)\cdot \lambda$, which exists by Proposition \ref{prop:Ergodic:Fredholm}. Recall the principal eigenfunction $\psi(\cdot;\zeta)$ of \eqref{eq:Intro:EigenvalueProblem} and let $\Tilde{\psi}(\cdot;\zeta )>0$ satisfy the adjoint equation $(\mathcal{L}^\zeta + K^\zeta)^*\Tilde{\psi} = \gamma(\zeta)\Tilde{\psi}$ for $\mathcal{L}^\zeta$ and $K^\zeta$ as defined in \eqref{eq:Intro:GeneratorExponential}. The argument given in the proof of Lemma \ref{lemma:Interpolation:ProbaEffectiveDrift} can be replicated to show that if we take $\psi(\cdot;\zeta)$ and $\Tilde{\psi}(\cdot;\zeta)$ so that $\int_{\TT^d}\psi(x;\zeta)\Tilde{\psi}(x;\zeta)dx =1$, $\nu(\cdot;\zeta)= \psi(\cdot;\zeta)\Tilde{\psi}(\cdot;\zeta)$. Thus, by \cite[Theorem 8.2.10 (iii)]{MR1326606} and \textbf{\ref{AssumptionA:SigmaEllipticity}}, we have
    \begin{equation*}
        \lambda ^T \nabla^2 \gamma(\zeta)  \lambda = \int_{\TT^d} \Big( \lambda + \nabla \eta^\lambda(x)\Big)^T A\Big(\lambda + \nabla \eta^\lambda(x)\Big) \nu(x;\zeta) dx \geq \int_{\TT^d} \theta^{-1} \|\lambda + \nabla \eta^\lambda(x) \|^2 \nu(x;\zeta)dx.
    \end{equation*}
        If $\lambda + \nabla \eta^\lambda \equiv 0$, $\nabla \eta^\lambda \equiv -\lambda$, which would contradict the fact that $\eta^\lambda$ is periodic. It follows that the right-hand side of the inequality above is strictly positive. Thus, for every $\lambda\in \R^d$, $\lambda ^T \nabla^2 \gamma(\zeta)  \lambda>0$, which implies that $\nabla^2 \gamma(\zeta)$ is invertible. The fact that $\Phi$ is twice differentiable is a direct consequence of the inverse mapping theorem applied to the continuously differentiable vector-valued function $\nabla \gamma:\R^d \to \R^d$. Clearly, the Jacobian of said function is the Hessian $\nabla^2 \gamma$, which we showed is invertible. Thus, $(\nabla \gamma)^{-1}$ is differentiable, and its Jacobian is exactly the inverse of the Jacobian of $\nabla \gamma$, which is given by $(\nabla^2 \gamma)^{-1}$. In other words, the Hessian of $\Phi$ is $(\nabla^2 \gamma)^{-1}$. Finally, $(\nabla^2 \gamma)^{-1}$ is continuous because $\nabla^2 \gamma$ is, and thus, $\Phi \in C^2(\R^d)$. 
        
        For the last claim, it is clear that the level sets are always closed and convex. The fact that $\Phi$ is co-finite and strictly convex imply that $\Phi$ is coercive, meaning that $\lim_{|x|\to \infty}\Phi(x)=+\infty$, and the compactness of the level sets follows. The characterization of the boundary follows by the strict convexity and continuity of $\Phi$.
\end{proof}

\begin{remark}\label{remark:RegularityGamma_Legendre}
    If $\gamma \in C^k(\R^d)$ for $k\geq 2$, the argument in the proof of Lemma \ref{lemma:LegendreTransformGamma}(i) gives $\Phi \in C^k(\R^d)$. 
\end{remark}

We characterize the shape in terms of $\Phi$.
\begin{proposition}\label{prop:shape=LargeDev}
    The shape of the periodic branching diffusion satisfies $\cW = \{y\in \R^d: \Phi(y)\leq 0\}$.
\end{proposition}

Before proving Proposition \ref{prop:shape=LargeDev}, we first prove the following auxiliary lemma employing the asymptotics of the transition kernel \cite[Theorem 2.2]{MR4162842}. 

\begin{lemma}\label{lemma:LargeDevBalls}
    The following holds:
    \begin{enumerate}[(i)]
        \item Let $\zeta \in \{y\in \R^d: \Phi(y)> 0\}$ and $\epsilon \in (0,1)$ be such that $B(\zeta, \epsilon)\subset \{y\in \R^d: \Phi(y)\geq \epsilon\}$, then
    \begin{equation}
        \mathbf{E}_0\Bigg[\sum_{v\in \cN_t}\indc_{\{X_v(t)\in tB(\zeta,\epsilon)\}}\Bigg] \to 0 \ \ \mathrm{ as } \ \ t\to +\infty.
    \end{equation}
        \item Let $\zeta \in \mathrm{int}(\{y\in \R^d: \Phi(y)\leq 0\})$ and $\epsilon \in (0,1)$ be such that $B(\zeta, \epsilon)\subset \{y\in \R^d: \Phi(y)\leq \epsilon\}$, then
    \begin{equation}
        \mathbf{E}_0\Bigg[\sum_{v\in \cN_t}\indc_{\{X_v(t)\in tB(\zeta,\epsilon)\}}\Bigg] \to + \infty \ \ \mathrm{ as } \ \ t\to +\infty.
    \end{equation}
    \end{enumerate}
\end{lemma}
\begin{proof}
    For the first claim, for $z\in \R^d$ let $\rho(z):=\nabla \Phi(z)\in \R^d$. For $\rho \in \R^d$, recall the principal eigenfunction $\psi(\cdot;\rho)$ of \eqref{eq:Intro:EigenvalueProblem} and let $\Tilde{\psi}(\cdot;\rho)>0$ satisfy the adjoint equation $(\mathcal{L}^\rho + K^\rho)^*\Tilde{\psi} = \gamma(\rho)\Tilde{\psi}$ for $\mathcal{L}^\rho$ and $K^\rho$ as defined in \eqref{eq:Intro:GeneratorExponential}. In \cite[Theorem 2.2]{MR4162842}, let us take $L>0$ sufficiently large so that $1+o_L(1)<2$. Employing said theorem, for every $t$ sufficiently large,
    \begin{equation}\label{eq:LargeDeviationBalls}
        \begin{split}
            \mathbf{E}_0\Bigg[\sum_{v\in \cN_t}&\indc_{\{X_v(t)\in tB(\zeta,\epsilon)\}}\Bigg] = \int_{\R^d}m(t,0,y) \indc_{\{y\in tB(\zeta,\epsilon)\}}dy\\
            & \leq 2\int_{tB(\zeta, \epsilon)} (2\pi t)^{-2d}\mathrm{det}[\nabla^2 \Phi(y/t)]^{1/2}
            e^{-t \Phi (\frac{y}{t})} \psi(0;\rho(y/t))\Tilde{\psi}(y;\rho(y/t))  dy\\
            & \leq 2 (2\pi t)^{-2d} \sup_{z\in B(\zeta,\epsilon)}\Big\{\mathrm{det}[\nabla^2 \Phi(z)]^{1/2} \psi(0;\rho(z))\Tilde{\psi}(tz;\rho(z))\Big \} t^{d}\int_{B(\zeta,\epsilon)} e^{-t\Phi(z)} dz\\
            & \leq C t^{-d} \sup_{z\in B(\zeta,\epsilon)}\Big\{\mathrm{det}[\nabla^2 \Phi(z)]^{1/2} \psi(0;\rho(z))\Tilde{\psi}(tz;\rho(z))\Big \} |B(\zeta,\epsilon)|e^{-t\inf_{z\in B(\zeta,\epsilon)}\Phi(z)}.
        \end{split}
    \end{equation}
    By the continuity of $\nabla \Phi$, of $\nabla^2 \Phi$, and of the function $\rho \mapsto \psi(x;\rho)\Tilde{\psi}(x;\rho)$ for every $\rho \in \R^d$,
    \begin{equation*}
        \sup_{z\in B(\zeta,\epsilon)}\Big\{\mathrm{det}[\nabla^2 \Phi(z)]^{1/2} \psi(0;\rho(z))\Tilde{\psi}(tz;\rho(z))\Big \}<+\infty.
    \end{equation*}
    Thus, the right-hand side of the equation above converges to zero as $t\to \infty$.
    
    A symmetric argument shows the second claim.
\end{proof}

\begin{proof}[Proof of Proposition \ref{prop:shape=LargeDev}]
    Let us write $\cW':=\{y\in \R^d: \Phi(y)\leq 0\}$. Recall that $\cW = \Tilde{\cW} + \vb$. First, we show that $\cW\subset \cW'$. To do so, fix 
    $\zeta \in \cW$. For $\epsilon>0$ and $t>0$, observe that
    \begin{equation*}
        \mathbf{E}_0\Bigg[\sum_{v\in \cN_t}\indc_{\{X_v(t)\in tB(\zeta,\epsilon)\}}\Bigg] = \mathbf{E}_0\Bigg[\sum_{v\in \cN_t}\indc_{\{X_v(t)-t\vb\in tB(\zeta-\vb,\epsilon)\}}\Bigg] \geq \mathbf{P}_0\Bigg[\sum_{v\in \cN_t}\indc_{\{X_v(t)-t\vb\in tB(\zeta-\vb,\epsilon)\}}\geq 1\Bigg].
    \end{equation*}
     Since $\zeta-\vb \in \Tilde{\cW}$, by Proposition \ref{prop:probabilityParticlesBalls}, for every $\epsilon>0$, as $t\to \infty$ the right-hand side of the equation above converges to $1$. On the other hand, observe that if $\zeta\not \in \cW'$, we could choose $\epsilon >0$ such that $B(\zeta, \epsilon)\subset (\cW')^c$, and Lemma \ref{lemma:LargeDevBalls}(i) tells us that $\mathbf{E}_0[\sum_{v\in \cN_t}\indc_{\{X_v(t)\in tB(\zeta,\epsilon)\}}] \to 0 $ as $ t\to +\infty$, a clear contradiction to the display above. Hence, $\zeta$ must belong to $\cW'$, and since $\zeta$ was chosen arbitrarily on $\cW$, we conclude that $\cW\subset \cW'$.

    To prove $ \cW'\subset \cW$ we will show that $ \mathrm{int}(\cW')\subset \cW$. Let us fix $\zeta \in \mathrm{int}(\cW')$ and $\epsilon>0$ such that $B(\zeta, \epsilon)\subset \cW'$. By Lemma \ref{lemma:LargeDevBalls}(ii), we know that, as $t\to \infty$,
    \begin{equation}\label{eq:auxLargeDevBall}
        \mathbf{E}_0\Bigg[\sum_{v\in \cN_t}\indc_{\{X_v(t)\in tB(\zeta,\epsilon)\}}\Bigg]\to +\infty.
    \end{equation}
    On the other hand, observe that if $\zeta\not \in\cW$, we can further make $\epsilon>0$ smaller in such a way that $\mathrm{dist}(B(\zeta,\epsilon), \cW)=\delta >0$. Thus, by the convexity of $\cW(\css)$, there exists a supporting hyperplane $\cH$ of $\cW(\css)$ such that 
    \begin{equation*}
        \cW \subset \bar{v}+\cH^- \ \ \mathrm{ and } \ \ B(\zeta,\epsilon) \subset \bar{v}+(1+\delta)\cH^+.  
    \end{equation*}
    By Proposition \ref{prop:ShapeTheorem:ContinuityCSS}(ii), we know that there exists $e\in S^{d-1}$ such that $\cH=\cH_{e, \css(e)}$. Hence,
    \begin{equation*}
        \begin{split}
            \mathbf{E}_0\Bigg[\sum_{v\in \cN_t}\indc_{\{X_v(t)\in tB(\zeta,\epsilon)\}}\Bigg]& \leq \mathbf{E}_0\Bigg[\sum_{v\in \cN_t}\indc_{\{(X_v(t) \in t\bar{v}+t(1+\delta)\cH^{+}_{e,\css(e)} \}}\Bigg]= \mathbf{E}_0\Bigg[\sum_{v\in \cN_t}\indc_{\{(X_v(t)-t\bar{v})\cdot e \geq t(1+\delta)\css(e) \}}\Bigg],
        \end{split}
    \end{equation*}
    but the right-hand side of the expression above tends to $0$ as $t\to \infty$  by \eqref{eq:Intro:UpperEstimate} which is in a clear violation of \eqref{eq:auxLargeDevBall}. It follows that $\zeta \in \cW$, and since $\zeta$ was chosen arbitrarily in $\mathrm{int}(\cW')$, we conclude that $\cW\supset \mathrm{int}(\cW')$. Finally, since $\cW$ is closed, $\cW' = \mathrm{cl}(\mathrm{int}(\cW'))\subset \cW$, and the claim follows, finishing the proof. 
\end{proof}

\begin{corollary}
    The shape $\cW$ is strictly convex.
\end{corollary}
\begin{proof}
    By Proposition \ref{prop:shape=LargeDev}, the boundary of $\cW$ is given by the set $\{y\in \R^d:\Phi(y)=0\}$. Let $x,y\in \partial \cW$ distinct, and $\lambda\in (0,1)$. Since $\Phi$ is strictly convex, 
     \begin{equation*}
         \Phi(\lambda x + (1-\lambda)y) < \lambda\Phi(x) + (1-\lambda)\Phi(y) = 0.
     \end{equation*}
     Thus, $\lambda x + (1-\lambda)y \notin \partial \cW$. Since we took $x$, $y$ and $\lambda$ to be arbitrary in $\partial \cW$ and $(0,1)$ respectively, we conclude that $\partial \cW$ does not contain straight line segments, and thus, $\cW$ is strictly convex.
\end{proof}

\begin{proposition}\label{prop:RegularityW}
    The boundary $\partial \cW$ is of class $C^2$, meaning that locally, up to a change of coordinates,  $\partial \cW$ is the graph of a $C^2$ function on some open subset of $\R^{d-1}$. In addition, considering $\partial \cW$ as a submanifold, let $K_{\cW}:\partial\cW\to \R$ be the Gaussian curvature of $\partial\cW$. Then, $\inf_{x\in \partial\cW}K_{\cW}(x)>0$.
\end{proposition}
\begin{proof}
    The first part is a consequence of the implicit function theorem for level sets of $C^k$ functions. Recall that $\partial \cW = \{x\in \R^d: \Phi(x)=0\}$. First, we show that every $x\in \partial \cW$ is regular, meaning that $\nabla \Phi(x)\neq 0$. Since $\nabla \Phi = (\nabla\gamma)^{-1}$ by Lemma \ref{lemma:LegendreTransformGamma}(ii), it is enough to show that $\nabla \gamma(0)$ does not belong to $\partial \cW$. By the properties of $\gamma$, the $\inf$ in \eqref{eq:PropertiesShape:Legendre} is achieved by a point $\zeta'\in \R^d$, and thus, by the strict convexity of $\gamma$ and Proposition \ref{Prop:PropertiesGamma}(ii), 
    \begin{equation*}
        \Phi(\nabla \gamma(0)) = \zeta'\cdot \nabla\gamma(0)-\gamma(\zeta') \leq -\gamma(0) < 0.
    \end{equation*}
    Since $\partial \cW = \{y\in \R^d: \Phi(y)=0\}$, we conclude that $\nabla \gamma(0)$ does not belong to $\partial \cW$. Thus, every $x=(x_1,...,x_d)\in \partial \cW$ is regular. By the implicit function theorem, we can write $x_d = \phi(x_1,...,x_{d-1})$ locally for a $C^2$ function $\phi$ on some open subset of $\R^{d-1}$, and the result follows.

    The fact that the Gaussian curvature is positive follows from the following formula for the Gaussian curvature of level sets (see, for example, \cite[Page 647]{MR2169053}). Letting $\mathrm{adj}(\cdot)$ denote the adjugate matrix, for $x\in \partial \cW$, the Gaussian curvature at $x$ satisfies
    \begin{equation}
        K_{\cW}(x) = \frac{(\nabla \Phi(x))^T \mathrm{adj}(\nabla^2 \Phi(x)) \nabla \Phi(x) }{|\nabla \Phi(x)|^{d+1}}.
    \end{equation}
    From Lemma \ref{lemma:LegendreTransformGamma}(iii), we know that $\nabla^2 \Phi$ is invertible, hence, $\mathrm{adj}(\nabla^2 \Phi(x))= \mathrm{det}(\nabla^2 \Phi(x))\nabla^2 \Phi(x)^{-1}$. From the same lemma, $\nabla^2 \Phi(x)^{-1} = \nabla^2 \gamma(\zeta)$, where $\zeta\in \R^d$ is such that $\nabla \gamma(\zeta)=x$. We know from Proposition \ref{Prop:PropertiesGamma} that $\gamma$ is strictly convex and belongs to $C^2(\R^d)$, and from Lemma \ref{lemma:LegendreTransformGamma}(iii) we know that $\nabla^2 \gamma$ is invertible. Using this, by regularity of $\gamma$ and the continuity of $\nabla^2 \gamma$, $\nabla^2 \gamma$ is uniformly elliptic on compact sets of $\R^d$, meaning that, for every compact subset $K\subset \R^d$, there exists $\theta(K)>0$ such that $\theta(K)\mathrm{Id} < \nabla^2 \gamma(\zeta)$ for every $\zeta \in K$. Hence, since the gradient of $\Phi$ and $\mathrm{det}(\nabla^2 \Phi)$ do not vanish on $\partial \cW$,
    \begin{equation*}
       \begin{split}
            \inf_{x\in \partial \cW}K_{\cW}(x) &= \inf_{x\in \partial \cW}\frac{\nabla^T \Phi(x) \mathrm{adj}(\nabla^2 \Phi(x)) \nabla \Phi(x) }{|\nabla \Phi(x)|^{d+1}} = \inf_{x\in \partial \cW} \mathrm{det}(\nabla^2 \Phi(x)) \frac{\nabla^T \Phi(x) \nabla^2 \gamma(\zeta) \nabla \Phi(x) }{|\nabla \Phi(x)|^{d+1}}\\
            & \geq \theta(\partial \cW)\inf_{x\in \partial \cW} \mathrm{det}(\nabla^2 \Phi(x)) \frac{|\nabla \Phi(x)|^2 }{|\nabla \Phi(x)|^{d+1}}>0,
       \end{split}
    \end{equation*}
    which finishes the proof.
\end{proof}

\begin{remark}
    A direct generalization of the argument above shows that, if $\Phi \in C^k(\R^d)$ for $k\geq 2$, $\partial \cW$ is a $C^k$ manifold. By Remark \ref{remark:RegularityGamma_Legendre}, for $k\geq 2$, $\Phi \in C^k(\R^d)$ if $\gamma \in C^k(\R^d)$, so $\partial \cW$ inherits the regularity of $\gamma$.
\end{remark}

\begin{proposition}\label{prop:PropertiesShape:regularitySupport}
    Consider the $1$-homogeneous extension $\cs:\R^{d}\setminus\{0\}\to \R$ of $\cs$, then $\cs$ belongs to $C^2(\R^d\setminus\{0\})$.
\end{proposition}
\begin{proof}
    This result follows from known properties of the support function. By Remark \ref{remark:SupportFunction}, we know that the $1$-homogeneous extension of $\cs$ is the support function of the convex body $\cW$. It is known (see \cite[Section 2.5]{MR3155183}) that if $\cs$ is the support function of $\cW$, a convex body with $\partial\cW$ a $C^2$ manifold of dimension $d-1$ with positive Gaussian curvature, then $\cs$ is of class $C^2$ on $\R^d\setminus \{0\}$. By Proposition \ref{prop:RegularityW}, we conclude $\cs \in C^2(\R^d\setminus \{0\})$.
\end{proof}

\appendix

\section{Higher order regularity of the principal eigenvalue}\label{append:PrincipalEigenvalue}

We sketch the proof of the following result.

\begin{proposition}\label{remark:Cinfinity:Gamma}
    The principal eigenvalue $\gamma:\R^d \to \R$ defined in \eqref{eq:Intro:EigenvalueProblem} belongs to $C^{\infty}(\R^d)$.
\end{proposition}

\begin{proof}
     The argument in \cite[Theorem 8.2.10]{MR1326606} can be employed inductively to show that $\gamma \in C^{\infty}(\R^d)$. We now sketch the argument given in \cite[Theorem 8.2.10]{MR1326606} to prove $\gamma \in C^{2}(\R^d)$, and we describe the inductive step to show that in fact $\gamma \in C^\infty(\R^d)$. 

     In \cite[Theorem 8.2.10(ii)]{MR1326606}, it is shown that 
     \begin{equation}\label{Append:eq:DerivativeGamma}
         \frac{\partial }{\partial e_i} \gamma(\zeta) = \int_{\TT^d} P_i^1(\psi(x;\zeta),\nabla_x \psi(x;\zeta), \Tilde{\psi}(x;\zeta), \zeta,x) dx,
     \end{equation}
     where for fixed $x\in \R^d$, $P_i^1$ is a multivariate polynomial on its remaining variables. The fact that $P_i^1$ has this polynomial form transforms the problem into justifying a finite difference approximation of the partial derivatives of $\psi$ with respect to $\zeta$. More precisely, we observe from \eqref{Append:eq:DerivativeGamma} that $\gamma \in C^2(\R^d)$ provided one can justify the following limits in the $C^1$-norm with respect to $x$,
     \begin{equation}\label{appen:eq:FiniteDifferencePrincipal}
         \lim_{\epsilon \to 0} \frac{\psi(\cdot;\zeta+\epsilon e_j)-\psi(\cdot;\zeta)}{\epsilon} = \partial_{e_j}\psi(\cdot;\zeta) \ \ \mathrm{and } \ \ \lim_{\epsilon \to 0} \frac{\Tilde{\psi}(\cdot;\zeta+\epsilon e_j)-\Tilde{\psi}(\cdot;\zeta)}{\epsilon} = \partial_{e_j}\Tilde{\psi}(\cdot;\zeta).
     \end{equation}
     Let us sketch how the first limit above is justified, the second limit follows analogously. By a finite difference approximation of the equation of the principal eigenvalue \eqref{eq:Intro:EigenvalueProblem} with respect to the $j$-th entry of $\zeta$, we observe that $\partial_{e_j}\psi(\cdot;\zeta)$ must satisfy a second-order elliptic equation having the form
     \begin{equation}
         (\mathcal{L}^\zeta + K^\zeta - \gamma(\zeta))\partial_{e_j}\psi(x;\zeta) = Q_j^1(\psi(x;\zeta),\nabla_x \psi(x;\zeta), \gamma(\zeta),\nabla \gamma(\zeta), \zeta,x),
     \end{equation}
    where for fixed $x\in \R^d$, $Q_i^1$ is a multivariate polynomial on its remaining variables. Notice that $Q_j^1$ only depends on $\psi$, $\nabla_x \psi$, $\gamma$, $\nabla \gamma$, which are all well-defined up to this point. The limit in \eqref{appen:eq:FiniteDifferencePrincipal} is then justified by a classical compactness argument, employing global Schauder estimates of $\psi$, and the well-posedness theory of the elliptic equation above. This is the argument given in \cite[Theorem 8.2.10]{MR1326606} to prove that $\gamma \in C^2(\R^d)$.
     
    With the limits \eqref{appen:eq:FiniteDifferencePrincipal} justified, we can then differentiate the right-hand side of \eqref{Append:eq:DerivativeGamma} with respect to the $j$-th entry of $\zeta$ to obtain
    \begin{equation*}
        \frac{\partial^2}{\partial e_i \partial e_j} \gamma(\zeta) = \int_{\TT^d}P_{ij}^2(\psi(x;\zeta),\nabla_x \psi(x;\zeta), \Tilde{\psi}(x;\zeta), \nabla_{\zeta}\psi(x;\zeta), \nabla_x \nabla_{\zeta}\psi(x;\zeta), \nabla_{\zeta}\Tilde{\psi}(x;\zeta), \zeta,x) dx,
     \end{equation*}
     where for fixed $x\in \R^d$, $P_{ij}^2$ is a multivariate polynomial on its remaining variables. We then repeat the argument sketched above inductively to obtain $\gamma \in C^3(\R^d)$, and in general, to prove that $\gamma \in C^k(\R^d)$ for any $k\in \N$.     
\end{proof}

\begin{proof}[Proof of Proposition \ref{prop:Eigenvalue:Analytic}]
    For $\zeta_1, \zeta_2\in \R^d$ with $\zeta_2\neq 0$, let us introduce $\gamma^{\zeta_1,\zeta_2}:\R\to \R$ defined, for every $\lambda\in \R$, by $\gamma^{\zeta_1,\zeta_2}(\lambda)= \gamma(\zeta_1+\lambda \zeta_2)$. The proof employs \cite[Theorem 1]{MR279263}, which states that if $\gamma \in C^{\infty}(\R^d)$, and, if for every $\zeta_1, \zeta_2\in \R^d$ the function $\gamma^{\zeta_1,\zeta_2}$ is real-analytic near zero, then $\gamma$ is analytic. 
    
    We already know that $\gamma \in C^\infty(\R^d)$ by Proposition \ref{remark:Cinfinity:Gamma}, so we focus on proving that for every $\zeta_1, \zeta_2\in \R^d$, the function $\gamma^{\zeta_1,\zeta_2}$ is locally real-analytic. To do so, we employ \cite[Remark VII.2.9]{MR1335452}, which provides a criterion to prove that the principal eigenvalue of a family of operators parametrized by $z\in \mathbb{C}$ is holomorphic, we proceed to introduce the criterion properly adjusted to our setting.
    
    For $k\in \N\cup\{0\}$, let us denote the complexification of $C^{k,\alpha}_p(\R^d)$ by $C^{k,\alpha}_p(\R^d;\mathbb{C})$, endowed with the Bochner norm $\|\varphi_1 + i\varphi_2\|_{C^{k,\alpha}_p(\R^d;\mathbb{C})} = \sup_{\theta \in [0,2\pi]} \|\varphi_1\cos\theta - \varphi_2\sin\theta\|_{k,\alpha}$. Let $D\subset \mathbb{C}$ be a domain of the complex plane with $0\in D$, and consider a family of linear operators $(R(\Tilde{\lambda}))_{\Tilde{\lambda}\in D}$ such that $R(\Tilde{\lambda}):C^{2,\alpha}(\R^d;\mathbb{C})\to C^{0,\alpha}(\R^d;\mathbb{C})$. We say that the family $(R(\Tilde{\lambda}))_{\Tilde{\lambda}\in D}$ is holomorphic of type A if for every $\varphi\in C^{2,\alpha}(\R^d;\mathbb{C})$, we can find unique linear operators $R^{(1)},R^{(2)},R^{(3)},...:C^{2,\alpha}(\R^d;\mathbb{C})\to C^{0,\alpha}(\R^d;\mathbb{C})$ such that, for every $\varphi\in C^{2,\alpha}(\R^d;\mathbb{C})$, the series
    \begin{equation}
        R(\Tilde{\lambda}) \varphi = R(0)\varphi + z R^{(1)}\varphi + z^2 R^{(2)}\varphi + ...,
    \end{equation}
    converges in a disk $|z|<r$ with radius $r$ independent of $\varphi$. We will apply the criterion in \cite[Remark VII.2.9]{MR1335452} to the principal eigenvalue of a family of operators $(R(\Tilde{\lambda}))_{\Tilde{\lambda}\in D}$, applied to our setting, this criterion reads as follows: Suppose that the family $(R(\Tilde{\lambda}))_{\Tilde{\lambda}\in D}$ is holomorphic of type A, and that the principal eigenvalue of $R(0)$ is isolated, then the principal eigenvalue is holomorphic near zero in the complex plane. 
    
    We do not apply directly this criterion to $\gamma$ because the family of eigenvalue problems \eqref{eq:Intro:EigenvalueProblem} is parametrized by $\zeta\in \R^d$ instead of by $z\in \mathbb{C}$. To overcome this, we fix $\zeta_1,\zeta_2\in \R^d$ with $\zeta_2\neq 0$. For $\lambda\in \R$, we consider the operator $T^{\zeta_1,\zeta_2}(\lambda):C^{2,\alpha}_p(\R^d)\to C^{0,\alpha}_p(\R^d)$, defined for $\varphi \in C^{2,\alpha}_p(\R^d)$ by
    \begin{equation*}
        T^{\zeta_1,\zeta_2}(\lambda) \varphi = \mathcal{L}^{\zeta_1+\lambda\zeta_2}\varphi + K^{\zeta_1+\lambda\zeta_2}\varphi,
    \end{equation*}
    where $\mathcal{L}^\zeta$ and $K^\zeta$ are the operators defined in \eqref{eq:Intro:GeneratorExponential}. Clearly, for every $\lambda\in \R$, $\gamma^{\zeta_1,\zeta_2}(\lambda)$ is the principal eigenvalue of the operator $T^{\zeta_1,\zeta_2}(\lambda)$. We extend the family of operators $(T^{\zeta_1,\zeta_2}(\lambda))_{\lambda \in \R}$ parametrized by $\R$, to a family $(T^{\zeta_1,\zeta_2}(\Tilde{\lambda}))_{{\Tilde{\lambda}} \in \mathbb{C}}$ by setting
    \begin{equation*}
        T^{\zeta_1,\zeta_2}(\Tilde{\lambda})\varphi = T^{\zeta_1,\zeta_2}(\mathrm{Re}(\Tilde{\lambda}))\varphi + i T^{\zeta_1,\zeta_2}(\mathrm{Im}(\Tilde{\lambda}))\varphi.
    \end{equation*}
    This is the family of operators to which we will apply the criterion. 
    
    Let us show how to conclude assuming that the hypotheses of the criterion \cite[Remark VII.2.9]{MR1335452} are satisfied. After applying the criterion, the function $\gamma^{\zeta_1,\zeta_2}:\mathbb{C}\to \mathbb{C}$ which sends $\Tilde{\lambda}\in \mathbb{C}$ to the principal eigenvalue of $T^{\zeta_1,\zeta_2}(\Tilde{\lambda})$ is analytic in a neighborhood of $0\in \mathbb{C}$. Hence, we can find $a_0^{\zeta_1,\zeta_2}, a_1^{\zeta_1,\zeta_2}, a_2^{\zeta_1,\zeta_2},... \in \mathbb{C}$ such that, in an open ball around $0\in \mathbb{C}$,
    \begin{equation*}
        \gamma^{\zeta_1,\zeta_2}(\Tilde{\lambda}) = \sum_{n=0}^\infty a_n^{\zeta_1,\zeta_2} \Tilde{\lambda}^n.
    \end{equation*}
    Recall that for $\lambda\in \R$, $\gamma^{\zeta_1,\zeta_2}(\lambda)\in \R$, from which we conclude that the coefficients $a_0^{\zeta_1,\zeta_2}, a_1^{\zeta_1,\zeta_2},...$ belong to $\R$. We deduce that the restriction of $\gamma^{\zeta_1,\zeta_2}$ to $\R$ is real-analytic near $0 \in \R$. Hence, up to verifying the hypotheses of the criterion stated at the start of the proof, $\gamma:\R^d \to \R$ is analytic.

    Let us verify that the hypotheses of \cite[Remark VII.2.9]{MR1335452} are satisfied. First, we show that the family $(T^{\zeta_1,\zeta_2}(\Tilde{\lambda}))_{{\Tilde{\lambda}} \in \mathbb{C}}$ is holomorphic of type A. To do so we apply \cite[Theorem VII.2.2.6]{MR1335452}, for which we group the terms of \eqref{eq:Intro:GeneratorExponential}, to express $T^{\zeta_1,\zeta_2}$ as a series on $\Tilde{\lambda}$ by letting
    \begin{equation*}
        T^{\zeta_1,\zeta_2}(\Tilde{\lambda}) \varphi = T_0^{\zeta_1,\zeta_2} \varphi + \Tilde{\lambda} T_1^{\zeta_1,\zeta_2} \varphi + \Tilde{\lambda}^2 T_2^{\zeta_1,\zeta_2} \varphi,
    \end{equation*}
    where
    \begin{equation*}
        \begin{split}
            T_0^{\zeta_1,\zeta_2} \varphi &= \frac{1}{2} \sum_{i=1}^d \sum_{j=1}^d a_{ij} \partial_{ij} \varphi + \zeta^T_1 A \nabla \varphi + b\cdot \nabla \varphi + \frac{1}{2}\zeta^T_1 A \zeta_1 \varphi + \zeta_1\cdot b\varphi + R\varphi,\\
            T_1^{\zeta_1,\zeta_2} \varphi &= \zeta^T_2 A \nabla \varphi + \zeta^T_2 A \zeta_1 \varphi + \zeta_2\cdot b\varphi, \ \ \mathrm{and} \ \ T_2^{\zeta_1,\zeta_2} \varphi = \frac{1}{2}\zeta^T_2 A \zeta_2 \varphi.
        \end{split}
    \end{equation*}
    Since the series above is finite, it is sufficient to prove that $T_0^{\zeta_1,\zeta_2}, T_1^{\zeta_1,\zeta_2}$ and $T_2^{\zeta_1,\zeta_2}$ are bounded operators from $C^{2,\alpha}(\R^d;\mathbb{C})$ to $C^{0,\alpha}(\R^d;\mathbb{C})$.
    
    Clearly, there exists $C_1=C_1(\zeta_1,\zeta_2)>0$ such that $\|T_2^{\zeta_1,\zeta_2} \varphi \|_{0,\alpha} \leq C_1\|\varphi\|_{2,\alpha}$ and such that $\|\zeta^T_2 A \zeta_1 \varphi + \zeta_2\cdot b\varphi \|_{0,\alpha} \leq C_1\|\varphi\|_{2,\alpha}$. On the other hand, 
    \begin{equation*}
        \| \zeta^T_2 A \nabla \varphi \|_{0,\alpha} \leq C_2 \|\nabla \varphi\|_\infty + C_3 \sup_{y\neq x, i\in [d]}\frac{|\partial_i \varphi(x)-\partial_i \varphi(y)|}{|x-y|^\alpha} \leq C_2 \|\nabla \varphi\|_\infty + C_4 \|\nabla^2 \varphi\|_\infty\leq C_5  \|  \varphi\|_{2,\alpha}.
    \end{equation*}
    Which implies that we can find $C_1=C_1(\zeta_1,\zeta_2)$ such that
    \begin{equation*}
        \|T_0^{\zeta_1,\zeta_2} \varphi \|_{0,\alpha} \leq C_1\|\varphi\|_{2,\alpha} \ \mathrm{and} \ \|T_1^{\zeta_1,\zeta_2} \varphi \|_{0,\alpha} \leq C_1\|\varphi\|_{2,\alpha}.
    \end{equation*}
    These bounds are sufficient to argue that $T_0^{\zeta_1,\zeta_2}, T_1^{\zeta_1,\zeta_2}$ and $T_2^{\zeta_1,\zeta_2}$ are bounded operators from $C^{2,\alpha}(\R^d;\mathbb{C})$ to $C^{0,\alpha}(\R^d;\mathbb{C})$. By \cite[Theorem VII.2.2.6]{MR1335452}, it follows that $(T^{\zeta_1,\zeta_2}(\Tilde{\lambda}))_{\Tilde{\lambda}\in \mathbb{C}}$ is holomorphic of type A.

    Finally, we verify that the principal eigenvalue $\gamma(\zeta_1)$ associated to $T^{\zeta_1,\zeta_2}(0)=\mathcal{L}^{\zeta_1} + K^{\zeta_1}$ is isolated. This follows by a classical argument. We express the spectrum of $T^{\zeta_1,\zeta_2}(0)$ in terms of the resolvent of an associated linear operator. Let $\theta_1 = \sup_{x\in \R^d}K^{\zeta_1}(x)$. The operator $(T^{\zeta_1,\zeta_2}(0) - \theta_1)^{-1}:C^{0,\alpha}_p(\R^d) \to C^{2,\alpha}_p(\R^d)$ is well defined and is compact. Applying the argument in \cite[Proposition 3.5.4]{MR1326606}, it follows that $\Tilde{\lambda}$ belongs to the spectrum of $T^{\zeta_1,\zeta_2}(0)$ if and only if $-1/(\theta_1-\Tilde{\lambda})$ belongs to the spectrum of $(T^{\zeta_1,\zeta_2}(0) - \theta_1)^{-1}$. Finally, since $(T^{\zeta_1,\zeta_2}(0) - \theta_1)^{-1}$ is compact, we can apply the Riesz-Schauder theorem (see \cite[Theorem 3.1.1]{MR1326606}), which states that the spectrum of $(T^{\zeta_1,\zeta_2}(0) - \theta_1)^{-1}$ is a compact subset of $\mathbb{C}$ containing $0$ as it unique accumulation point. It follows that $\gamma(\zeta_1)>0$ is isolated, and \cite[Remark VII.2.9]{MR1335452} implies that $\Tilde{\lambda} \mapsto \gamma^{\zeta_1,\zeta_2}(\Tilde{\lambda})$ is complex-analytic. This verifies that the hypotheses of the criterion are satisfied, and finishes the proof.
\end{proof}

\footnotesize
\bibliographystyle{plain}
\bibliography{refs}

\end{document}